\documentclass{article}
\usepackage[left=4cm, right=4cm, top=4cm, bottom=2.25cm]{geometry}
\usepackage{amsmath,color}
\usepackage{amsfonts,amsthm}
\usepackage{amssymb,enumerate,enumitem,verbatim}
\usepackage[pdftex]{graphicx}

\usepackage{setspace,scalefnt}
\usepackage[dvipsnames,svgnames,table]{xcolor}
\usepackage{tikz,caption,subcaption}
\usetikzlibrary{math,calc,intersections, scopes}
\usetikzlibrary{positioning,arrows,shapes,decorations.markings,decorations.pathreplacing, decorations.pathmorphing,matrix,patterns}
\tikzstyle{vertex}=[circle,draw=black,fill=black,inner sep=0,minimum size=3pt,text=white,font=\footnotesize]

\newtheorem{lemma}{Lemma}[section]
\newtheorem{corollary}[lemma]{Corollary}
\newtheorem{theorem}[lemma]{Theorem}

\newtheorem{conjecture}[lemma]{Conjecture}

\theoremstyle{definition}
\newtheorem{defn}[lemma]{Definition}
\newtheorem{claim}[lemma]{Claim}

\global\long\def\eps{\varepsilon}

\usepackage[colorlinks=true,allcolors=black]{hyperref}

\newcommand{\claimproofstart}[1][Proof]{\begin{proof}[#1]
\renewcommand{\qedsymbol}{$\boxdot$}}
\newcommand\claimproofend{
\end{proof}
\renewcommand{\qedsymbol}{$\square$}}

\definecolor{darkgreen}{rgb}{0.1,0.7,0.1}

\newcommand\subsetsim{\mathrel{%
  \ooalign{\raise0.2ex\hbox{$\subset$}\cr\hidewidth\raise-0.8ex\hbox{\scalebox{0.9}{$\sim$}}\hidewidth\cr}}}
\newcommand\midusedassuchthat{:}
\newcommand{\numberofabsorbable}{m}

  \makeatletter
  \newcommand{\labelinthm}[1]{%
     \label{temp#1}
     \protected@write \@auxout {}{\string \newlabel{#1}{{\emph{\ref{temp#1}}}{\thepage}{\emph{\ref{temp#1}}}{temp#1}{}} }%
  }
  \makeatother

\newcommand\su{\subseteq}
\newcommand\rk{\operatorname{rk}}
\newcommand\spn{\operatorname{span}}

\title{Asymptotically-tight packing and covering with transversal bases in Rota's basis conjecture}
\author{Richard Montgomery\thanks{Mathematics Institute, University of Warwick, Coventry, United Kingdom. Email: {\tt
richard.montgomery@warwick.ac.uk}. Supported by the European Research Council (ERC) under the European Union Horizon 2020 research and innovation programme (grant agreement No.\ 947978).}
\and Lisa Sauermann\thanks{Institute for Applied Mathematics, University of Bonn, Germany. Email: {\tt sauermann@iam.uni-bonn.de}. Supported by the
DFG Heisenberg Program.}}

\begin{document}

\maketitle

\begin{abstract}
In 1989, Rota conjectured that, given any $n$ bases $B_1,\dots,B_n$ of a vector space of dimension $n$, or more generally a matroid of rank $n$, it is possible to rearrange these into $n$ disjoint transversal bases. Here, a transversal basis is a basis consisting of exactly one element from each of the original bases $B_1,\dots,B_n$. Two natural approaches to this conjecture are, to ask in this setting a) how many disjoint transversal bases can we find and b) how few transversal bases do we need to cover all the elements of $B_1,\dots,B_n$? In this paper, we give asymptotically-tight answers to both of these questions.

For a), we show that there are always $(1-o(1))n$ disjoint transversal bases, improving a result of Buci{\'c}, Kwan, Pokrovskiy, and Sudakov that $(1/2-o(1))n$ disjoint transversal bases always exist. 
For b), we show that $B_1\cup\dots \cup B_n$ can be covered by $(1+o(1))n$ transversal bases, improving a result of Aharoni and Berger using instead $2n$ transversal bases, and a subsequent result of the Polymath project on Rota's basis conjecture using $2n-2$ transversal bases.
\end{abstract}

\section{Introduction}

Rota's basis conjecture is a famous conjecture concerning the rearrangeability of bases in a vector space or matroid. Rota  posed it in 1989 (see~\cite[Conjecture~4]{huang1994relations}), both in the special case of vector spaces, and for general matroids (see Section~\ref{sec:prelim} for the definition and the basic properties of a matroid). Given any $n$ bases $B_1,\dots,B_n$ of an $n$-dimensional vector space (or, more generally, in a matroid of rank $n$), the conjecture states that the elements in each of the bases $B_1,\dots,B_n$ can be ordered in such a way that the first elements of $B_1,\dots,B_n$ together form a basis, the second elements of $B_1,\dots,B_n$ together form a basis, and so on. In other words, the (multi-)set, $B_1\cup \dots\cup B_n$ can be decomposed into $n$ disjoint bases each consisting of exactly one element from each of the original bases $B_1,\dots,B_n$. Saying that such a basis containing exactly one element from each of $B_1,\dots,B_n$ is a \emph{transversal basis}, Rota's basis conjecture can be restated as follows.

\begin{conjecture}[Rota's Basis Conjecture]\label{conj:matroidrota}
Any collection of $n$ disjoint bases in a rank-$n$ matroid can be decomposed into $n$ transversal bases.
\end{conjecture}

Aside from its intriguingly simple statement, an appealing aspect of this conjecture is that it has many unexpected connections. In particular, Huang and Rota~\cite{huang1994relations} discovered a surprising link to the Alon--Tarsi conjecture~\cite{alon1992colorings} on the number of odd and even Latin squares of order $n$ (as discussed further below). Additionally, Conjecture~\ref{conj:matroidrota} for graphic matroids can be phrased as a question about rainbow spanning trees in a multi-graph (namely, given a multi-graph on $n+1$ vertices with $n$ spanning trees in $n$ different colours, does there always exist a decomposition of the coloured edge set into rainbow spanning trees?). Here, a \emph{rainbow} subgraph is one with at most one edge of each colour. Studying the existence of rainbow subgraphs in graphs with various properties, and applications thereof, is more generally a very active research direction, as can be seen in the recent surveys (on related but different topics) by Pokrovskiy~\cite{Alexeysurvey}, Sudakov~\cite{sudakov2024restricted}, and the current first author~\cite{transversalsurvey}.

Despite a great deal of attention, including as the subject of the 12th Polymath project~\cite{chowpolymathblog} in 2017, Conjecture~\ref{conj:matroidrota} remains widely open, even for the case of vector spaces (i.e., representable matroids). The conjecture has only been confirmed in some very special cases, including for `paving matroids' by Geelen and Humphries~\cite{geelen2006rota}, for `strongly base orderable matroids' by Wild~\cite{wild1994rota}, and for matroids of rank at most $4$ by Cheung~\cite{cheung-rank-4}. The link to the Alon--Tarsi conjecture, discovered by Huang and Rota~\cite{huang1994relations} (and later simplified by Onn~\cite{onn1997colorful}), reduces Rota's basis conjecture for real-representable matroids of rank $n$, for even $n$, to the Alon--Tarsi conjecture for Latin squares of order $n$. The Alon--Tarsi conjecture was proved when $n-1$ is prime by Drisko~\cite{drisko1997number} and when $n+1$ is prime by Glynn~\cite{glynn2010conjectures}, thus together giving a proof of Conjecture~\ref{conj:matroidrota} for real-representable matroids of rank $n$ when $n-1$ or $n+1$ is an odd prime.

In addition to these special cases, there is some further evidence towards Conjecture~\ref{conj:matroidrota}. In 2006, Aharoni and Berger~\cite{aharoni2006intersection} proved that the natural fractional relaxation of Conjecture~\ref{conj:matroidrota} holds. More recently, the current second author~\cite{sauermann2024rota} showed that the conjecture holds with high probability for representable matroids over finite fields if the $n$ bases are chosen independently and uniformly at random.

There are two natural directions for working towards Conjecture~\ref{conj:matroidrota} in full, as was highlighted by Pokrovskiy~\cite{pokrovskiy2020rota}, and which are the focus of this paper. 
The first of these is the \emph{packing problem} in this setting: given $n$ bases in a rank-$n$ matroid, how many disjoint transversal bases can we find? The second is the \emph{covering problem} in this setting: given $n$ bases in a rank-$n$ matroid, how many transversal bases do we need to cover all of the elements of these bases? As discussed in Section~\ref{sec:results} below, in this paper we give asymptotically-tight answers to both of these questions, thus solving Problems~4.2 and~4.3 in~\cite{pokrovskiy2020rota}.

That one transversal basis can be found in the setting of Conjecture~\ref{conj:matroidrota} is an easy consequence of the matroid augmentation property, which can be used to greedily select one element from each basis step by step, while maintaining that the selected elements form an independent set. However, already, the existence of two disjoint transversal bases is not immediate. In 2007, Geelen and Webb~\cite{geelen2007rota} gave the first packing result towards Conjecture~\ref{conj:matroidrota}, by showing that there are always $\Omega(\sqrt{n})$ disjoint transversal bases in this setting. Geelen and Dong~\cite{dong2019improved} later introduced a beautiful probabilistic argument to show that there are always $\Omega(n/\log n)$ disjoint transversal bases. Finally, in 2020, 
Buci{\'c}, Kwan, Pokrovskiy, and Sudakov~\cite{bucic2020halfway} made an important breakthrough by showing there is always a linear number of disjoint transversal bases, as follows.

\begin{theorem}[\cite{bucic2020halfway}]\label{thm:halfway}
For any $\eps>0$, the following holds whenever $n$ is sufficiently large with respect to $\eps$. Any collection of $n$ bases of a rank-$n$ matroid has at least $(1/2-\eps)n$ disjoint transversal bases.
\end{theorem}

Theorem~\ref{thm:halfway} shows that, given bases $B_1,\dots,B_n$ in a rank-$n$ matroid, there is a family of disjoint transversal bases using almost one half of the elements in $B_1\cup\dots\cup B_n$. Pokrovskiy~\cite{pokrovskiy2020rota} subsequently showed that if instead of disjoint transversal bases one forms a family of $n$ disjoint rainbow independent sets, then almost all of the elements can be covered, as follows. Here, a \emph{rainbow independent set} is an independent set containing at most one element from each of the given bases.

\begin{theorem}[\cite{pokrovskiy2020rota}]\label{thm:pokrovskiy}
For any $\eps>0$, the following holds whenever $n$ is sufficiently large with respect to $\eps$. Any collection of $n$ bases of a rank-$n$ matroid has at least $(1-\eps)n$ disjoint rainbow independent sets of size at least $(1-\eps)n$ each.
\end{theorem}

We remark that, with a simple greedy argument, one can deduce from Theorem~\ref{thm:pokrovskiy} that in the same setting one can also find $n$ disjoint rainbow independent sets, each of size at least $(1-\eps)n$ (this was observed by Kwan, see~\cite[Section~4]{pokrovskiy2020rota}). However, there does not seem to be a simple argument in order to increase the size of the rainbow independent sets in Theorem~\ref{thm:pokrovskiy}.

In the second direction, concerning covering results towards Conjecture~\ref{conj:matroidrota}, neither Theorem~\ref{thm:halfway} nor Theorem~\ref{thm:pokrovskiy} imply a strong covering result. Indeed, even if we can find $(1-o(1))n$ disjoint transversal bases in some matroid, it may still be the case that an additional $n$ transversal bases are needed in order to cover the remaining $o(n^2)$ elements (it is easy to construct examples for this).

Using topological tools, Aharoni and Berger~\cite[Assertion 8.11]{aharoni2006intersection} showed in 2006 that $2n$ transversal bases are sufficient to cover all of the elements in the bases in the setting of Conjecture~\ref{conj:matroidrota}, as follows (in fact, they proved a more general result \cite[Theorem 8.9]{aharoni2006intersection}, see also the discussion just before Theorem~\ref{thm:aharoni-berger}).

\begin{theorem}[\cite{aharoni2006intersection}]\label{thm:covering-2n}
Any collection of $n$ bases of a rank-$n$ matroid can be covered by $2n$ transversal bases.
\end{theorem}

The Polymath 12 project~\cite{chowpolymathblog} later improved this bound slightly, showing that $2n-2$ transversal bases are already sufficient.

\subsection{Our results}\label{sec:results}

In this paper, we give asymptotically-tight bounds for both the packing and the covering problem in the setting of Rota's basis conjecture, as follows, thus improving Theorems~\ref{thm:halfway} to \ref{thm:covering-2n} above.

\begin{theorem}[Asymptotic Packing Theorem]\label{thm:main-packing}
For any $\eps>0$, the following holds whenever $n$ is sufficiently large with respect to $\eps$. Any collection of $n$ bases of a rank-$n$ matroid has at least $(1-\eps)n$ disjoint transversal bases.
\end{theorem}

\begin{theorem}[Asymptotic Covering Theorem]\label{thm:main-covering} 
For any $\eps>0$, the following holds whenever $n$ is sufficiently large with respect to $\eps$. Any collection of $n$ bases of a rank-$n$ matroid can be covered by at most $(1+\eps)n$ transversal bases.
\end{theorem}

To prove Theorem~\ref{thm:main-packing} and \ref{thm:main-covering}, we modify and extend the methods of \cite{bucic2020halfway} and \cite{pokrovskiy2020rota}, while introducing new randomised approaches to extend families of rainbow independent sets to transversal bases, careful notions of `density' within subsets of the matroid, and intricate techniques to modify and improve families of disjoint rainbow independent sets. Outlines of our proofs can be found in Section~\ref{sec:outline}.

As discussed by Buci\'c, Kwan, Pokrovskiy, and Sudakov \cite{bucic2020halfway}, and by Pokrovskiy~\cite{pokrovskiy2020rota}, a plausible path to a proof of Conjecture~\ref{conj:matroidrota} might be to combine an approximate form of the conjecture with the \emph{absorption method}. Conjecture~\ref{conj:matroidrota} seems to be a particularly complex setting for the use of absorption, but, if the absorption method can be applied, then the two asymptotic forms of Conjecture~\ref{conj:matroidrota} shown here will very likely be helpful.


\section{Preliminaries}\label{sec:prelim}

We begin by recalling the definition and some basic properties of matroids (for more details see, for example, the textbook \cite{oxley2011book}).
A matroid $\mathcal{M}$ is given by a finite set $M$ and a collection $\mathcal{I}$ of subsets of $M$ satisfying the following three conditions:
\begin{itemize}
    \item $\emptyset\in \mathcal{I}$.
    \item For any set $S\in \mathcal{I}$, and any subset $S'\su S$, we also have $S'\in \mathcal{I}$.
    \item For any sets $S,T\in \mathcal{I}$ with $|S|<|T|$, there exists an element $t\in T\setminus S$ with $S\cup \{t\}\in \mathcal{I}$.
\end{itemize}
The last condition above is often referred to as the \emph{augmentation property}.
The set $M$ is called the \emph{ground set} of the matroid and the members of $\mathcal{I}$ are called the \emph{independent sets} in the matroid $\mathcal{M}$. Common examples of matroids are given by considering linearly independent sets in vector spaces, or acyclic subsets of the edges of a graph.
Given a matroid $\mathcal{M}$ with ground set $M$, the \emph{rank} of any subset $S\su M$ is defined by
\[\rk(S):=\max\{|S'|\midusedassuchthat  S'\text{ independent},\,S'\su S\}.\]
In other words, the rank of $S$ is the size of the largest independent subset of $S$. Clearly, we have $\rk(S)\le |S|$ (with equality if and only if $S$ itself is independent), and furthermore $\rk(S)\le \rk(T)$ whenever $S\su T$.

The rank of the matroid is the rank of its ground set, i.e., the maximum size of an independent set in the matroid. A maximum size independent set is called a \emph{basis} of the matroid.

Given a matroid $\mathcal{M}$ with ground set $M$, the \emph{span} of a subset $S\su M$ is defined as
\[\spn(S)=\{x\in M\midusedassuchthat \rk(S\cup \{x\})=\rk(S)\}. \]
In other words, $\spn(S)$ is the set consisting of all elements whose addition to $S$ does not increase the rank. Note that for any basis $B$ of $\mathcal{M}$ we have $\spn(B)=M$. Some authors use the name \emph{closure} instead of span. Also note that this concept of span agrees with the linear-algebraic span in matroids obtained from vector spaces (by taking the vector space as the ground set and linearly independent sets as the independent sets of the matroid).

It is not hard to see that $S\su \spn(S)$ and $\rk(\spn(S))=\rk(S)$ for every subset $S\su M$ (see \cite[Lemma 1.4.2]{oxley2011book}). One can also show that $\spn(\spn(S))=\spn(S)$ and $\spn(S)\su \spn(T)$ for any subsets $S\su T\su M$ (see \cite[Lemma 1.4.3]{oxley2011book}).
Finally, for any subsets $S\su T\su M$ with $\rk(S)=\rk(T)$, we must have $\spn(S)= \spn(T)$. Indeed, if there was an element $x\in \spn(T)\setminus \spn(S)$, then we would have $\rk(T\cup \{x\})\ge \rk(S\cup \{x\})>\rk(S)=\rk(T)$, a contradiction.

In this paper, we are concerned with rainbow independent sets in coloured matroids. A \emph{coloured matroid} is simply a matroid in which every element is assigned a colour. A set is called a \emph{rainbow set} if all of its elements have distinct colours. Thus, a rainbow independent set is a set which is independent and consists of elements of distinct colours (this may be interpreted as a set in the intersection of the underlying matroid with the partition matroid given by the colouring, but this perspective is not important in our paper). Furthermore, a rainbow basis is a basis of the matroid whose elements have distinct colours.

Using the augmentation property above repeatedly, one can show that, for any independent sets $S$ and $T$ in a matroid, one can find an independent set $S^*$ with $S\su S^*\su S\cup T$ and $|S^*|\ge |T|$. This last condition can equivalently be stated as $|(S\cup T)\setminus S^*|\le |(S\cup T)\setminus T|$, or again equivalently as $|T\setminus S^*|\le |S\setminus T|$. The following lemma proves a similar statement for rainbow independent sets in coloured matroids.

\begin{lemma}\label{lem:rainbow-independence-simple-adding}
Let $S$ and $T$ be rainbow independent sets in a coloured matroid. Then there exists a rainbow independent set $S^*$ with $S\su S^*\su S\cup T$ and $|T\setminus S^*|\le 2\cdot |S\setminus T|$.
\end{lemma}

We give the simple proof of this lemma in Section~\ref{sec:simple-matroid-lemmas}, in which we collect the proofs of various statements about matroids used in this paper.

In the settings of Theorems~\ref{thm:main-packing} and \ref{thm:main-covering}, we are given a matroid $\mathcal{M}$ of rank $n$ and bases $B_1,\dots,B_n$. We may assume that the bases $B_1,\dots,B_n$ are disjoint, by ``duplicating'' elements appearing in more than one of these bases. More formally, if an element $x$ appears in the bases $B_{i(1)},\dots,B_{i(k)}$ (with $k\ge 2$), we define a matroid $\mathcal{M}'$ where we replace $x$ by $k$ different elements $x_1,\dots,x_k$ in the ground set, and we replace $x$ by $x_j$ in each of the bases $B_{i(j)}$ (for $j=1,\dots,k$). In this new matroid $\mathcal{M}'$ a set is independent if it contains at most one of the elements $x_1,\dots,x_k$ and it corresponds to an independent set in $\mathcal{M}$ when identifying  $x_1,\dots,x_k$ with $x$. Repeating this operation for all elements appearing in more than one of the bases $B_1,\dots,B_n$, we can ensure that the bases $B_1,\dots,B_n$ are disjoint.

Thus, in the settings of Theorems~\ref{thm:main-packing} and \ref{thm:main-covering}, we may consider a matroid of rank $n$ with disjoint bases $B_1,\dots,B_n$. We may furthermore restrict the ground set of the matroid to be $B_1\cup \dots \cup B_n$ (note that all elements outside $B_1\cup \dots \cup B_n$ are irrelevant to the conclusions of Theorems~\ref{thm:main-packing} and \ref{thm:main-covering}). Now, having a matroid with ground set $B_1\cup \dots \cup B_n$ and disjoint bases $B_1,\dots ,B_n$, we can naturally define a colouring: We simply say that an element $x\in B_1\cup \dots \cup B_n$ has colour $i$ if $x\in B_i$. We denote the colour of any element $x\in B_1\cup \dots \cup B_n$ by $c(x)$  (note that then we have $x\in B_{c(x)}$). A transversal basis is now simply a rainbow basis.

We close this section by introducing some notational conventions, and stating two more matroid theory lemmas which we will use throughout this paper.
For a family $\mathcal{T}$ of disjoint rainbow independent sets in a coloured matroid, we write $E(\mathcal{T})=\bigcup_{T\in \mathcal{T}} T$ for the set of matroid elements covered by the sets in $\mathcal{T}$.
For a subset $S$ of a matroid and an element $x\in S$, we write $S-x$ for the set $S\setminus \{x\}$ obtained when removing $x$ from $S$ (we will not use the notation $S-x$ for $S\setminus \{x\}=S$ if $x\not\in S$). Similarly, for an element $y$ of the matroid, we write $S+y$ for the set obtained from $S$ when adding $y$. In the case, where $y$ is already contained in $S$, we interpret $S+y$ as a multi-set. Note that, in this case, the set $S+y$ cannot be independent (since it contains the element $y$ twice).

The following statement is a slight strengthening of Lemma~2.7 in \cite{bucic2020halfway} and, as noted there, also follows from a result of Brualdi~\cite{brualdi1969comments}. We include a proof in Section~\ref{sec:simple-matroid-lemmas}.
\begin{lemma}\label{lem:injecttoBi}
    Let $S$ be an independent set in some matroid and let $B$ be a basis of the matroid. Then, there is an injection $\phi: S\to B$ such that, for each $x\in S$, the set $S-x+\phi(x)$ is independent, and, for each $b\in B\setminus \phi(S)$, the set $S+b$ is independent.
\end{lemma}

Recall that here $S-x+\phi(x)$ and $S+b$ are to be interpreted as multi-sets (namely, obtained from $S-x=S\setminus\{x\}$ when adding $\phi(x)$, and from $S$ when adding $b$, respectively). That these sets are independent automatically means that they cannot contain an element twice, meaning that $\phi(x)\not\in S-x$ and $b\not\in S$.

Our final lemma in this section is the following, and is also proved in Section~\ref{sec:simple-matroid-lemmas}.

\begin{lemma}\label{lem:injecttST}
For any independent sets $S$ and $T$ in a matroid, at least one of the following statements holds.
\begin{enumerate}[label = \emph{(\alph{enumi})}]
\item\labelinthm{injopt:1} There is some $x\in S$ such that the set $T+x$ is independent.
\item\labelinthm{injopt:2} There is an injection $\phi: S\to T$ such that, for each $x\in S$, the set $T-\phi(x)+x$ is independent and $\spn(T-\phi(x)+x)=\spn(T)$.
\end{enumerate}
\end{lemma}

Again, in \ref{injopt:1}, the set $T+x$ is to be interpreted as a multi-set. So if \ref{injopt:1} holds, this automatically means that $T+x$ cannot contain an element twice, and so we have $x\not\in T$. Similarly, if \ref{injopt:2} holds, this automatically means that $x\not\in T-\phi(x)$ for every $x\in S$.

Our notation follows general convention. In particular, for a positive integer $n$, we write $[n]=\{1,\dots,n\}$ as usual. For a bipartite graph with vertex classes $X$ and $Y$, by slight abuse of notation, we interpret its edges to be pairs of the form $(x,y)$ with $x\in X$ and $y\in Y$, meaning that its edge set can be interpreted as a subset of $X\times Y$.


\section{Proof outlines}
\label{sec:outline}

In this section, we give outlines of the proofs of our main results in Theorems~\ref{thm:main-packing} and \ref{thm:main-covering}, and state several key lemmas used in these proofs. The proof of Theorem~\ref{thm:main-covering} can be found in Section~\ref{sec:proof-main-covering}, using the results in Section~\ref{sect-first-proof}. The proof of Theorem~\ref{thm:main-packing} can be found in Section~\ref{sec:finish-packing}, using the results in Section~\ref{sec:proof-almostalmost-new} (which, in turn, build on some of the results in Section~\ref{sect-first-proof}).

\subsection{Covering theorem}\label{sec-outline-covering}

As discussed in the previous section, in the setting of our covering theorem (Theorem~\ref{thm:main-covering}), we can consider a coloured matroid of rank $n$, where the colour classes are bases $B_1,\dots,B_n$. We then need to show that their union $B_1\cup\dots\cup B_n$ can be covered by $(1+\eps)n$ rainbow bases. This is equivalent to showing that $B_1\cup\dots\cup B_n$ can be decomposed into $(1+\eps)n$ rainbow independent sets (indeed, given such a decomposition into rainbow independent sets, it is not hard to extend these rainbow independent sets to rainbow bases greedily using the augmentation property of the matroid).

Therefore, our goal is to find a family of $(1+\eps)n$ disjoint rainbow independent sets whose union is the entire ground set $B_1\cup\dots\cup B_n$. Instead of finding such a family directly, for some $\lambda<\eps$ (which we will later choose as $\lambda=\eps/3$), we aim to find a family $\mathcal{T}$ of $(1+\lambda)n$ disjoint rainbow independent sets such that the set of leftover elements not covered by $\mathcal{T}$ can also be decomposed into few (namely, at most $(\eps-\lambda)n$) rainbow independent sets. As discussed in the previous section, we denote the set of elements covered by $\mathcal{T}$ as $E(\mathcal{T})=\bigcup_{T\in \mathcal{T}} T$. Now, letting $U=(B_1\cup \dots\cup B_n)\setminus E(\mathcal{T})$ be the set of uncovered elements, we want to decompose $U$ into few rainbow independent sets. In order to have any hope for such a decomposition, we need $U$ to be sufficiently small (namely, we need $|U|\le (\eps-\lambda)n^2$ if $U$ is to be decomposed into at most $(\eps-\lambda)n$ rainbow independent sets).

It follows from Pokrovskiy's result \cite{pokrovskiy2020rota} stated as Theorem~\ref{thm:pokrovskiy} that it is possible to find a family $\mathcal{T}$ of $\lfloor(1+\lambda)n\rfloor$ disjoint rainbow independent sets such that the set $U=(B_1\cup \dots\cup B_n)\setminus E(\mathcal{T})$ of uncovered elements is rather small. However, instead of Theorem~\ref{thm:pokrovskiy}, we need the following stronger statement, in which the set $U$ of uncovered elements is small for \emph{every} family $\mathcal{T}$ of $\lfloor(1+\lambda)n\rfloor$ disjoint rainbow independent sets such that $E(\mathcal{T})$ is \emph{inclusion-wise maximal}. Here, $E(\mathcal{T})$ being inclusion-wise maximal means that there does not exist a family $\mathcal{S}$ of $\lfloor(1+\lambda)n\rfloor$ disjoint rainbow independent sets with $E(\mathcal{T})\subsetneqq E(\mathcal{S})$.

\begin{lemma}\label{lem:inclusion-maximal-is-large}
For any $\lambda,\nu>0$ the following holds for any sufficiently large $n$ (sufficiently large with respect to $\lambda$ and $\nu$) and any coloured rank-$n$ matroid with colour classes $B_1,\dots,B_n$, such that each of the sets $B_i$ for $i=1,\dots,n$ is a basis. Let $\mathcal{T}$ be a family of $\lfloor (1+\lambda)n\rfloor$ disjoint rainbow independent sets such that $E(\mathcal{T})$ is inclusion-wise maximal. Then we have $|E(\mathcal{T})|\ge (1-\nu)n^2$.
\end{lemma}

This lemma is proved in Section~\ref{sect-first-proof}, and is one of the key ingredients for the proof of Theorem~\ref{thm:main-covering}. It also directly implies Theorem~\ref{thm:pokrovskiy}, as we will show at the end of Section~\ref{sect-first-proof} (so Section~\ref{sect-first-proof} gives a self-contained proof of Theorem~\ref{thm:pokrovskiy} which is somewhat simpler than the original proof in \cite{pokrovskiy2020rota}). The proof of Lemma~\ref{lem:inclusion-maximal-is-large} uses a similar approach as in \cite{pokrovskiy2020rota}, but with some crucial differences, enabling us to obtain the stronger statement in Lemma~\ref{lem:inclusion-maximal-is-large}, which we need for the proof of Theorem~\ref{thm:main-covering}. Some arguments from the proof of Lemma~\ref{lem:inclusion-maximal-is-large} will also later be used in the proof of Theorem~\ref{thm:main-packing}.

The conclusion $|E(\mathcal{T})|\ge (1-\nu)n^2$ of Lemma~\ref{lem:inclusion-maximal-is-large} means that $|U|=n^2-|E(\mathcal{T})|\le \nu n^2$. Of course, just knowing that $U$ has size at most $\nu n^2$ is not enough to decompose $U$ into few rainbow independent sets.
To ensure the existence of such a decomposition, we will need two clearly necessary conditions: that $U$ can be decomposed into few rainbow sets and that $U$ can be decomposed into few independent sets.
It was shown by Aharoni and Berger \cite{aharoni2006intersection} in 2006 that these two conditions  are actually sufficient for a set $U$ to be decomposable into (almost as) few rainbow independent sets (see Theorem~\ref{thm:aharoni-berger} for a restatement of their result). In particular, their result implies that, if $U$ can be decomposed both into $\lambda n$ rainbow sets and into $\lambda n$ independent sets, then $U$ can be decomposed in at most $2\lambda n$ rainbow independent sets. If we can achieve this, then, together with the at most $(1+\lambda)n$ rainbow independent sets in the family $\mathcal{T}$, this would give a decomposition of $B_1\cup \dots\cup B_n$ into at most $(1+3\lambda)n$ rainbow independent sets (so, taking $\lambda=\eps/3$, this gives the desired conclusion).

The first condition, demanding that $U$ is decomposable into $\lambda n$ rainbow sets, is equivalent to requiring that every colour appears in
$U$ at most $\lambda n$ times. On its own, ensuring this would not require significantly more effort than proving Lemma~\ref{lem:inclusion-maximal-is-large}. Much more difficult will be the second condition about decomposing $U$ into independent sets.
A classical theorem of Edmonds \cite{edmonds1965matroiddecomp} implies that $U$ can be decomposed into at most $\lambda n$ independent sets if and only if $|U'|\le \lfloor \lambda n\rfloor\cdot \rk(U')$ for all subsets $U'\su U$. This is clearly necessary, as every independent subset of $U$ can contain at most $\rk(U')$ elements from $U'$.

Effectively, then, we need to make sure that $U$ does not have any ``dense spots'' $U'$ with $|U'|> \lfloor \lambda n\rfloor\cdot \rk(U')$. However, it turns out that keeping track of all subsets $U'\su U$ with $|U'|> \lfloor \lambda n\rfloor\cdot \rk(U')$ is very difficult. In order to overcome this, we work with a different notion of ``dense spots''. We call these dense spots \emph{deadlocks}, and introduce this notion in the first part of Section~\ref{sec:proof-main-covering}. Crucially, the $\lfloor \lambda n\rfloor$-deadlock of $U$ has a non-empty intersection with every subset $U'\su U$ such that $|U'|> \lfloor \lambda n\rfloor\cdot \rk(U')$ (this intersection can be viewed as the ``dense part'' of $U'$ that causes $U'$ to be dense overall). Thus, if the  $\lfloor \lambda n\rfloor$-deadlock of $U$ is empty, then there cannot be any subsets $U'\su U$ with $|U'|> \lfloor \lambda n\rfloor\cdot \rk(U')$, and we will therefore be able to decompose $U$ into $\lfloor \lambda n\rfloor$ independent sets. Our goal is thus to find a family $\mathcal{T}$ of $\lfloor(1+\lambda)n\rfloor$ disjoint rainbow independent sets such that $U=(B_1\cup \dots\cup B_n)\setminus E(\mathcal{T})$ contains every colour at most $\lambda n$ times and such that the $\lfloor \lambda n\rfloor$-deadlock of $U$ is empty.

Using Lemma~\ref{lem:inclusion-maximal-is-large}, we can find  a family $\mathcal{T}$ of $\lfloor(1+\lambda)n\rfloor$ disjoint rainbow independent sets such that $U=(B_1\cup \dots\cup B_n)\setminus E(\mathcal{T})$ satisfies $|U|\leq \nu n^2$. Later, we will denote the $\lfloor \lambda n\rfloor$-deadlock of $U$ by $D_{\lfloor \lambda n\rfloor}(U)$, but in this sketch we will simply call it $D(U)$.
This deadlock will have the property that removing any element $e$ in $D(U)$ from $U$ will reduce the size of $D(U)$, a good thing as we wish to reduce $D(U)$ to the empty set. But in order to remove $e$ from $U$, we need to add it to some $T\in \mathcal{T}$, and there will likely\footnote{If we choose $\mathcal{T}$ as in Lemma~\ref{lem:inclusion-maximal-is-large} such that $E(\mathcal{T})$ is inclusion-wise maximal, then definitely there will be no such $T$.} be no such $T$ for which $T+e$ is a rainbow independent set. So, we will need to pick some $T\in \mathcal{T}$ and, when adding $e$ to $T$, in order to still have a rainbow independent set, we need to remove one or two other elements from $T$ (perhaps one element to maintain rainbowness and perhaps one element to maintain independence). In other words, we replace $T\in \mathcal{T}$ by $(T+e)\setminus F$, where $F$ is a set of at most two elements of $T$ (see part a) of Figure~\ref{fig:alterationsforcovering}). But then the elements in $F$ get added to the uncovered set $U$, and, while removing $e$ from $U$ decreases the size of the deadlock $D(U)$, adding the elements in $F$ to $U$ may increase the size of $D(U)$ again! Therefore, we will have to choose the element $e\in U$, the set $T\in \mathcal{T}$, and the elements in $F$ very carefully.

Utilising the meticulous definition for our notion of deadlock, it turns out that the right condition for this is to ensure that $F$ is disjoint from the span of the deadlock $D'(U)$ defined by a slightly different parameter (namely, to be precise, the $(\lfloor \lambda n\rfloor-2)$-deadlock of $U$). This parameter choice will mean that we have $D(U)\su D'(U)$. However, in order to be able to choose $e$, $T$ and $F$ such that $F$ is disjoint from the span of the deadlock $D'(U)$, we also want $D'(U)$ not to be too much larger than $D(U)$ (more, precisely, we want the rank of $D'(U)$ not to be too much larger than the rank of $D(U)$).

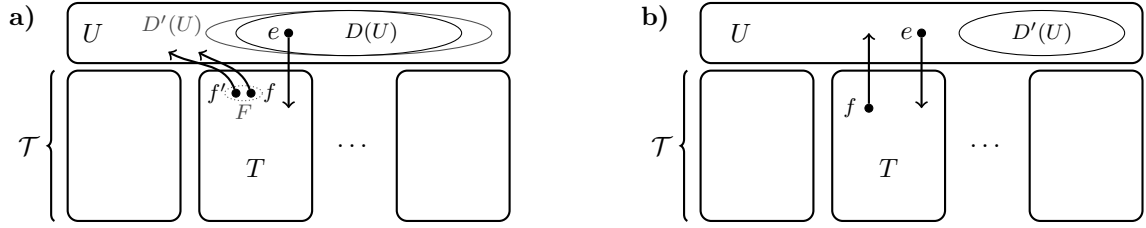
\begin{figure}
\centering
\begin{tikzpicture}
				\def\spacer{1.75};
				\def\Ahgt{1.3};
				\def\Bhgt{1.3};
				\def\ver{0.5};

        \def\verr{1};
        \def\wid{0.75};
				\def\widd{0.75};
        	\def\verrr{0.4};

\coordinate (A) at (0,0);
\coordinate (B) at (\spacer,0);
\coordinate (C) at (1.5*\spacer,0);
\coordinate (D) at (2.5*\spacer,0);

\coordinate (Bplus) at (2*\spacer,0);
\coordinate (Cplus) at (4*\spacer,0);

\coordinate (U) at ($0.5*(B)+0.5*(C)+(0,3*\ver)+(0.4,0)$);
\coordinate (UforD) at ($0.5*(B)+0.5*(C)+(0,3*\ver)+(1,0)$);
\coordinate (Ucent) at ($0.5*(B)+0.5*(C)+(0,3*\ver)$);

\draw  [black!70] ($(UforD)-(0.2,0)$) circle [x radius=1.9,y radius=0.3];

\foreach \midd in {A,B,D}
{
\draw [thick,rounded corners] ($(\midd)-(\widd,0)$) -- ++(0,\verr)
-- ++(2*\widd,0) -- ++(0,-2*\verr) -- ++(-2*\widd,0) -- cycle;
}

\draw ($0.5*(B)+0.5*(D)$) node {$\dots$};

\draw [thick,rounded corners] ($(Ucent)-(1.25*\spacer+\widd,0)$) -- ++(0,\verrr)
-- ++(2.5*\spacer+2*\widd,0) -- ++(0,-2*\verrr) -- ++(-2.5*\spacer-2*\widd,0) -- cycle;

\draw ($(Ucent)-(1.25*\spacer+\widd,-\verrr)-(0.6,0.2)$) node {\bfseries a)};

\node[vertex] (e) at ($(U)+(-0.4,0)$) {};
\draw ($(e)-(0.2,0)$) node {$e$};
\draw [thick,->](e) --++(0,-1);

\node[vertex] (f) at ($(e)+(0,-0.8)+(-0.5,0)$) {};
\node[vertex] (fpr) at ($(e)+(0,-0.8)+(-0.7,0)$) {};
\draw ($(f)+(0.25,0)$) node {\footnotesize $f$};
\draw ($(fpr)+(-0.25,0)$) node {\footnotesize $f'$};

\draw [densely dotted,black!70] ($0.5*(f)+0.5*(fpr)$) circle [x radius=0.25,y radius=0.1];

\draw [black!70] ($0.5*(f)+0.5*(fpr)+(0,-0.25)$) node {\footnotesize $F$};

\draw [thick,->] (f) to[out=110,in=330] ($(f)+(-0.7,0.545)$);
\draw [thick,->] (fpr) to[out=110,in=330] ($(fpr)+(-0.9,0.55)$);

\draw ($(B)-(0,0.3)$) node {$T$};

\draw ($(UforD)-(0.2,0)$) circle [x radius=1.5,y radius=0.3];

\draw ($(UforD)+(0.1,0)$) node { \footnotesize $D(U)$};
\coordinate (U1) at ($(U)+(-2.7,0.1)$);
\draw [black!70] ($(U1)+(0.75,0)$) node {\footnotesize $D'(U)$};

\draw ($(Ucent)-(2.6,0)$) node {$U$};
\draw[decorate,decoration={brace,mirror,raise=5pt},thick]
  (-\widd,\verr) -- (-\widd,-\verr) node[midway,left=7pt]{$\mathcal{T}$};
\end{tikzpicture}\hspace{1.5cm}
\begin{tikzpicture}
				\def\spacer{1.75};
				\def\Ahgt{1.3};
				\def\Bhgt{1.3};
				\def\ver{0.5};

        \def\verr{1};
        \def\wid{0.75};
				\def\widd{0.75};
        	\def\verrr{0.4};

\coordinate (A) at (0,0);
\coordinate (B) at (\spacer,0);
\coordinate (C) at (1.5*\spacer,0);
\coordinate (D) at (2.5*\spacer,0);

\coordinate (Bplus) at (2*\spacer,0);
\coordinate (Cplus) at (4*\spacer,0);

\coordinate (U) at ($0.5*(B)+0.5*(C)+(0,3*\ver)+(0.4,0)$);
\coordinate (UforD) at ($0.5*(B)+0.5*(C)+(0,3*\ver)+(1,0)$);
\coordinate (Ucent) at ($0.5*(B)+0.5*(C)+(0,3*\ver)$);

\draw ($(Ucent)-(1.25*\spacer+\widd,-\verrr)-(0.6,0.2)$) node {\bfseries b)};

\foreach \midd in {A,B,D}
{
\draw [thick,rounded corners] ($(\midd)-(\widd,0)$) -- ++(0,\verr)
-- ++(2*\widd,0) -- ++(0,-2*\verr) -- ++(-2*\widd,0) -- cycle;
}

\draw ($0.5*(B)+0.5*(D)$) node {$\dots$};

\draw [thick,rounded corners] ($(Ucent)-(1.25*\spacer+\widd,0)$) -- ++(0,\verrr)
-- ++(2.5*\spacer+2*\widd,0) -- ++(0,-2*\verrr) -- ++(-2.5*\spacer-2*\widd,0) -- cycle;

\node[vertex] (e) at ($(U)+(-0.4,0)$) {};
\draw ($(e)-(0.2,0)$) node {$e$};
\draw [thick,->](e) --++(0,-1);

\node[vertex] (f) at ($(e)+(0,-1)+(-0.7,0)$) {};
\draw ($(f)+(-0.25,0)$) node {\footnotesize $f$};

\draw [thick,->](f) --++(0,1);

\draw ($(B)-(0,0.3)$) node {$T$};

\draw ($(UforD)+(0.6,0)$) circle [x radius=1.1,y radius=0.3];
\draw ($(UforD)+(0.6,0)$) node { \footnotesize $D'(U)$};

\draw ($(Ucent)-(2.4,0)$) node {$U$};
\draw[decorate,decoration={brace,mirror,raise=5pt},thick]
  (-\widd,\verr) -- (-\widd,-\verr) node[midway,left=7pt]{$\mathcal{T}$};
\end{tikzpicture}
    \caption{\textbf{a)} Given a family $\mathcal{T}$, and its set of uncovered elements $U$, in order to make the deadlock $D(U)$ smaller, we look for some $e\in D(U)$, a set $T\in \mathcal{T}$, and a subset $F\su T$ of size $|F|\le 2$ so that $(T+e)\setminus F$ is a rainbow independent set. We then replace $T\in \mathcal{T}$ by $(T+e)\setminus F$. To ensure the size of the deadlock $D(U)$ decreases, we make sure that $F$ is disjoint from the span of the slightly larger deadlock $D'(U)$. \smallskip
    \newline 
    \textbf{b)} Where a colour $c$ appears too often in $U$, we look for an element $e\in U$ with colour $c$, a set $T\in \mathcal{T}$ and $f\in T$ so that $c$ does not appear on $T$ and $T+e-f$ is an independent set. We then replace $T\in \mathcal{T}$ by $T+e-f$. To maintain the property that $D(U)=\emptyset$, we ensure that $f\notin \spn(D'(U))$. We also make sure that $c(f)$ does not appear that often on $U$.
    }\label{fig:alterationsforcovering}
\end{figure}

To ensure this, instead of two deadlocks, we will consider a sequence of deadlocks with gradually decreasing parameters. These deadlocks will be nested, and increasing in size, and thus we will be able to choose a consecutive pair of deadlocks in this sequence where the rank does not grow too much. We can then decrease the size of the smaller of these two deadlocks via a switching operation as described above, replacing some $T\in \mathcal{T}$ by $(T+e)\setminus F$ for some element $e\in U$ (and some set $F\su T$ of size $|F|\le 2$). Repeating this, we will ultimately be able to decrease the size of the $\lfloor \lambda n\rfloor$-deadlock $D(U)$, and with even more repetitions we obtain $D(U)=\emptyset$.

For clarity, and for a better organisation of the proof, we do not write the proof via a sequence of many such switching steps. Instead, we choose the initial family $\mathcal{T}$ of rainbow independent sets via a certain minimisation, lexicographically minimising the sizes of the relevant sequence of deadlocks. This minimisation will also be chosen in such a way that $E(\mathcal{T})$ is automatically inclusion-wise maximal. 
Hence, Lemma~\ref{lem:inclusion-maximal-is-large} implies that the set $U$ of uncovered elements has size $|U|\leq \nu n^2$. This is crucial for our arguments in order to find suitable elements $e\in U$, $T\in \mathcal{T}$ and $F\su T$ for our desired switching operation replacing $T$ by $(T+e)\setminus F$. Note that this operation can actually increase the size of $U$ (namely, it replaces $U$ by $(U-e)\cup F$ for some set $F$ of size $|F|\le 2$), so during a long sequence of such switching steps the size of $U$ could a priori grow by a lot. This is why it is very important for us that the conclusion of Lemma~\ref{lem:inclusion-maximal-is-large} holds for \emph{every} family $\mathcal{T}$ such that 
$E(\mathcal{T})$ is inclusion-wise maximal. It would not suffice to just find \emph{some} initial family $\mathcal{T}$ such that $|U|\leq \nu n^2$.

Having found a family $\mathcal{T}$ of $\lfloor(1+\lambda)n\rfloor$ disjoint rainbow independent sets with $D(U)=\emptyset$ for the set $U$ of uncovered elements (and having completed a lot of background work to achieve this), the remaining task is to modify the family $\mathcal{T}$ such that $U$ does not contain too many elements of any one colour. Suppose we have a colour $c$ that appears at least $\lambda n$ times on $U$. Picking a set $T\in \mathcal{T}$ without an element of colour $c$, we would like to find an element $e\in U$ of colour $c$ and an element $f\in T$ such that $T+e-f$ is an independent set (it will automatically be rainbow by the choice of $T$). Our goal is then to replace $T$ by $T+e-f$ (see part b) of Figure~\ref{fig:alterationsforcovering}), because then the set $U$ of uncovered elements gets replaced by $U-e+f$, which reduces the number of times that colour $c$ appears in $U$. However, in order to maintain the condition $D(U)=\emptyset$, for some deadlock $D'(U)$ with a slightly different parameter, we need to ensure that $f\notin \spn(D'(U))$ (similarly to the discussion above). We also want to make sure that the colour of $f$ does not appear too often in $U$ (because the number of appearances of that colour increases when replacing $U$ by $U-e+f$). Since the rank of $D'(U)$ will be fairly small, and only few colours appear very often on $U$ (both of these properties are due to the fact that again $|U|$ will be small by Lemma~\ref{lem:inclusion-maximal-is-large}), and the elements in $U$ of colour $c$ have rank at least $\lambda n$, it will be possible to choose elements $e$ and $f$ with the desired properties. Then, upon replacing $T$ by $T+e-f$ we can reduce the number of times that colour $c$ appears in $U$ (while maintaining $D(U)=\emptyset$ and not increasing the number of appearances of any other colour beyond $\lambda n$).

Again, for clarity and a better proof organisation, we actually choose the family $\mathcal{T}$ via a certain minimisation (this time minimising $\sum_{c=1}^{n} |U\cap B_c|^2$), subject to the constraint that $D(U)=\emptyset$. We then show that this family $\mathcal{T}$ must have the property that every colour appears at most $\lambda n$ times in $U$, because otherwise we could improve $\mathcal{T}$ when replacing $T$ by $T+e-f$ as above.

All in all, this yields the desired family $\mathcal{T}$ of $\lfloor(1+\lambda)n\rfloor$ disjoint rainbow independent sets, such for the set $U$ of uncovered elements the $\lfloor \lambda n\rfloor$-deadlock is empty, and such that $U$ contains every colour at most $\lambda n$ times. As discussed above, this is sufficient to prove Theorem~\ref{thm:main-covering}.

The detailed proof of Theorem~\ref{thm:main-covering}, as well as our definition of deadlocks and some lemmas about them, can be found in Section~\ref{sec:proof-main-covering}.

\subsection{Packing theorem}\label{sec-outline-packing}

As discussed in the previous section, in the setting of our packing theorem (Theorem~\ref{thm:main-packing}), we can again consider a coloured matroid of rank $n$, where the colour classes are bases $B_1,\dots,B_n$. This time, we need to show that we can we can find at least $(1-\eps) n$ disjoint rainbow bases in $B_1\cup\dots\cup B_n$. Our basic strategy to prove Theorem~\ref{thm:main-packing} is first to reserve a random subset $R\su B_1\cup \dots\cup B_n$ (the `reservoir'), then to form a family $\mathcal{T}$ of $\lceil (1-\eps)n\rceil$ disjoint rainbow independent sets of large total size in the complement of $R$, and finally to turn $\mathcal{T}$ into a family of disjoint rainbow bases with a procedure using the elements in $R$.

This strategy relies on two key lemmas, stated below.
The first of these lemmas is somewhat similar to Lemma~\ref{lem:inclusion-maximal-is-large} and Pokrovskiy's result \cite{pokrovskiy2020rota} stated as Theorem~\ref{thm:pokrovskiy}.
 It gives a family $\mathcal{T}$ of $\lceil (1-\eps)n\rceil$ disjoint rainbow independent sets of large total size, but with the very important constraint that $\mathcal{T}$ avoids our reservoir $R\su B_1\cup \dots\cup B_n$ of randomly chosen elements. 

\begin{lemma}\label{keylem:almostalmost-new}
For any $0<\eps<\eta<1$ and $\nu>0$ the following holds for any sufficiently large $n$ (sufficiently large with respect to $\eps$, $\eta$ and $\nu$) and any coloured rank-$n$ matroid with colour classes $B_1,\dots,B_n$, such that each of the sets $B_i$ for $i=1,\dots,n$ is a basis. Let $R\su B_1\cup \dots\cup B_n$
be a set of elements drawn independently at random with probability $\eta$. Then, with high probability (more precisely, with probability tending to $1$ as $n\to \infty$ with $\eps$, $\eta$ and $\nu$ fixed), there is a family $\mathcal{T}$ of $\lfloor (1-\eps)n\rfloor$ disjoint rainbow independent sets in $(B_1\cup \dots\cup B_n)\setminus R$ such that $|E(\mathcal{T})|\geq (1-\eta-\nu)n^2$.
\end{lemma}

The proof of Lemma~\ref{keylem:almostalmost-new} can be found in Section~\ref{sec:proof-almostalmost-new}. It is similar to the proof of Lemma~\ref{lem:inclusion-maximal-is-large} in Section~\ref{sect-first-proof} and uses some of the results from that section. However, in the setting of Lemma~\ref{keylem:almostalmost-new}, we need to be more careful as we must avoid using elements from the set $R$. Moreover, it is crucial that in Lemma~\ref{keylem:almostalmost-new} we find a family consisting of only $\lfloor (1-\eps)n\rfloor$ rainbow independent sets, as opposed to the $\lfloor (1+\lambda)n\rfloor$ rainbow independent sets in Lemma~\ref{lem:inclusion-maximal-is-large}. Indeed, the fact that the family consist of fewer sets means that the average size of these sets is larger, and so fewer alterations are needed in order to turn these sets into rainbow bases.

Our second key lemma for the proof of Theorem~\ref{thm:main-packing}, stated next, is used in order to turn the family $\mathcal{T}$ obtained from the previous lemma into a family of disjoint rainbow bases. Roughly speaking, this lemma states that for any family $\mathcal{S}$ of $\lceil (1-\eps)n\rceil$ disjoint rainbow independent sets of total size $|E(\mathcal{S})|< \lceil (1-\eps)n\rceil\cdot n$ which does not contain too many elements from $R$, we can find a family $\mathcal{S}'$ of $\lceil (1-\eps)n\rceil$ disjoint rainbow independent sets of strictly larger total size which contains not too many more elements from $R$. Applying this lemma repeatedly, we can keep increasing the total size of the family until we arrive at such a family of total size $\lceil (1-\eps)n\rceil\cdot n$, i.e., a family of $\lceil (1-\eps)n\rceil$ disjoint  rainbow bases. This proves Theorem~\ref{thm:main-packing} (the details of the deduction of Theorem~\ref{thm:main-packing} from the two key lemmas stated here can be found at the start of Section~\ref{sec:finish-packing}).

\begin{lemma}\label{keylem:completingbases-new}
For any $0<\eps<1/10$, there exist $\sigma=\sigma(\eps)>0$ and $L=L(\eps)>0$, such that for any $\eta$ with $\eps\le \eta< 1$, any sufficiently large $n$ (sufficiently large with respect to $\eps$ and $\eta$), and any coloured rank-$n$ matroid with colour classes $B_1,\dots,B_n$ with the property that each of the sets $B_i$ for $i=1,\dots,n$ is a basis, the following holds with high probability (more precisely, with probability tending to $1$ as $n\to \infty$ with $\eps$ and $\eta$ fixed) for a random subset $R\su B_1\cup \dots\cup B_n$ obtained by drawing elements independently at random with probability $\eta$.
For any family $\mathcal{S}$ of $\lceil (1-\eps)n\rceil$ rainbow independent sets with $\frac{3}{4}n^2\le |E(\mathcal{S})|< \lceil (1-\eps)n\rceil\cdot n$ and
$|E(\mathcal{S})\cap R|\leq \sigma n^2$, there is a family $\mathcal{S}'$ of $\lceil (1-\eps)n\rceil$ rainbow independent sets such that $|E(\mathcal{S}')|>|E(\mathcal{S})|$ and
\begin{equation}\label{eq:betterthefurtherfromabasis}
|E(\mathcal{S}')\cap R|\leq |E(\mathcal{S})\cap R|+L \cdot \ln\Big(\frac{n^2}{\lceil (1-\eps)n\rceil\cdot n-|E(\mathcal{S})|}\Big).
\end{equation}
\end{lemma}

The term $\lceil (1-\eps)n\rceil\cdot n-|E(\mathcal{S})|$ in the denominator here counts the number of elements missing in $\mathcal{S}$ compared to a family of $\lceil (1-\eps)n\rceil$ rainbow bases. The proof of Lemma~\ref{keylem:completingbases-new} can be found in Section~\ref{sec:finish-packing}, but we give an outline of the proof here. To prove the lemma we will use some `cascade-like' switching operations. In order for this to work, we may first have to alter the family $\mathcal{S}$ to get a family $\mathcal{T}$ of $\lceil (1-\eps)n\rceil$ rainbow independent sets of the same total size $|E(\mathcal{T})|=|E(\mathcal{S})|$, before starting these operations. For the purposes of this proof outline, we can assume that $\mathcal{T}=\mathcal{S}$, and we comment on the more technical alteration to $\mathcal{S}$ later, just before the proof of Lemma~\ref{keylem:completingbases-new} in Section~\ref{sec:finish-packing}.

Now, we have a family $\mathcal{T}$ of $\lceil (1-\eps)n\rceil$ disjoint rainbow independent sets, not all of which are bases, such that $E(\mathcal{T})$ does not contain too many elements from the random reservoir $R$ that we reserved before. As before, we write $U=(B_1\cup \dots\cup B_n)\setminus E(\mathcal{T})$ for the set of elements not covered by $\mathcal{T}$. Let us choose a set $T_1\in \mathcal{T}$ which is not a basis, and a colour $c$ not appearing on $T$. From the matroid augmentation property applied to $T_1$ and the basis $B_c$ in colour $c$, we know that there is an element $e$ of colour $c$ for which $T_1+e$ is an independent set (and also automatically rainbow). If $e\in U$ then in the family $\mathcal{T}$ we can simply replace $T_1$ by $T_1+e$ in order to obtain the desired family $\mathcal{S}'$ in Lemma~\ref{keylem:completingbases-new}. However, we should anticipate that there is some $T_2\in \mathcal{T}\setminus \{T\}$ with $e\in T_2$. Then we could move $e$ from $S$ to $T$ to increase the size of $T$, but overall we would not change the total size of $\mathcal{T}$. However, moving $e$ from $T_2$ to $T_1$ will mean that $T_2$ gets replaced by $T_2-e$, which is no longer a basis. So we can perhaps find some $e'\neq e$ which we can add to $T_2-e$. We wish to do this iteratively, making some cascade of changes, until the number of elements we can move is so large that one of them must be in $U$, at which point the corresponding cascade of changes will make an improvement to $\mathcal{T}$. However, it turns out that one cannot hope to succeed simply by moving elements between the different sets in $\mathcal{T}$ until one finds an element in $U$ that can be added to a set in $\mathcal{T}$ without violating rainbowness and independence. Instead, some more complicated switching operations are needed at every step, where in addition to moving elements between different sets $T\in \mathcal{T}$, we also modify the sets $T\in \mathcal{T}$ by exchanging certain elements for other elements in $U$. The cascade of switches we use is motivated by the proof of Theorem~\ref{thm:halfway} by Buci{\'c}, Kwan, Pokrovskiy, and Sudakov \cite{bucic2020halfway}, but crucially different in order to obtain our asymptotically-tight improvement.

To sketch our approach, let us focus on just one step of our switching cascade (as shown in Figure~\ref{fig:alterationsforpacking}). That is, suppose we have found some sets $T_1,\dots,T_r\in \mathcal{T}$ (for some small $r$) and some $T\in \mathcal{T}\setminus \{T_1,\dots,T_r\}$ and $e\in T$ such that $e$ can be moved from $T$ to some modified version of $T_1,\dots,T_r$, while maintaining $|E(\mathcal{T})|$. Now, $T-e$ is certainly missing some colour, $c$ say. There are at least $\eps n$ elements $x\in U$ with colour $c$ (recalling that $|\mathcal{T}|=\lceil (1-\eps)n\rceil$, and every set in $\mathcal{T}$ contains at most one element of colour $c$), and only few of them were used in one of the previous modifications to $T_1,\dots,T_r$ to allow $e$ to be incorporated. For each of the not-yet-used elements $x\in U$ of colour $c$, we consider the set $T-e+x$. This set is always rainbow, but might not be independent anymore, so we now choose some $x'\in T-e$ such that $T-e+x-x'$ is independent. Essentially, this means that in the set $T-e$ we exchange the element $x'\in T-e$ against the element $x\in U$. The different choices for $x$ form an independent set, and therefore this will yield many different pairs $(x,x')$ with distinct elements $x'\in T$ (which have many different colours).

For each of these possible pairs $(x,x')$, we now repeat this procedure for $T-e+x-x'$ instead of $T-e$. Namely, the colour of $x'$ is missing on $T-e+x-x'$, so we can find some elements $q\in U$ with $c(q)=c(x')$ and $q'\in T-e$ such that $T-e+x-x'+q-q'$ is a rainbow independent set (i.e., such that we can exchange $q'$ against $q$ in the rainbow independent set $T-e+x-x'$). However, this time, as there are many different colours appearing among the possible elements $x'$ obtained in the previous step, we also have a lot more possibilities for the element $q$. In fact, it will suffice to restrain ourselves to choose $q$ from the set $R\cap U$. Using the randomness of $R$, it is not hard to show that we have a large span of the elements in $R$ whose colours appear on the possible elements $x'$ in the previous step (and in fact, this span is even robustly large when deleting a small number of elements from $R$). By our assumption that $E(\mathcal{T})$ does not contain many elements of $R$, most of the elements of $R$ are in $U$, and so we can show that there are many potential choices for $q\in R\cap U$ and these choices have a large span. This is the crucial part of our argument, and allows us to show that we will have a huge number of possibilities for $q'\in T-e$ (namely, almost all elements of $T-e$ will appear as a possible choice for $q'$ for some appropriately chosen $x$, $x'$, and $q$). Then we also have many different options for finding an element $e'$ such that $T-e+x-x'+q-q'+e'$ is a rainbow independent set. Most of these elements $e'$ will be contained in some set in $\mathcal{T}\setminus \{T_1,\dots,T_r,T\}$. By the pigeonhole principle, we will find some new set $T'\in \mathcal{T}\setminus \{T_1,\dots,T_r,T\}$ containing a relatively large number of such elements $e'$. This means that, starting with just one element $e\in T$ that we moved out of $T$ (incorporating $e$ into $T_1,\dots,T_r$ via some small modifications), we now found a set $T'$ containing  several different possible elements $e'$ that can be incorporated into $T_1,\dots,T_r,T$ via some small modifications. Iterating this, like a cascade, at every step we obtain an increase in the number of possible elements that can be incorporated into the previous sets. This number will keep increasing (in fact, it will increase exponentially), until at some point we must find an element in $U$ that can be incorporated into the family $\mathcal{T}$ (after some small modifications), meaning that we found a family  $\mathcal{S}'$  improving $\mathcal{T}$, as desired in Lemma~\ref{keylem:completingbases-new}.

\begin{figure}[t]
\centering

\begin{tikzpicture}
				\def\spacer{1.75};
				\def\Ahgt{1.3};
				\def\Bhgt{1.3};
				\def\ver{0.5};

        \def\verr{0.5};
        \def\wid{0.75};
				\def\widd{0.75};
        	\def\verrr{0.3};

\coordinate (A) at (0.5*\spacer,0);
\coordinate (B) at (\spacer,0);
\coordinate (Bplus) at (2*\spacer,0);
\coordinate (C) at (3*\spacer,0);
\coordinate (Cplus) at (4*\spacer,0);
\foreach \n in {5,6,7,8}
{
\coordinate (V\n) at (\n*\spacer-\spacer,0);
}

\coordinate (D) at ($0.5*(B)+0.5*(C)+(0,3*\verr)$);
\coordinate (Dpr) at ($0.5*(B)+0.5*(C)+(0,1.7*\verr)$);

\def\rebal{0.1};

\foreach \midd in {A,Bplus,C}
{
\draw [thick,rounded corners] ($(\midd)-(\widd,0)$) -- ++(0,\verr+\rebal)
-- ++(2*\widd,0) -- ++(0,-3*\verr-\rebal) -- ++(-2*\widd,0) -- cycle;
}

\foreach \n in {5,6,7,8}
{
\draw [thick,rounded corners,dashed] ($(V\n)-(\widd,0)$) -- ++(0,\verr+\rebal)
-- ++(2*\widd,0) -- ++(0,-3*\verr-\rebal) -- ++(-2*\widd,0) -- cycle;
}

\def\reducehgt{0.15};
\def\reducehgtt{0.1};

\def\shrinkR{0.05}
\draw [thick,rounded corners] ($(D)+(0.5*\spacer,0)-(0,0.05)-(2*\spacer+\widd-\shrinkR,\reducehgt)$) -- ++(0,\verrr)
-- ++(6.5*\spacer+2*\widd-2*\shrinkR,0) -- ++(0,-2*\verrr+2*\shrinkR-0.25+0.05) -- ++(-6.5*\spacer-2*\widd+2*\shrinkR,0) -- cycle;
\draw [thick,rounded corners] ($(Dpr)+(0.5*\spacer,0)-(2*\spacer+\widd,-\rebal)$) -- ++(0,3*\verrr-\rebal-\reducehgtt)
-- ++(6.5*\spacer+2*\widd,0) -- ++(0,-4*\verrr+\rebal+\reducehgtt) -- ++(-6.5*\spacer-2*\widd,0) -- cycle;

\draw ($(Dpr)-1.5*(\spacer,0)-(\widd,0)-(0.25,-0.2)$) node {$U$};

\foreach \midd/\lab in {A/T_1,Bplus/T_r}
{
\draw ($(\midd)-(0,0.25)-(0,0.5*\verr)$) node {$\lab$};
}

\draw ($(C)-(-0.1,0.25)-(0,0.5*\verr)$) node {$T$};

\foreach \midd in {C}
{
\node[vertex] (e\midd) at ($(\midd)-(0.75*\widd,0)-(0,0.5*\verr)$) {};
\node[vertex] (fprime\midd) at ($(\midd)+(0.2*\widd,3*\verr)-(0,0.2)$) {};
\node[vertex] (fnew\midd) at ($(fprime\midd)+(0,-2.4*\verr)-0.4*(\widd,0)+(0,0.2)$) {};
\node[vertex] (gnew\midd) at ($(fprime\midd)+(0,-2.4*\verr)+0.4*(\widd,0)+(0,0.2)$) {};

\draw [thick,->] (fnew\midd) --++($(0,0.9*\widd)-(0,0.2)$);
\draw [thick,->] (gnew\midd)  --++($(0,0.9*\widd)-(0,0.2)$);
\draw [thick,->] (fprime\midd) --++($(0,-2.4*\verr)+(0,0.2)$);
}

\foreach \midd in {A,Bplus}
{
\node[vertex,black!60] (fprime\midd) at ($(\midd)+(0.2*\widd,3*\verr)-(0,0.2)$) {};
\node[vertex,black!60] (fnew\midd) at ($(fprime\midd)+(0,-2.4*\verr)-0.4*(\widd,0)+(0,0.2)$) {};
\node[vertex,black!60] (gnew\midd) at ($(fprime\midd)+(0,-2.4*\verr)+0.4*(\widd,0)+(0,0.2)$) {};
\draw [thick,->,black!60] (fnew\midd) --++($(0,0.9*\widd)-(0,0.2)$);
\draw [thick,->,black!60] (gnew\midd) --++($(0,0.9*\widd)-(0,0.2)$);
\draw [thick,->,black!60] (fprime\midd) --++($(0,-2.4*\verr)+(0,0.2)$);
}
\node[vertex,black!60] (eBplus) at ($(Bplus)-(0.75*\widd,0)-(0,0.5*\verr)$) {};
\node[vertex,black!60] (eB) at ($(B)-(0.75*\widd,0)-(0,0.5*\verr)+(0.5*\spacer,0)+(-0.3*\widd,0)$) {};
\foreach \midd in {B,Bplus}
{
\draw [thick,->,black!60] (e\midd) --++(-0.6*\widd,0);
}

\draw [thick,->] (eC) --++(-0.8*\widd,0);
\node[vertex] (eCplus) at ($(Cplus)-(0.75*\widd,0)-(0,0.5*\verr)$) {};
\draw [thick,->] (eCplus) --++(-0.8*\widd,0);

\node[vertex,black!60] (eprimeA) at ($(A)+(-0.6*\widd,1.5*\verr)+(0,0.025)$) {};
\node[vertex,black!60] (eprimeB) at ($(Bplus)+(-0.6*\widd,3*\verr)-(0,0.2)+(0,0.025)$) {};
\node[vertex] (eprimeC) at ($(C)+(-0.6*\widd,1.5*\verr)+(0,0.025)$) {};
\draw [thick,->,black!60] (eprimeB) --++($(0,-2.4*\verr)+(0,0.2)$);
\draw [thick,->,black!60] (eprimeA) --++(0,-0.9*\verr);
\draw [thick,->] (eprimeC) --++(0,-0.9*\verr);

\draw ($(eC)-(0,0.2)$) node {\footnotesize $e$};
\draw ($(fnewC)-(0,0.25)$) node {\footnotesize $x'$};
\draw ($(fprimeC)+(0.2,0)$) node {\footnotesize $q$};
\draw ($(gnewC)-(0,0.25)$) node {\footnotesize $q'$};
\draw ($(eprimeC)-(0.2,0)$) node {\footnotesize $x$};
\draw ($(eCplus)+(0.25,0)$) node {\footnotesize $e'$};

\draw ($0.5*(A)+0.5*(Bplus)-(0,0.5*\verr)+(0.05,0)$) node {\dots};
\draw ($(D)-(1.8,0.2)+(3*\spacer,0)$) node {$R\cap U$};
\draw[decorate,decoration={brace,raise=4pt},thick]
  (7.5*\spacer,\verr+\rebal) -- (7.5*\spacer,-2*\verr) node[midway,right=6pt]{$\mathcal{T}$};
				\end{tikzpicture}

    \caption{Given an element $e\in T$ which is $(T_1,\dots,T_r)$-absorbable, we find an element $e'$ which is $(T_1,\dots,T_r,T)$-absorbable (in fact, we find many such elements). The implicit cascade structure found by our proof is depicted in grey but the critical point is that a small number of alterations to $T_1,\dots,T_r$ can be made using elements from $U$ to allow $e$ to be incorporated. To incorporate $e'$ into $T-e$, we first swap certain elements $x,q\in U$ against elements $x',q'\in T-e$ so that in the end $T-e+x-x'+q-q'+e'$ is a rainbow independent set. 
}\label{fig:alterationsforpacking}\end{figure}
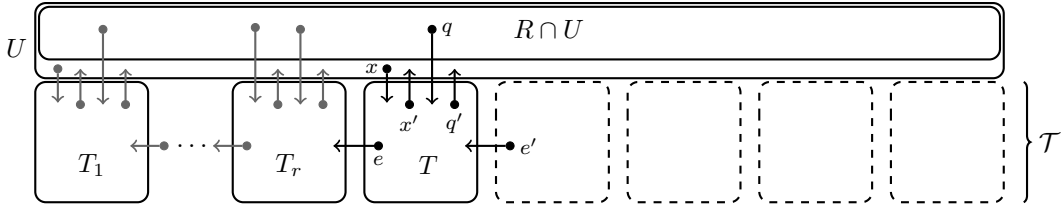

To avoid cumbersome notation, we do not wish to remember all of the individual switching steps that are used in this (potentially rather long) cascade. Instead, we will say that an element $e$ is \emph{$(T_1,\dots,T_r)$-absorbable} if we can make some small alterations to $T_1,\dots,T_r$ in order to incorporate $e$, and then we just count the number of $(T_1,\dots,T_r)$-absorbable elements at every step. Slightly more formally (see Definition~\ref{def:absorbable} for the precise definition), $e$ is $(T_1,\dots,T_r)$-absorbable if there are rainbow independent sets $T_1',\dots,T_r'$ that can be obtained from $T_1,\dots,T_r$ via small alterations (where $T_1',\dots,T_r'$ need to be disjoint from all sets in $\mathcal{T}\setminus \{T_1,\dots,T_r\}$ apart from potentially containing the element $e$) such that in total the size $|T_1'|+\dots+|T_r'|$ increases compared to $|T_1|+\dots+|T_r|$. While, formally, we do not demand that $e$ is contained in one of the sets $T_1',\dots,T_r'$, this will automatically follow unless we found an element in $U$ to incorporate into $T_1,\dots,T_r$ (after some small modifications) to increase their size (in which case we have already found the desired family $\mathcal{S}'$ improving $\mathcal{T}$). At every step, it is essential that we only make `small' alterations, because we need to iterate this argument, and because we need to bound the number of additional elements in $R$ used by the family $\mathcal{S}'$ compared to the initial families $\mathcal{T}$ and $\mathcal{S}$ in order to satisfy \eqref{eq:betterthefurtherfromabasis}. For this reasons, it is also important to bound the total number of steps in the above cascade. In fact, the further away the initial families $\mathcal{S}$ and $\mathcal{T}$ are from a family of $\lceil (1-\eps)n\rceil$ disjoint rainbow bases, the fewer steps we need in the cascade (and therefore fewer additional elements from $R$ are needed). In the proof sketch above, this is because we will be able to start the cascade with some set $T_1\in \mathcal{T}$ which is further away from a basis, so that already at the next step there will be more possibilities for elements to incorporate into $T_1$ (after the usual small modifications) and thus will need to iterate our argument less.

\section{Inclusion-wise maximal families are large: proof of Lemma~\ref{lem:inclusion-maximal-is-large}}
\label{sect-first-proof}

Given a family $\mathcal{T}$ of disjoint rainbow independent sets in a matroid coloured with colours $1,\dots,n$, the \emph{colour-availability graph $\mathcal{A}(\mathcal{T})$} is the bipartite graph with vertex classes $\mathcal{T}$ and $[n]$ and edges $(T,c)$ with $T\in \mathcal{T}$ and $c\in [n]$ exactly when $c$ is a colour that does not appear on $T$. In other words, the colour-availability graph is the auxiliary graph recording which colours are missing, i.e., still available, on which of the rainbow independent sets in $\mathcal{T}$.
Note that for any $T\in \mathcal{T}$ we have $\mathrm{deg}_{\mathcal{A}(\mathcal{T})}(T)=n-|T|$ and for any colour $c\in [n]$, if $B_c$ is the set of elements of colour $c$, then we have $\mathrm{deg}_{\mathcal{A}(\mathcal{T})}(c)=|\mathcal{T}|-|E(\mathcal{T})\cap B_c|\ge |\mathcal{T}|-|B_c|$.

For our arguments, we need to consider edges $(T,c)$ in the colour-availability graph $\mathcal{A}(\mathcal{T})$ of a family $\mathcal{T}$ of disjoint rainbow independent sets such that $\mathrm{deg}_{\mathcal{A}(\mathcal{T})}(c)$ can be bounded in  terms of $\mathrm{deg}_{\mathcal{A}(\mathcal{T})}(T)$. The following graph-theoretic lemma enables us to find many such edges.

\begin{lemma}\label{lem:findinggoodedges}
Let $0<\alpha<\beta$ and $\delta>0$. Let $G$ be a bipartite graph with vertex classes $X$ and $Y$ and edge set $E(G)\su X\times Y$. Assume that $|X|\leq \alpha \cdot |Y|$ and that every vertex $y\in Y$ has degree $\mathrm{deg}_G(y)\geq \delta \cdot |X|$.
Then, there are at least $(\delta (\beta-\alpha)/\beta) \cdot |X|\cdot |Y|$ edges $(x,y)\in E(G)$ with $\mathrm{deg}_G(y)\leq \beta \cdot \mathrm{deg}_G(x)$.
\end{lemma}
\begin{proof}
Assign each edge $(x,y)\in E(G)$ with $x\in X$ and $y\in Y$ the weight $w(x,y)=1/\mathrm{deg}_G(y)$, and call the edge $(x,y)$ \emph{good} if $\mathrm{deg}_G(y)\leq \beta \cdot \mathrm{deg}_G(x)$, and \emph{bad} otherwise. We count the weight on the edges from each side of the partition. Firstly,
\begin{equation}\label{eq:weight-sum-graph}
\sum_{(x,y)\in E(G)}w(x,y)=\sum_{y\in Y}\sum_{x\in N_G(y)}\frac{1}{\mathrm{deg}_G(y)}=\sum_{y\in Y}1=|Y|.
\end{equation}
Secondly, letting $e_{\mathrm{gd}}(G)$ be the number of good edges in $G$, we have
\begin{align*}
\sum_{(x,y)\in E(G)}w(x,y)&= \sum_{\substack{(x,y)\in E(G)\\xy\text{ good}}}\frac{1}{\mathrm{deg}_G(y)}
+\sum_{\substack{(x,y)\in E(G)\\xy\text{ bad}}}\frac{1}{\mathrm{deg}_G(y)}\\
&\leq \sum_{\substack{(x,y)\in E(G)\\xy\text{ good}}}\frac{1}{\delta |X|}
+\sum_{\substack{(x,y)\in E(G)\\xy\text{ bad}}}\frac{1}{\beta \cdot \mathrm{deg}_G(x)}\\
& \leq \frac{e_{\mathrm{gd}}(G)}{\delta  |X|}
+\sum_{x\in X}\sum_{y\in N_G(x)}\frac{1}{\beta \cdot \mathrm{deg}_G(x)}= \frac{e_{\mathrm{gd}}(G)}{\delta  |X|}
+\frac{|X|}{\beta}
\leq \frac{e_{\mathrm{gd}}(G)}{\delta  |X|}
+\frac{\alpha |Y|}{\beta}.
\end{align*}
Combining this with \eqref{eq:weight-sum-graph}, and rearranging, we indeed obtain
\[e_{\mathrm{gd}}(G)\geq \delta  |X|\cdot \Big(1-\frac{\alpha}{\beta}\Big)\cdot |Y|=\frac{\delta (\beta-\alpha)}{\beta}\cdot |X|\cdot |Y|.\qedhere\]
\end{proof}

Instead of using Lemma~\ref{lem:findinggoodedges} directly, we will use the following corollary. The set $E$ in this corollary will simply be obtained by taking the set of edges satisfying the property in Lemma~\ref{lem:findinggoodedges}.

\begin{corollary}\label{cor:findinggoodedges} For any $\alpha,\mu,\delta,\sigma,\lambda>0$ with $\sigma+\lambda\ge \alpha>\sigma$, there is some $\rho>0$ such that the following holds. Let $G$ be a bipartite graph with vertex classes $X$ and $Y$ and edge set $E(G)\su X\times Y$, such that $|X|\leq \alpha \cdot |Y|$. Suppose that $\mathrm{deg}_G(y)\geq \delta \cdot |X|$ for every vertex $y\in Y$ and $\mathrm{deg}_G(x)\leq (1-\mu)\cdot |Y|$ for every vertex $x\in X$. Then, there is a set $E\su E(G)$ such that
\[\sigma\cdot \sum_{(x,y)\in E}\mathrm{deg}_G(x)\geq \rho \cdot |X|\cdot |Y|^2+\sum_{(x,y)\in E}\Big(\mathrm{deg}_G(y)-\lambda \cdot |Y|\Big).\]
\end{corollary}

\begin{proof} Note that $(\alpha-\sigma)(1-\mu)<\alpha-\sigma\le \lambda$, so we can choose $\beta>\alpha$ such that $(\beta-\sigma)(1-\mu)<\lambda$. Now, let us define $\rho=(\delta (\beta-\alpha)/\beta)\cdot (\lambda-(\beta-\sigma)(1-\mu))>0$. By Lemma~\ref{lem:findinggoodedges}, we can find a set $E\su  E(G)$ of size $|E|\geq (\delta (\beta-\alpha)/\beta) \cdot |X|\cdot |Y|$ such that $\mathrm{deg}_G(y)\leq \beta \cdot \mathrm{deg}_G(x)$ for each $(x,y)\in E$. Now, noting that $\beta>\alpha>\sigma$, and thus $(\beta-\sigma)\mathrm{deg}_G(x)\leq (\beta-\sigma)\cdot (1-\mu)|Y|$ for each $x\in X$, we have
\begin{align*}
\sigma\cdot \!\!\sum_{(x,y)\in E}\!\!\mathrm{deg}_G(x)-\!\!\sum_{(x,y)\in E}\!\!\Big(\mathrm{deg}_G(y)-\lambda \cdot |Y|\Big)&= \sum_{(x,y)\in E}\Big(\sigma \cdot \mathrm{deg}_G(x)-\mathrm{deg}_G(y)+\lambda \cdot |Y|\Big)\\
&\ge \sum_{(x,y)\in E}\Big(\sigma \cdot \mathrm{deg}_G(x)-\beta \cdot \mathrm{deg}_G(x)+\lambda \cdot |Y|\Big)\\
&\ge \sum_{(x,y)\in E}\Big(\lambda \cdot |Y|-(\beta-\sigma) \cdot (1-\mu)\cdot |Y|\Big)\\
&= |E|\cdot (\lambda-(\beta-\sigma)(1-\mu))\cdot |Y|\\
&\ge \frac{\delta (\beta-\alpha)}{\beta}(\lambda-(\beta-\sigma)(1-\mu))\cdot |X|\cdot |Y|^2=\rho \cdot |X|\cdot |Y|^2.
\end{align*}
Rearranging yields the desired inequality.
\end{proof}

The following definition, due to Pokrovskiy~\cite{pokrovskiy2020rota} will also play an important role in our proofs.

\begin{defn}\label{def:reduction} Given a family $\mathcal{T}$ of disjoint rainbow independent sets in a coloured matroid, and a positive integer $\ell$, the \emph{$\ell$-reduction} $\mathcal{T}'$ of $\mathcal{T}$ is defined as follows. We consider all elements $x\in E(\mathcal{T})$ with the property that there exist $\ell$ different sets $T\in \mathcal{T}$ such that $T+x$ is a rainbow independent set. Now, let $\mathcal{T}'$ be the family of disjoint rainbow independent sets  obtained from $\mathcal{T}$ by deleting all elements $x\in E(\mathcal{T})$ with this property from their respective rainbow independent sets in $\mathcal{T}$.
\end{defn}

Recall that $T+x$ is to be interpreted as a multi-set here. Hence, $T+x$ being a rainbow independent set in particular means $x\not\in T$.
Intuitively, for every element $e\in E(\mathcal{T})$ with the property in the definition above, there is plenty of flexibility to move $e$ to different member of $\mathcal{T}$ without violating rainbowness or independence. This intuition is formalised by the following lemma, also due to Pokrovskiy~\cite{pokrovskiy2020rota} (this lemma is essentially a special case of \cite[Lemma 11]{pokrovskiy2020rota}). For the reader's convenience we include a simple proof of this lemma here.

\begin{lemma}\label{lem:switching-reduction} Fix integers $r\ge 0$ and $\ell>3^r$. Consider a family $\mathcal{T}=\{T_1,\dots,T_m\}$ of disjoint rainbow independent sets in a coloured matroid, and define families $\mathcal{T}^{(0)},\dots,\mathcal{T}^{(r)}$ of disjoint rainbow independent sets by setting $\mathcal{T}^{(0)}=\mathcal{T}$ and taking $\mathcal{T}^{(i)}$ to be the $\ell$-reduction of the family $\mathcal{T}^{(i-1)}$ for $i=1,\dots,r$. Suppose that there is an element $e\not\in E(\mathcal{T})$ of the matroid such that $T+e$ is a rainbow independent set for some $T\in \mathcal{T}^{(i)}$ for some index $i\in\{0,\dots,r\}$. Then there exists a family $\mathcal{S}=\{S_1,\dots,S_m\}$ of disjoint rainbow independent sets with $E(\mathcal{S})=E(\mathcal{T})+e$, such that $\sum_{j=1}^{m} |S_j\setminus T_j|\le 3^r$.
\end{lemma}

\begin{proof}
We prove the lemma by induction on $r$. In the case $r=0$, by assumption $T+e$ is a rainbow independent set for some $T\in \mathcal{T}^{(0)}=\mathcal{T}$. Then we can take $\mathcal{S}$ to be the family obtained from $\mathcal{T}$ by replacing $T$ with $T+e$ (i.e., by adding $e$ to the set $T$).

So let us now assume that $r\ge 1$ and that we already proved the statement in the lemma for $r-1$. First note that $\mathcal{T}^{(1)}$ is the $\ell$-reduction of the family $\mathcal{T}^{(0)}=\mathcal{T}$. This means that, letting $\mathcal{T}^{(1)}=\{T_1',\dots,T_m'\}$, for $j=1,\dots,m$ we have $T_j'\su T_j$, and for every element $x\in E(\mathcal{T})\setminus E(\mathcal{T}^{(1)})$ there are $\ell$ different indices $j\in [m]$ such that $T_j+x$ is a rainbow independent set.

By assumption, there exist $i\in\{0,\dots,r\}$ and $T\in \mathcal{T}^{(i)}$ such that $T+e$ is a rainbow independent set. If $i\le r-1$, the desired statement already follows from the induction hypothesis for $r-1$, so let us assume that $i=r$. By applying the induction hypothesis to the family $\mathcal{T}^{(1)}=\{T_1',\dots,T_m'\}$ we find a family $\mathcal{S}'=\{S'_1,\dots,S'_m\}$ of disjoint rainbow independent sets with $E(\mathcal{S}')=E(\mathcal{T}^{(1)})+e$, such that $\sum_{j=1}^{m} |S'_j\setminus T_j'|\le 3^{r-1}$.

Now, we can extend $\mathcal{S}'$ to a family $\mathcal{S}^*=\{S^*_1,\dots,S^*_m\}$ of disjoint rainbow independent sets such that for $j=1,\dots,m$ we have $S'_j\su S^*_j\su S'_j\cup (T_j\setminus T'_j)$ and $|T_j\setminus (T'_j\cup S^*_j)|\le 2\cdot |S'_j\setminus T'_j|$. Indeed, for each $j=1,\dots,m$ let us add the elements of $S'_j\setminus T'_j$ to the rainbow independent set $(S'_j\cap T'_j)\cup (T_j\setminus T'_j)$, reinstating independence and resolving colouring conflicts by deleting at most $2\cdot |S'_j\setminus T'_j|$ elements chosen from $T_j\setminus T'_j$ (formally this follows from Lemma~\ref{lem:rainbow-independence-simple-adding} applied to the rainbow independent sets $S_j'$ and $(S'_j\cap T'_j)\cup (T_j\setminus T'_j)$). Then we have
\begin{equation}\label{eq:sumSstarminusT}
\sum_{j=1}^{m} |S^*_j\setminus T_j|=\sum_{j=1}^{m} |S'_j\setminus T_j|\le \sum_{j=1}^{m} |S'_j\setminus T_j'|\le 3^{r-1}.
\end{equation}

Furthermore, note that $E(\mathcal{T}^{(1)})+e=E(\mathcal{S}')\su E(\mathcal{S}^*)\su E(\mathcal{T})+e$ and therefore
\[|E(\mathcal{T})\setminus E(\mathcal{S}^*)|=|E(\mathcal{T})\setminus (E(\mathcal{T}^{(1)}) \cup E(\mathcal{S}^*))|\le  \sum_{j=1}^{m} |T_j\setminus (T'_j\cup S^*_j)|\le 2\sum_{j=1}^{m} |S'_j\setminus T_j'|\le 2\cdot 3^{r-1}.\]
Defining $X=E(\mathcal{T})\setminus E(\mathcal{S}^*)$, let $X=\{x_1,\dots,x_z\}$, and note that then $z=|X|\leq 2\cdot 3^{r-1}$. We will now show that we can distribute the elements in $X$ among the sets in $\mathcal{S}^*$ while maintaining rainbowness and independence.
\begin{claim}\label{clm:manypossibilities}
For each $h\in [z]$, $S^*_j+x_h$ is a rainbow independent set for at least $z$ indices $j\in [m]$.
\end{claim}

Using this claim, we can complete the proof easily. Indeed, by Claim~\ref{clm:manypossibilities}, we can greedily choose  distinct indices $j(1),\dots,j(z)\in [m]$ such that $S^*_{j(h)}+x_h$ is a rainbow independent set for each $h\in [z]$. 
Then, letting $S_{j(h)}=S^*_{j(h)}+x_{j(h)}$ for each $h\in [z]$, and $S_j=S_j^*$ for each $j\in [m]\setminus \{j(1),\dots,j(z)\}$, the family $\mathcal{S}=\{S_1, \dots,S_m\}$ of disjoint rainbow independent sets satisfies $E(\mathcal{S})=E(\mathcal{S}^*)\cup X=E(\mathcal{T})+e$ (as $e\in E(\mathcal{S}')\su E(\mathcal{S}^*)$). By \eqref{eq:sumSstarminusT}, we furthermore have
 $\sum_{j=1}^{m} |S_j\setminus T_j|\leq |X|+\sum_{j=1}^{m} |S^*_j\setminus T_j|\leq 2\cdot 3^{r-1}+3^{r-1}=3^{r}$, as desired. Thus, it only remains to prove Claim~\ref{clm:manypossibilities}.

\smallskip

\noindent\emph{Proof of Claim~\ref{clm:manypossibilities}.} 
Let $h\in [z]$. As $E(\mathcal{T}^{(1)})+e\su E(\mathcal{S}^*)$, we have that $X\su E(\mathcal{T})\setminus E(\mathcal{T}^{(1)})$ and, hence, that $T_j+x_h$ is a rainbow independent set for at least $\ell$ values of $j\in [m]$ by the definition of $\mathcal{T}^{(1)}$. From \eqref{eq:sumSstarminusT}, we can observe that $S_j^*\su T_j$ for all but at most $3^{r-1}$ indices $j\in [m]$. Thus, for at least $\ell-3^{r-1}\geq 2\cdot 3^{r-1}\geq z$ indices $j\in [m]$ we have that $S_j^*+x_h\su T_j+x_h$ and $T_j+x_h$ (and hence $S_j^*+x_h$) is a rainbow independent set. This completes the proof of the claim and hence the lemma.\hspace{2.0cm}$\boxdot$
\end{proof}

We will deduce Lemma~\ref{lem:inclusion-maximal-is-large} from Lemma~\ref{lem:switching-reduction} and the following lemma. The statement of this lemma has some similarities to the statement of \cite[Lemma 10]{pokrovskiy2020rota} due to Pokrovskiy, but there are several important differences between the two statements. These differences are crucial in order to obtain the stronger statement in Lemma~\ref{lem:inclusion-maximal-is-large} compared to Theorem~\ref{thm:pokrovskiy} from \cite{pokrovskiy2020rota}.

\begin{lemma}\label{lem:threeoutcomes} For any $\mu>0$ and $0<\lambda<1$ there exists $\gamma=\gamma(\lambda,\mu)>0$ such that for any positive integer $\ell$  and any sufficiently large $n$ (sufficiently large with respect to $\lambda$, $\mu$ and $\ell$) the following holds. Consider a coloured rank-$n$ matroid with colour classes $B_1,\dots,B_n$, such that each of the sets $B_i$ for $i=1,\dots,n$ is a basis, and let $\mathcal{T}$ be a family of $\lfloor (1+\lambda)n\rfloor$ disjoint rainbow independent sets. Then, at least one of the following three statements holds.
\begin{enumerate}[label = \emph{(\roman{enumi})}]
\item\labelinthm{outcome:1} There is some $T\in \mathcal{T}$  of size $|T|\le \mu n$.
\item\labelinthm{outcome:2} There is some $T\in \mathcal{T}$ and some colour $c\in [n]$ such that there are strictly more than $\mathrm{deg}_{\mathcal{A}(\mathcal{T})}(c)-\lfloor\lambda n\rfloor$ elements $e\in B_c\setminus E(\mathcal{T})$ such that $T+e$ is a rainbow independent set.
\item \labelinthm{outcome:3} The $\ell$-reduction $\mathcal{T}'$ of the family $\mathcal{T}$ satisfies $|E(\mathcal{T}')|\leq |E(\mathcal{T})|-\gamma n^2$.
\end{enumerate}
\end{lemma}

\begin{proof} Let $\rho>0$ be such that the statement in Corollary~\ref{cor:findinggoodedges} holds with $\alpha=1+\lambda$ and $\delta=\lambda/4$ and $\sigma=1$ (as well as the given values of $\lambda$ and $\mu$), noting that $1+\lambda\ge \alpha >1$.  Now, let $\gamma=\rho/3>0$, let $\ell$ be a positive integer, and assume that $n$ is sufficiently large with respect to all the other parameters.  Consider a coloured matroid and a family $\mathcal{T}$ of disjoint rainbow independent sets as in the assumptions of the lemma, and let us assume for a contradiction that none of \ref{outcome:1}, \ref{outcome:2}, or \ref{outcome:3} holds.

Recall the definition of the colour-availability graph $\mathcal{A}(\mathcal{T})$ from the beginning of this section, and note that for every colour $c\in [n]$ we have
\[\mathrm{deg}_{\mathcal{A}(\mathcal{T})}(c)=|\mathcal{T}|-|E(\mathcal{T})\cap B_c|\ge |\mathcal{T}|-n=\lfloor (1+\lambda)n\rfloor-n=\lfloor \lambda n\rfloor\ge \frac{\lambda}{2}\cdot n\ge \frac{\lambda}{4}\cdot |\mathcal{T}|\] (using that $n$ is sufficiently large).
Furthermore, for each $T\in \mathcal{T}$, as \ref{outcome:1} does not hold and thus $|T|>\mu n$, we have $\mathrm{deg}_{\mathcal{A}(\mathcal{T})}(T)=n-|T|< (1-\mu)n$.
Thus, by the property of $\rho$ from its choice via Corollary~\ref{cor:findinggoodedges}, applied with $G=\mathcal{A}(\mathcal{T})$, $X=\mathcal{T}$ and $Y=[n]$, there is some set $E\su  E(\mathcal{A}(\mathcal{T}))$ such that
\begin{align}
\sum_{(T,c)\in E}\mathrm{deg}_{\mathcal{A}(\mathcal{T})}(T)&\geq \rho \cdot \lfloor(1+\lambda)n\rfloor \cdot n^2+\sum_{(T,c)\in E}(\mathrm{deg}_{\mathcal{A}(\mathcal{T})}(c)-\lambda n)\nonumber\\
&\geq \rho n^3-|E|+\sum_{(T,c)\in E}(\mathrm{deg}_{\mathcal{A}(\mathcal{T})}(c)-\lfloor\lambda n\rfloor)\nonumber\\
&\geq 2\gamma n^3+\sum_{(T,c)\in E}|\{e\in B_c\setminus E(\mathcal{T})\midusedassuchthat T+e\text{ rainbow independent}\}|,\label{eq:first-inequality}
\end{align}
where in the last step we have used that \ref{outcome:2} does not hold (and that $\rho n^3-|E|\ge 3\gamma n^3-(1+\lambda)n^2\ge 2\gamma n^3$).

Now, for each $e\in E(\mathcal{T}')$, by Definition~\ref{def:reduction} there are at most $\ell$ different sets $T\in \mathcal{T}$ such that $T+e$ is a rainbow independent set. Thus,
\begin{align*}
\sum_{(T,c)\in E}|\{e\in B_c\cap E(\mathcal{T})\midusedassuchthat T+e\text{ rainbow independent}\}|&\le \sum_{T\in \mathcal{T}}|\{e\in E(\mathcal{T})\midusedassuchthat T+e\text{ rainbow independent}\}|\\
&=\sum_{e\in E(\mathcal{T})}|\{T\in \mathcal{T}\midusedassuchthat T+e\text{ rainbow independent}\}| \\
&\leq |E(\mathcal{T}')|\cdot \ell+|E(\mathcal{T})\setminus E(\mathcal{T}')| \cdot |\mathcal{T}|\\
&\le n^2\cdot \ell+\gamma n^2\cdot (1+\lambda)n
<2\gamma n^3,
\end{align*}
where we have used that \ref{outcome:3} does not hold and therefore $|E(\mathcal{T})\setminus E(\mathcal{T}')|=|E(\mathcal{T})|-|E(\mathcal{T}')|\le \gamma n^2$. Combining this with \eqref{eq:first-inequality} yields
\begin{equation}\label{eq:forcontradictionnew}
    \sum_{(T,c)\in E}\mathrm{deg}_{\mathcal{A}(\mathcal{T})}(T)>\sum_{(T,c)\in E}|\{e\in B_c\midusedassuchthat T+e\text{ rainbow independent}\}|.\end{equation}

On the other hand, for each pair $(T,c)\in E$, there is an edge in $\mathcal{A}(\mathcal{T})$ between $T$ and $c$, so the colour $c$ does not appear on $T$. Furthermore, as $B_c$ is an independent set, there are at least $|B_c|-|T|=n-|T|=\mathrm{deg}_{\mathcal{A}(\mathcal{T})}(T)$ different elements $e\in B_c$ such that $T+e$ is independent. For all of these elements $e$, the set $T+e$ is then a rainbow independent set. Therefore we have
\[\sum_{(T,c)\in E}|\{e\in B_c\midusedassuchthat T+e\text{ rainbow independent}\}|\geq \sum_{(T,c)\in E}\mathrm{deg}_{\mathcal{A}(\mathcal{T})}(T),\]
a contradiction to \eqref{eq:forcontradictionnew}.
\end{proof}

Finally here, let us combine Lemmas~\ref{lem:switching-reduction} and \ref{lem:threeoutcomes} to prove Lemma~\ref{lem:inclusion-maximal-is-large}.

\begin{proof}[Proof of Lemma~\ref{lem:inclusion-maximal-is-large}]
Let $\mu=\nu/2$ and let $\gamma>0$ be chosen as in Lemma~\ref{lem:threeoutcomes} for $\lambda$ and $\mu$. Fix an integer $r>1/\gamma$ and let $\ell=3^r+1$. We may assume that $n$ is sufficiently large with respect to $\lambda$, $\gamma$, $r$ and $\ell$. Suppose, then, that $\mathcal{T}$ is a family of $\lfloor (1+\lambda)n\rfloor$ disjoint rainbow independent sets such that $E(\mathcal{T})$ is inclusion-wise maximal and assume for contradiction that $|E(\mathcal{T})|< (1-\nu)n^2$.

Let $U=(B_1\cup\dots\cup B_n)\setminus E(\mathcal{T})$ be the set of elements of $B_1\cup\dots\cup B_n$ not covered by $\mathcal{T}$, so that $|U|=n^2-|E(\mathcal{T})|> \nu n^2$.
Now, let us define families $\mathcal{T}^{(0)},\dots,\mathcal{T}^{(r)}$ of disjoint rainbow independent sets as in Lemma~\ref{lem:switching-reduction} by setting $\mathcal{T}^{(0)}=\mathcal{T}$ and taking $\mathcal{T}^{(i)}$ to be the $\ell$-reduction of the family $\mathcal{T}^{(i-1)}$ for $i=1,\dots,r$.

If there is some element $e\in U=(B_1\cup\dots\cup B_n)\setminus E(\mathcal{T})$, index $i\in \{0,\dots,r\}$, and $T\in \mathcal{T}^{(i)}$ such that $T+e$ is a rainbow independent set, then, by Lemma~\ref{lem:switching-reduction} there is a family $\mathcal{S}$ of $\lfloor (1+\lambda)n\rfloor$ disjoint rainbow independent sets with $E(\mathcal{S})=E(\mathcal{T})\cup \{e\}$, which contradicts the inclusion-wise maximality of $E(\mathcal{T})$. Thus, we have the observation that there is no element $e\in U=(B_1\cup\dots\cup B_n)\setminus E(\mathcal{T})$ such that $T+e$ is a rainbow independent set for some $T\in \mathcal{T}^{(i)}$ for some index $i\in\{0,\dots,r\}$.

For each $i=0,\dots,r-1$, let us now apply Lemma~\ref{lem:threeoutcomes}  to the family $\mathcal{T}^{(i)}$. Then, for each $i=0,\dots,r-1$, one of the three statements \ref{outcome:1} to  \ref{outcome:3} in Lemma~\ref{lem:threeoutcomes} holds. We will now show that \ref{outcome:1}, and afterwards \ref{outcome:2}, cannot hold for any index $i\in \{0,\dots,r-1\}$, before using that \ref{outcome:3} holds for every $i=0,\dots,r-1$ to reach our final contradiction.

If for some index $i\in \{0,\dots,r-1\}$ statement \ref{outcome:1} holds, then there is a set $T\in \mathcal{T}^{(i)}$ with $|T|\le \mu n=\nu n/2$. If for some colour $c$ not appearing on $T$ we had $|U\cap B_c|>|T|$, then we could find an element $e\in U\cap B_c$ such that $T+e$ is independent. But then $T+e$ would be a rainbow independent set, which is a contradiction to our observation above. Therefore, for every colour $c$ not appearing on $T$ we have $|U\cap B_c|\le |T|$. For the $|T|$ colours $c$ appearing on $T$, we clearly have $|U\cap B_c|\le n$, and so in total we obtain
\[|U|=\sum_{c=1}^n |U\cap B_c|\le n\cdot |T|+|T|\cdot n=|T|\cdot 2n\le \nu n^2,\]
contradicting $|U|> \nu n^2$. Thus, \ref{outcome:1} cannot hold for any index $i\in \{0,\dots,r-1\}$.

If for some index $i\in \{0,\dots,r-1\}$ statement \ref{outcome:2} holds, then there is a set $T\in \mathcal{T}^{(i)}$ and some colour $c\in [n]$ such that there are strictly more than $\mathrm{deg}_{\mathcal{A}(\mathcal{T}^{(i)})}(c)-\lfloor\lambda n\rfloor$ elements $e\in B_c\setminus E(\mathcal{T}^{(i)})$ such that $T+e$ is a rainbow independent set. By our above observation we must have $e\in E(\mathcal{T})$ and hence $e\in B_c\cap (E(\mathcal{T})\setminus E(\mathcal{T}^{(i)}))$ for all of these elements $e$, and consequently
\begin{equation}
|(B_c\cap E(\mathcal{T}))\setminus (B_c\cap E(\mathcal{T}^{(i)}))|=|B_c\cap (E(\mathcal{T})\setminus E(\mathcal{T}^{(i)}))|> \mathrm{deg}_{\mathcal{A}(\mathcal{T}^{(i)})}(c)-\lfloor\lambda n\rfloor.\label{eq:forcontradiction}
\end{equation}
On the other hand, $\mathrm{deg}_{\mathcal{A}(\mathcal{T}^{(i)})}(c)=|\mathcal{T}|-|B_c\cap E(\mathcal{T}^{(i)})|=\lfloor (1+\lambda)n\rfloor-|B_c\cap E(\mathcal{T}^{(i)})|$, and therefore
\begin{align*}
\mathrm{deg}_{\mathcal{A}(\mathcal{T}^{(i)})}(c)-\lfloor\lambda n\rfloor&=\lfloor (1+\lambda)n\rfloor-|B_c\cap E(\mathcal{T}^{(i)})|-\lfloor\lambda n\rfloor=n-|B_c\cap E(\mathcal{T}^{(i)})|\\
&\ge |B_c\cap E(\mathcal{T})|-|B_c\cap E(\mathcal{T}^{(i)})|=|(B_c\cap E(\mathcal{T}))\setminus (B_c\cap E(\mathcal{T}^{(i)}))|.
\end{align*}
This contradicts \eqref{eq:forcontradiction}, and thus \ref{outcome:2} does not hold for any index $i\in \{0,\dots,r-1\}$.

Thus, \ref{outcome:3} must hold for all $i=0,\dots,r-1$. But this means that $|E(\mathcal{T}^{(i+1)})|\leq |E(\mathcal{T}^{(i)})|-\gamma n^2$ for all $i=0,\dots,r-1$. Hence $|E(\mathcal{T}^{(r)})|\leq |E(\mathcal{T}^{(0)})|-r\cdot \gamma n^2=|E(\mathcal{T})|-r\cdot \gamma n^2\le n^2-r\cdot \gamma n^2<0$ (recalling that $r>1/\gamma$), which is again a contradiction. This finishes the proof of  Lemma~\ref{lem:inclusion-maximal-is-large}.
\end{proof}

We end this section by showing how Theorem~\ref{thm:pokrovskiy} can be deduced from Lemma~\ref{lem:inclusion-maximal-is-large}.

\begin{proof}[Proof of Theorem~\ref{thm:pokrovskiy}]
Let $\lambda=\eps^2/2$ and $\nu=\eps^2/2$, and let $n$ be sufficiently large such that the statement in Lemma~\ref{lem:inclusion-maximal-is-large} holds.
Let us choose a family $\mathcal{T}$ of $\lfloor (1+\lambda)n\rfloor$ disjoint rainbow independent sets such that $E(\mathcal{T})$ is inclusion-maximal among all such families. By Lemma~\ref{lem:inclusion-maximal-is-large} we have $|E(\mathcal{T})|\ge (1-\nu)n^2$. Let $m$ be the number of sets $T\in \mathcal{T}$ with size at least $(1-\eps)n$. Then, the $m$ largest sets $T\in \mathcal{T}$ each have size at most $|T|\le n$, and the remaining $\lfloor (1+\lambda)n\rfloor-m\le (1+\lambda) n-m$ sets $T\in \mathcal{T}$ each have size $|T|< (1-\eps)n$, so the total size $|E(\mathcal{T})|$ of the family $\mathcal{T}$ satisfies
\[(1-\nu)n^2\le |E(\mathcal{T})| \le m\cdot  n+((1+\lambda) n-m)\cdot (1-\eps)n=(1+\lambda)(1-\eps)n^2+m\cdot \eps n\le (1-\eps+\lambda)n^2+m\cdot \eps n.\]
This implies $m\ge (\eps-\lambda-\nu)n^2/(\eps n)=(\eps-\eps^2)n/\eps =(1-\eps) n$, meaning that $\mathcal{T}$ contains at least $(1-\eps)n$ sets $T\in \mathcal{T}$ of size $|T|\ge (1-\eps)n$, each of which is rainbow and independent, so the conclusion of the theorem holds.
\end{proof}


\section{Large families avoiding a random set: proof of Lemma~\ref{keylem:almostalmost-new}}
\label{sec:proof-almostalmost-new}

The proof of Lemma~\ref{keylem:almostalmost-new} is similar to the proof of Lemma~\ref{lem:inclusion-maximal-is-large} in the previous section, and will in particular also rely on Corollary~\ref{cor:findinggoodedges} and Lemma~\ref{lem:switching-reduction}.
First, we show that with high probability the random set $R\su B_1\cup \dots\cup B_n$ in  Lemma~\ref{keylem:almostalmost-new} has the properties in the following two lemmas. The property \ref{property:diamond} in the first lemma estimates how often each colour appears in $R$, and the property \ref{property:star} in the second lemma states, roughly speaking, that, for every independent set $T$ and any set $C\su [n]$, there are many ways to extend $T$ to a larger independent set by adding an element in $\bigcup_{c\in C} B_c\setminus R$.

\begin{lemma}\label{lem:reservoir-colour}
For any $0<\eta<1$ and $\gamma>0$, the following holds for any coloured rank-$n$ matroid with colour classes $B_1,\dots,B_n$ of size $n$. Let $R\su B_1\cup \dots\cup B_n$ be a set of elements drawn independently at random with probability $\eta$. Then, with high probability (more precisely, with probability tending to $1$ as $n\to \infty$ with $\eta$ and $\gamma$ fixed), the following holds.
\begin{enumerate}[label = \emph{($\blacklozenge$)}]
\item \labelinthm{property:diamond} For every colour  $c\in [n]$,  we have $(\eta-\gamma)n\le |B_c\cap R|\le (\eta+\gamma)n$.
\end{enumerate}
\end{lemma}
\begin{proof}
For each colour $c\in [n]$, note that  $|B_c\cap R|$ is a sum of $n$ independent Bernoulli random variables  and $\mathbb{E}[|B_c\cap R|]=\eta n$. Thus, by a Chernoff bound (see e.g.\ \cite[Theorem A.1.4]{alon-spencer}) we have
\[\P[||B_c\cap R|-\eta n|>\gamma n]\le 2\cdot e^{-2(\gamma n)^2/n}=2\cdot e^{-2\gamma^2 n}.\]
Taking a union bound over all $c\in [n]$ shows that with probability at least $1-n\cdot 2\cdot e^{-2\gamma^2 n}=1-o(1)$ we indeed have $(\eta-\gamma)n\le |B_c\cap R|\le (\eta+\gamma)n$  for all $c\in [n]$.
\end{proof}

\begin{lemma}\label{lem:reservoir-extension}
For any $0<\eta<1$ and $\gamma>0$, the following holds for any coloured rank-$n$ matroid with colour classes $B_1,\dots,B_n$, such that each of the sets $B_i$ for $i=1,\dots,n$ is a basis. Let $R\su B_1\cup \dots\cup B_n$ be a set of elements drawn independently at random with probability $\eta$. Then, with high probability (more precisely, with probability tending to $1$ as $n\to \infty$ with $\eta$ and $\gamma$ fixed), the following holds.
\begin{enumerate}[label = \emph{($\bigstar$)}]
\item  For every independent set $T\su B_1\cup \dots\cup B_n$, and every colour subset $C\su [n]$, there are at least $(1-\eta)\cdot (n-|T|)\cdot |C|-\gamma n^2$ elements $e\in \bigcup_{c\in C} B_c\setminus R$ such that $T+e$ is independent.\labelinthm{property:star}
\end{enumerate}
\end{lemma}

\begin{proof}
For each independent set $T\su B_1\cup \dots\cup B_n$ and each colour subset $C\su [n]$, there are at least $(n-|T|)\cdot |C|$ elements $e\in \bigcup_{c\in C}B_c$ such that $T+e$ is independent (since there are at least $n-|T|$ such elements in every colour $c\in C$); say the set of those elements $e$ is $E_{T,C}$. Now, property \ref{property:star} holds for $T$ and $C$ precisely when $|E_{T,C}\setminus R|\ge (1-\eta)(n-|T|)|C|-\gamma n^2$. Note that we have $\mathbb{E}[|E_{T,C}\setminus R|]=(1-\eta)|E_{T,C}|\ge (1-\eta)(n-|T|)|C|$. Thus, observing that $|E_{T,C}\setminus R|$ is a sum of $|E_{T,C}|\le n^2$ independent Bernoulli random variables, by a Chernoff bound (see e.g.\ \cite[Theorem A.1.4]{alon-spencer}) we have
\[
\P\Big[|E_{T,C}\setminus R|\leq (1-\eta)(n-|T|)|C|-\gamma n^2\Big]\leq \P\Big[|E_{T,C}\setminus R|\leq \mathbb{E}[|E_{T,C}\setminus R|]-\gamma n^2\Big]\le e^{-2(\gamma n^2)^2/n^2}=e^{-2\gamma^2 n^2}.
\]
for each independent set $T\su B_1\cup \dots\cup B_n$ and each colour subset $C\su [n]$. Note there are at most $\sum_{t=0}^{n}\binom{n^2}{t}\le (n^2)^n=n^{2n}$ independent sets $T\su B_1\cup \dots\cup B_n$ and only  $2^n$ colour subsets $C\su [n]$. Thus, the probability that property \ref{property:star} fails for some choice of $T$ and $C$ is at most
$n^{2n}\cdot 2^n\cdot e^{-2\gamma^2 n^2}=o(1)$,
as required.
\end{proof}

The following lemma is the analogue of Lemma~\ref{lem:threeoutcomes} needed in the proof of Lemma~\ref{keylem:almostalmost-new}.

\begin{lemma}\label{lem:threeoutcomes-reservoir} For any $0<\eps<\eta<1$ and any $\mu>0$ there exists $\gamma=\gamma(\eps,\eta,\mu)>0$ such that for any positive integer $\ell$  and any sufficiently large $n$ (sufficiently large with respect to $\eps$, $\eta$, $\mu$ and $\ell$) the following holds. Consider a coloured rank-$n$ matroid with colour classes $B_1,\dots,B_n$, such that each of the sets $B_i$ for $i=1,\dots,n$ is a basis, and let $R\su B_1\cup \dots\cup B_n$ satisfy properties \emph{\ref{property:diamond}} and \emph{\ref{property:star}} above. Finally, let $\mathcal{T}$ be a family of $\lfloor (1-\eps)n\rfloor$ disjoint rainbow independent sets in $(B_1\cup \dots\cup B_n)\setminus R$. Then at least one of the following three statements holds.
\begin{enumerate}[label = \emph{(\roman{enumi})}]
\item\labelinthm{res-outcome:1} There is some $T\in \mathcal{T}$  of size $|T|\le \mu n$.
\item\labelinthm{res-outcome:2} For some $T\in \mathcal{T}$ and some colour $c\in [n]$ there are strictly more than $\mathrm{deg}_{\mathcal{A}(\mathcal{T})}(c)-|B_c\cap R|+\lceil\eps n\rceil$ elements $e\in B_c\setminus (E(\mathcal{T})\cup R)$ such that $T+e$ is a rainbow independent set.
\item \labelinthm{res-outcome:3} The $\ell$-reduction $\mathcal{T}'$ of the family $\mathcal{T}$ satisfies $|E(\mathcal{T}')|\leq |E(\mathcal{T})|-\gamma n^2$.
\end{enumerate}
\end{lemma}

Note that if this statement holds for some value of $\gamma=\gamma(\eps,\eta,\mu)>0$, it automatically also holds for all smaller positive values of $\gamma$ (indeed, the assumptions \ref{property:diamond} and \ref{property:star} get harder to satisfy for smaller  $\gamma$ and the conclusion in \ref{res-outcome:3} gets easier to satisfy for smaller $\gamma$).

The proof of the lemma is very similar to the proof of Lemma~\ref{lem:threeoutcomes} in the previous section (but it requires some adaptations to incorporate the set $R$).

\begin{proof}[Proof of Lemma~\ref{lem:threeoutcomes-reservoir}]
Let $\rho>0$ be such that the statement in Corollary~\ref{cor:findinggoodedges} holds with $\alpha=1-\eps$, $\delta=(\eta-\eps)/2$, $\sigma=1-\eta$ and $\lambda=\eta-\eps$ (as well as the given value of $\mu$), noting that $(1-\eta)+(\eta-\eps)\ge \alpha>1-\eta$.  Now, let $\gamma=\min(\rho(1-\eps)/4,(\eta-\eps)/4)>0$, let $\ell$ be a positive integer, and assume that $n$ is sufficiently large with respect to all the other parameters.  Let us consider a coloured matroid and a family $\mathcal{T}$ of disjoint rainbow independent sets as in the assumptions of the lemma, and assume for a contradiction that none of \ref{outcome:1}, \ref{outcome:2}, or \ref{outcome:3} holds.

Recall the definition of the colour-availability graph $\mathcal{A}(\mathcal{T})$ from the beginning of Section~\ref{sect-first-proof}, and note that for every colour $c\in [n]$ we have
\[\mathrm{deg}_{\mathcal{A}(\mathcal{T})}(c)=|\mathcal{T}|-|E(\mathcal{T})\cap B_c|\ge \lfloor (1-\eps)n\rfloor-|B_c\setminus R|\ge (1-\eps-\gamma)n-(1-\eta+\gamma)n= (\eta-\eps-2\gamma)\cdot n\ge \frac{\eta-\eps}{2}\cdot |\mathcal{T}|,\]
using that $|B_c\setminus R|=n-|B_c\cap R|\le (1-\eta+\gamma)n$ by {\ref{property:diamond}} and that $n$ is sufficiently large. Furthermore, for each $T\in \mathcal{T}$, as \ref{outcome:1} does not hold and thus $|T|>\mu n$, we have $\mathrm{deg}_{\mathcal{A}(\mathcal{T})}(T)=n-|T|< (1-\mu)n$.
Thus, by the statement in Corollary~\ref{cor:findinggoodedges}, there is some set $E\su E(\mathcal{A}(\mathcal{T}))$ such that
\begin{align}
(1-\eta)\cdot \sum_{(T,c)\in E}\mathrm{deg}_{\mathcal{A}(\mathcal{T})}(T)&\geq \rho \cdot \lfloor (1-\eps)n\rfloor \cdot n^2+\sum_{(T,c)\in E}(\mathrm{deg}_{\mathcal{A}(\mathcal{T})}(c)-(\eta-\eps) n)\nonumber\\
&\geq \rho(1-\eps) n^3-n^2-|E| \cdot(1+\gamma n)+\sum_{(T,c)\in E}(\mathrm{deg}_{\mathcal{A}(\mathcal{T})}(c)-(\eta-\gamma)n+\lceil \eps n\rceil)\nonumber\\
&\geq \rho(1-\eps) n^3-2\gamma n^3+\sum_{(T,c)\in E}(\mathrm{deg}_{\mathcal{A}(\mathcal{T})}(c)-|B_c\cap R|+\lceil\eps n\rceil)\nonumber\\
&\geq 2\gamma n^3+\sum_{(T,c)\in E}|\{e\in B_c\setminus (E(\mathcal{T})\cup R)\midusedassuchthat T+e\text{ rainbow independent}\}|,\label{eq:first-inequality-reservoir}
\end{align}
where in the third step we used that that $|B_c\cap R|\ge (\eta-\gamma)n$ by {\ref{property:diamond}} and in the last step that \ref{outcome:2} does not hold (and that $\rho(1-\eps)\ge 4\gamma$).

Now, for each $e\in E(\mathcal{T}')$, there are at most $\ell$ different sets $T\in \mathcal{T}$ such that $T+e$ is a rainbow independent set. Thus,
\begin{align*}
\sum_{(T,c)\in E}|\{e\in B_c\cap E(\mathcal{T})\midusedassuchthat T+e\text{ rainbow independent}\}|&\le \sum_{T\in \mathcal{T}}|\{e\in E(\mathcal{T})\midusedassuchthat T+e\text{ rainbow independent}\}|\\
&=\sum_{e\in E(\mathcal{T})}|\{T\in \mathcal{T}\midusedassuchthat T+e\text{ rainbow independent}\}| \\
&\leq |E(\mathcal{T}')|\cdot \ell+|E(\mathcal{T})\setminus E(\mathcal{T}')| \cdot |\mathcal{T}|\\
&\le n^2\cdot \ell+\gamma n^2\cdot (1-\eps)n
<\gamma n^3,
\end{align*}
where we have used that \ref{outcome:3} does not hold and therefore $|E(\mathcal{T})\setminus E(\mathcal{T}')|=|E(\mathcal{T})|-|E(\mathcal{T}')|\le \gamma n^2$. Combining this with \eqref{eq:first-inequality-reservoir} yields
\begin{equation}\label{eq:second-inequality-reservoir}
(1-\eta)\cdot\!\!\!\!\sum_{(T,c)\in E}\!\!\!\!(n-|T|)=(1-\eta)\cdot\!\!\!\!\sum_{(T,c)\in E}\!\!\mathrm{deg}_{\mathcal{A}(\mathcal{T})}(T)> \gamma n^3+\!\!\sum_{(T,c)\in E}\!\!\!\!|\{e\in B_c\setminus R\midusedassuchthat T+e\text{ rainbow independent}\}|.
\end{equation}

On the other hand, for each $T\in \mathcal{T}$, let $C(T)\su [n]$ be the set of colours $c\in [n]$ with $(T,c)\in E$. Then, for every $c\in C(T)$, there is an edge in $\mathcal{A}(\mathcal{T})$ between $T$ and $c$, so the colour $c$ does not appear on $T$. Furthermore, for each $T\in \mathcal{T}$, by \ref{property:star} there are at least $(1-\eta)\cdot (n-|T|)\cdot |C(T)|-\gamma n^2$ elements $e\in \bigcup_{c\in C(T)} B_c\setminus R$ such that $T+e$ is independent. For all of these elements $e$, the set $T+e$ is then a rainbow independent set. Therefore, for each $T\in \mathcal{T}$, we have
\[\sum_{c\in C(T)}|\{e\in B_c\setminus R\midusedassuchthat T+e\text{ rainbow independent}\}|\geq (1-\eta)\cdot (n-|T|)\cdot |C(T)|-\gamma n^2.\]
Summing this up for all $T\in \mathcal{T}$, we obtain
\begin{align*}
\sum_{(T,c)\in E}|\{e\in B_c\setminus R\midusedassuchthat T+e\text{ rainbow independent}\}|&\geq \sum_{T\in \mathcal{T}}\Big((1-\eta)\cdot (n-|T|)\cdot |C(T)|-\gamma n^2\Big),\\
&> (1-\eta)\cdot\!\!\sum_{(T,c)\in E}(n-|T|)-\gamma n^3,
\end{align*}
which contradicts \eqref{eq:second-inequality-reservoir}.
\end{proof}

Again, we can combine Lemmas~\ref{lem:switching-reduction} and \ref{lem:threeoutcomes-reservoir} to prove Lemma~\ref{keylem:almostalmost-new}.

\begin{proof}[Proof of Lemma~\ref{keylem:almostalmost-new}]
Let $\mu=\nu/3$ and let $\gamma>0$ be chosen as in Lemma~\ref{lem:threeoutcomes-reservoir}. We may assume without loss of generality that $\gamma<\nu/3$ (since making $\gamma$ smaller will not violate the statement in Lemma~\ref{lem:threeoutcomes-reservoir}). Fix an integer $r>1/\gamma$ and let $\ell=3^r+1$. We may assume that $n$ is sufficiently large with respect to $\eps$, $\eta$, $\gamma$, $r$ and $\ell$. Take, then, any coloured rank-$n$ matroid with colour classes $B_1,\dots,B_n$, such that each of the sets $B_i$ for $i=1,\dots,n$ is a basis, and let $R\su B_1\cup \dots\cup B_n$ be a set of elements drawn independently at random with probability $\eta$.

By Lemmas~\ref{lem:reservoir-colour} and \ref{lem:reservoir-extension}, with high probability the random set $R\su B_1\cup \dots\cup B_n$ satisfies \ref{property:diamond} and \ref{property:star}. Thus, it suffices to prove that whenever these properties are satisfied, there exists a family $\mathcal{T}$ of $\lfloor (1-\eps)n\rfloor$ disjoint rainbow independent sets in $(B_1\cup \dots\cup B_n)\setminus R$ of total size $|E(\mathcal{T})|\ge (1-\eta-\nu)n^2$.

Let us assume then that \ref{property:diamond} and \ref{property:star} hold, and let $\mathcal{T}$ be a family of $\lfloor (1-\eps)n\rfloor$ disjoint rainbow independent sets in $(B_1\cup \dots\cup B_n)\setminus R$ maximising $|E(\mathcal{T})|$. Let us suppose for contradiction that $|E(\mathcal{T})|< (1-\eta-\nu)n^2$. Then the set $U=(B_1\cup\dots\cup B_n)\setminus E(\mathcal{T})$ of elements of $B_1\cup\dots\cup B_n$ not covered by $\mathcal{T}$ has size $|U|=n^2-|E(\mathcal{T})|>(\eta+\nu)n^2$ and therefore
\[|U\setminus R|=|U|-|R|>(\eta+\nu)n^2-\sum_{c=1}^{n}|B_c\cap R|\overset{\ref{property:diamond}}{\ge}(\eta+\nu)n^2-n\cdot (\eta+\gamma)n> (2\nu/3) n^2.\]

Now, we again define families $\mathcal{T}^{(0)},\dots,\mathcal{T}^{(r)}$ of disjoint rainbow independent sets as in Lemma~\ref{lem:switching-reduction} by setting $\mathcal{T}^{(0)}=\mathcal{T}$ and taking $\mathcal{T}^{(i)}$ to be the $\ell$-reduction of the family $\mathcal{T}^{(i-1)}$ for $i=1,\dots,r$.
We observe that there cannot be an element $e\in U\setminus R=(B_1\cup\dots\cup B_n)\setminus (E(\mathcal{T})\cup R)$ such that $T+e$ is a rainbow independent set for some $T\in \mathcal{T}^{(i)}$ for some index $i\in\{0,\dots,r\}$ (since otherwise by Lemma~\ref{lem:switching-reduction} we would contradict the maximality of $|E(\mathcal{T})|$).
Now, for each $i=0,\dots,r-1$, we apply Lemma~\ref{lem:threeoutcomes-reservoir}  to the family $\mathcal{T}^{(i)}$, and conclude that one of the three statements \ref{res-outcome:1} to  \ref{res-outcome:3} holds.

If for some index $i\in \{0,\dots,r-1\}$ statement \ref{res-outcome:1} holds, then there is a set $T\in \mathcal{T}^{(i)}$ with $|T|\le \mu n=\nu n/3$. For every colour $c$ not appearing on $T$ we must have $|(U\setminus R)\cap B_c|\le |T|$ (otherwise we could find an element $e\in (U\setminus R)\cap B_c$ such that $T+e$ is a rainbow independent set). For the $|T|$ colours $c$ appearing on $T$, we have $|(U\setminus R)\cap B_c|\le n$, and so we can conclude
\[|U\setminus R|=\sum_{c=1}^n |(U\setminus R)\cap B_c|\le n\cdot |T|+|T|\cdot n=|T|\cdot 2n\le (2\nu/3) n^2.\]
Since this contradicts $|U\setminus R|> (2\nu/3) n^2$, statement \ref{res-outcome:1} cannot hold for any index $i\in \{0,\dots,r-1\}$.

If for some index $i\in \{0,\dots,r-1\}$ statement \ref{res-outcome:2} holds, then there is a set $T\in \mathcal{T}^{(i)}$ and some colour $c\in [n]$ such that there are strictly more than $\mathrm{deg}_{\mathcal{A}(\mathcal{T}^{(i)})}(c)-|B_c\cap R|+\lceil\eps n\rceil$ elements $e\in B_c\setminus (E(\mathcal{T}^{(i)})\cup R)$ such that $T+e$ is a rainbow independent set. By our above observation we have $e\in E(\mathcal{T})$ and hence $e\in B_c\cap (E(\mathcal{T})\setminus E(\mathcal{T}^{(i)}))$ for all of these elements $e$, so
\begin{equation}\label{eqn:BccapET}
|B_c\cap (E(\mathcal{T})\setminus E(\mathcal{T}^{(i)}))|=|(B_c\cap E(\mathcal{T}))\setminus (B_c\cap E(\mathcal{T}^{(i)}))|> \mathrm{deg}_{\mathcal{A}(\mathcal{T}^{(i)})}(c)-|B_c\cap R|+\lceil\eps n\rceil.
\end{equation}
On the other hand, $\mathrm{deg}_{\mathcal{A}(\mathcal{T}^{(i)})}(c)=|\mathcal{T}|-|B_c\cap E(\mathcal{T}^{(i)})|=\lfloor (1-\eps)n\rfloor-|B_c\cap E(\mathcal{T}^{(i)})|$, and therefore
\begin{align*}
\mathrm{deg}_{\mathcal{A}(\mathcal{T}^{(i)})}(c)-|B_c\cap R|+\lceil\eps n\rceil&=\lfloor (1-\eps)n\rfloor-|B_c\cap E(\mathcal{T}^{(i)})|-|B_c\cap R|+\lceil\eps n\rceil\\
&= n-|B_c\cap R|-|B_c\cap E(\mathcal{T}^{(i)})|=|B_c\setminus R|-|B_c\cap E(\mathcal{T}^{(i)})|\\
&\ge |B_c\cap E(\mathcal{T})|-|B_c\cap E(\mathcal{T}^{(i)})|=|(B_c\cap E(\mathcal{T}))\setminus (B_c\cap E(\mathcal{T}^{(i)}))|.
\end{align*}
This contradiction to \eqref{eqn:BccapET} shows that \ref{res-outcome:2} does not hold for any index $i\in \{0,\dots,r-1\}$.

Thus, \ref{res-outcome:3} holds for all $i=0,\dots,r-1$, so $|E(\mathcal{T}^{(i+1)})|\leq |E(\mathcal{T}^{(i)})|-\gamma n^2$ for all $i=0,\dots,r-1$. This implies $|E(\mathcal{T}^{(r)})|\leq |E(\mathcal{T}^{(0)})|-r\cdot \gamma n^2\le n^2-r\cdot \gamma n^2<0$ (as $r>1/\gamma$), again a contradiction.
\end{proof}

\section{Covering: proof of Theorem~\ref{thm:main-covering}}
\label{sec:proof-main-covering}

In addition to Lemma~\ref{lem:inclusion-maximal-is-large}, the second key for proving Theorem~\ref{thm:main-covering} is to find a suitable notion of ``dense spots'' in the set of uncovered elements to make the approach outlined in Section~\ref{sec:outline} work. Our notion will be introduced in Definition~\ref{def:deadlock} below. In order to make this definition, we first need some preparations.

\begin{defn}\label{def:overcrowded}
    For a positive integer $k$, a subset $S$ of a matroid is called \emph{$k$-overcrowded} if for every subset $S'\su S$ we have
    \[|S\setminus S'|\ge k\cdot (\rk(S)-\rk(S')).\]
\end{defn}

Note that for any $k$-overcrowded $S$, the set $S$ is automatically also $k'$-overcrowded for any positive integer $k'<k$. It turns out that the union of two $k$-overcrowded subsets $S_1$ and $S_2$ of some matroid is automatically also $k$-overcrowded.

\begin{lemma}\label{lem:union-overcrowded}
    Let $k$ be a positive integer, and consider $k$-overcrowded subsets $S_1$ and $S_2$ of some matroid. Then their union $S_1\cup S_2$ is also $k$-overcrowded.
\end{lemma}

The proof of this lemma is fairly simple and can be found in Section~\ref{sec:simple-matroid-lemmas}, in which we collect the proofs of various simple matroid lemmas needed in this paper. Lemma~\ref{lem:union-overcrowded} now enables us to make the following definition.

\begin{defn}\label{def:deadlock}
    For a positive integer $k$, and a subset $U$ of a matroid, let the \emph{$k$-deadlock} $D_k(U)$ of $U$ be the unique inclusion-wise maximal $k$-overcrowded subset of $U$.
\end{defn}

Note that $D_k(U)\su U$, that the $k$-deadlock $D_k(U)$ is $k$-overcrowded, and that for every $k$-overcrowded set $S\su U$ we have $S\su D_k(U)$. Furthermore note that for any subset $U'\su U$ we have $D_k(U')\su D_k(U)$ and for any positive integer $k'<k$ we have $D_k(U)\su D_{k'}(U)$

By our next lemma, the $k$-deadlock of a set $U$ in a matroid can detect whether there exists a subset $S\su U$ with $|S|>k\cdot \rk(S)$. This is important, because the existence of such a subset $S$ is equivalent to $U$ not being decomposable into $k$ independent sets (see also the discussion below).

\begin{lemma}\label{lem:deadlock-detects-dense-spot}
    Let $k$ be a positive integer, and let $U$ be a subset of some matroid. If there exists a set $S\su U$ with $|S|>k\cdot \rk(S)$, then $S$ has a non-empty $k$-overcrowded subset, and so in particular we have $D_k(U)\ne \emptyset$ for the $k$-deadlock of $U$.
\end{lemma}
\begin{proof}
    Let $S^*\su S$ be a minimal subset of $S$ with the property that $|S^*|>k\cdot \rk(S^*)$ (such a subset exists, since $S$ itself has this property). Clearly, $S^*$ must be non-empty. Now, for any subset $S'\subsetneqq S^*$ we have $|S'|\le k\cdot \rk(S')$ and hence
    \[|S^*\setminus S'|=|S^*|-| S'|\ge k\cdot \rk(S^*)-k\cdot \rk(S')=k\cdot (\rk(S^*)-\rk(S')).\]
    Noting that this inequality also holds for $S'= S^*$, we can conclude that  $S^*$ is $k$-overcrowded. The second part of the conclusion follows by noting that $S^*\su D_k(U)$ (since $D_k(U)$ contains all $k$-overcrowded subsets of $S\su U$).
\end{proof}

Due to the classical theorem of Edmonds \cite{edmonds1965matroiddecomp} that a matroid can be decomposed into $k$ independent sets if and only if $|S|\le k\cdot \rk(S)$ for all subsets of the matroid, we obtain the following corollary of Lemma~\ref{lem:deadlock-detects-dense-spot}.

\begin{corollary}\label{cor:deadlock-detects-decomposability}
    Let $k$ be a positive integer, and let $U$ be a subset of some matroid. If $D_k(U)= \emptyset$, then $U$ can be decomposed into $k$ independent sets.
\end{corollary}
\begin{proof}
    If $D_k(U)= \emptyset$, then by Lemma~\ref{lem:deadlock-detects-dense-spot} we have $|S|\le k\cdot \rk(S)$ for all subsets $S\su U$. Thus, by Edmonds' matroid decomposition theorem \cite{edmonds1965matroiddecomp}, $U$ can be decomposed into $k$ independent sets.
\end{proof}

Aharoni and Berger \cite[Theorem 8.9]{aharoni2006intersection} proved that in order to decompose a coloured matroid into few rainbow independent sets it is sufficient to assume that the underlying (uncoloured) matroid is decomposable into few independent sets and that every colour appears only few times. In fact, they actually proved a more general result for the intersection of two arbitrary matroids, but we state their result only in the special case relevant to us (where one of the two matroids is a partition matroid described by the colouring).

\begin{theorem}[\cite{aharoni2006intersection}]\label{thm:aharoni-berger}
    Let $k$ be a positive integer, and $U$ a subset of a coloured matroid, such that $U$ can be decomposed into $k$ independent sets and such that every colour appears at most $k$ times in $U$. Then $U$ can be decomposed into $2k$ rainbow independent sets.
\end{theorem}

We remark that this directly implies the result of Aharoni and Berger \cite{aharoni2006intersection} stated as Theorem~\ref{thm:covering-2n} in the introduction.

By combining Corollary~\ref{cor:deadlock-detects-decomposability} and Theorem~\ref{thm:aharoni-berger}, we see that a set $U$ in a coloured matroid can always be decomposed into $2k$ rainbow independent sets if $D_k(U)= \emptyset$ and every colour appears at most $k$ times in $U$. Our approach for proving Theorem~\ref{thm:main-covering} is to find a family $\mathcal{T}$ of $\lfloor (1+\lambda)n\rfloor$ disjoint rainbow independent sets (for some suitable parameter $\lambda$), such that the set of uncovered elements $U=(B_1\cup\dots\cup B_n)\setminus E(\mathcal{T})$ satisfies these two conditions, and then to decompose $U$ into few additional rainbow independent sets. In order to ensure these conditions for $U$, i.e.\ to ensure that $D_k(U)= \emptyset$ and that every colour appears only few times in $U$, we use switching arguments which exchange certain elements in a set $T\in \mathcal{T}$ by other elements. It is crucial to keep track of how $D_k(U)$ changes during these switching processes. The following lemma allows us to control this.

\begin{lemma}\label{lem:deadlock-lemma}
    Consider positive integers $k'<k$ and disjoint subsets $U$ and $U'$ of some  matroid such that $|U'|\le k-k'$. If $U'\cap \spn(D_{k'}(U))=\emptyset$, then we have $D_{k}(U)=D_{k}(U\cup U')$.
\end{lemma}

We postpone the proof of Lemma~\ref{lem:deadlock-lemma} to Section~\ref{sec:simple-matroid-lemmas}. Our next result shows that, in the setting relevant for Theorem~\ref{thm:main-covering}, we can find a family of $\lfloor (1+\lambda)n\rfloor$ rainbow independent sets so that the set of uncovered elements has an empty $\lfloor \lambda n\rfloor$-deadlock, as long as $n$ is sufficiently large with respect to $\lambda>0$.

\begin{lemma}\label{lem:collectionwithnodeadlock}
 For any $0<\lambda<1$ the following holds for any sufficiently large $n$ (sufficiently large with respect to $\lambda$).  In any coloured rank-$n$ matroid with colour classes $B_1,\dots,B_n$, such that each of the sets $B_i$ for $i=1,\dots,n$ is a basis, there is a family $\mathcal{T}$ of $\lfloor(1+\lambda)n\rfloor$ disjoint rainbow independent sets such that, setting $U=(B_1\cup\dots\cup B_n)\setminus E(\mathcal{T})$, we have $D_{\lfloor \lambda n\rfloor}(U)=\emptyset$.
\end{lemma}
\begin{proof}
Let $\nu=\lambda^2/10$, and recall that we are assuming that $n$ is sufficiently large with respect to $\lambda$ (in particular, large enough that the statement in Lemma~\ref{lem:inclusion-maximal-is-large} holds for $\lambda$ and $\nu=\lambda^2/10$). Let $z$ be the unique positive integer with $\lambda n -4< z\le \lambda n$ which is divisible by $4$. Since $z\le \lfloor \lambda n\rfloor$, we then have $D_{\lfloor \lambda n\rfloor}(U)\su D_{z}(U)$ for any subset $U$ of the matroid. Thus, it suffices to find a family $\mathcal{T}$ of $\lfloor(1+\lambda)n\rfloor$ disjoint rainbow independent sets such that we have $D_{z}(U)=\emptyset$ for $U=(B_1\cup\dots\cup B_n)\setminus E(\mathcal{T})$.

To this end, let us choose a family $\mathcal{T}$ of $\lfloor (1+\lambda)n\rfloor$ disjoint rainbow independent sets such that, when denoting the set of uncovered elements by $U=(B_1\cup\dots\cup B_n)\setminus E(\mathcal{T})$, the $(z/4+2)$-tuple
\begin{equation}\label{eq:tuple-lexicographic}
\Big(|D_z(U)|,|D_{z-2}(U)|,|D_{z-4}(U)|,\dots,|D_{z/2}(U)|, |U|\Big)\in \mathbb{Z}^{(z/4)+2}
\end{equation}
is lexicographically minimised (this means that among all families of $\lfloor (1+\lambda)n\rfloor$ disjoint rainbow independent sets, we choose $\mathcal{T}$ such that size of the $z$-deadlock $|D_z(U)|$ of $U$ is minimised, and among all such choices such that $|D_{z-2}(U)|$ is minimised, and so on).

Then, in particular, the set $E(\mathcal{T})$ is inclusion-wise maximal among all families $\mathcal{T}$ of $\lfloor (1+\lambda)n\rfloor$ disjoint rainbow independent sets. Indeed, if there was a family $\mathcal{T}^*$ of $\lfloor (1+\lambda)n\rfloor$ disjoint rainbow independent sets with $E(\mathcal{T})\subsetneqq E(\mathcal{T}^*)$, then for its set of uncovered elements $U^*=(B_1\cup\dots\cup B_n)\setminus E(\mathcal{T}^*)$ we would have $U^*\subsetneqq U$. This means that $|D_k(U^*)|\le |D_k(U)|$ for all positive integers $k$, and furthermore $|U^*| < |U|$, contradicting our choice of $\mathcal{T}$ to lexicographically minimise \eqref{eq:tuple-lexicographic}. Thus, $E(\mathcal{T})$ is indeed inclusion-wise maximal among all families $\mathcal{T}$ of $\lfloor (1+\lambda)n\rfloor$ disjoint rainbow independent sets.

Consequently, by Lemma~\ref{lem:inclusion-maximal-is-large}, we have $|E(\mathcal{T})|\ge (1-\nu)n^2$ and hence
\[|U|=n^2-|E(\mathcal{T})|\le \nu n^2=\lambda^2 n^2/10.\]

    Let us assume for contradiction that $D_z(U)\ne \emptyset$, then we have $\rk(D_z(U))\ge 1$ (as each element of $B_1\cup \dots\cup B_n$ forms an independent set by itself). Defining $r_k=\rk(D_k(U))$ for $k\in\{z,z-2,z-4,\dots,z/2\}$, we then have $1\le r_z\le r_{z-2}\le r_{z-4}\le \dots\le r_{z/2}\le n$. Since $(1+\lambda/10)^{z/4}\ge (1+\lambda/10)^{\lambda n/4-1}>n$ (by our assumption that $n$ is sufficiently large with respect to $\lambda$), there is an index $j\in\{z/2+2,z/2+4,\dots ,z\}$ such that $r_{j-2}\le (1+\lambda/10) r_j$. Note that, as $D_{j}(U)$ is $j$-overcrowded, by the condition in Definition~\ref{def:overcrowded} with $S=D_{j}(U)$ and $S'=\emptyset$ we have $|D_j(U)|\ge j\cdot \rk(D_{j}(U))$, and, hence
    \begin{align}\label{eq:rankDsmall}
r_{j}=\rk(D_{j}(U))\leq \frac{|D_{j}(U)|}{j}\leq \frac{|U|}{j}\leq \frac{\lambda^2 n^2/10}{z/2+2}\leq \frac{\lambda^2 n^2/10}{\lambda n/2}=\frac{\lambda n}{5}.
\end{align}

Now, let $S=\spn D_{j-2}(U)$, and note that then $\rk(S)=\rk(D_{j-2}(U))=r_{j-2}$ and
\begin{equation}\label{eq:S-is-small}
|S|=\sum_{c=1}^{n} |B_c\cap S|\le \sum_{c=1}^{n} \rk(S)=r_{j-2}\cdot n,
\end{equation}
using that each of the sets $B_1,\dots,B_n$ is independent. Furthermore, let $W\su D_j(U)\su U$ be a maximal independent set in $D_j(U)$, so that $|W|=\rk(D_{j}(U))=r_j$. 

As discussed in the proof outline in Section \ref{sec-outline-covering}, our goal is to move some element $e\in W\su D_j(U)\su U$ into one of the sets $T\in \mathcal{T}$, while removing up to two other elements from $T$ in order to maintain rainbowness and independence. However, we need to ensure that these removed elements do not lie in $S=\spn D_{j-2}(U)$. The following claim will allow us to find an element $e\in W$ and a set $T\in \mathcal{T}$ such that this is possible.

\begin{claim}
    There is some $e\in W$ and $T\in \mathcal{T}$ such that $(T\cap S)+e$ is a rainbow independent set.
\end{claim}
\claimproofstart
We need to show that there is a pair $(e,T)\in W\times \mathcal{T}$ such that $e\not\in \spn(T\cap S)$ and the colour $c(e)$ does not appear on $T\cap S$. Note that in total there are $|W|\cdot |\mathcal{T}|=r_j\cdot \lfloor (1+\lambda)n\rfloor$ pairs $(e,T)\in W\times \mathcal{T}$. Our proof proceeds via a counting argument, bounding above the number of pairs $(e,T)\in W\times \mathcal{T}$ with $e\in \spn(T\cap S)$, as well as the number of pairs $(e,T)\in W\times \mathcal{T}$ such that $c(e)$ appears on $T\cap S$.

First, the number of pairs $(e,T)\in W\times \mathcal{T}$ with $e\in \spn(T\cap S)$ can be bounded above by 
\[\sum_{T\in \mathcal{T}}|W\cap \spn(T\cap S)|\le \sum_{T\in \mathcal{T}}\rk(\spn(T\cap S))=\sum_{T\in \mathcal{T}}\rk(T\cap S)\le \sum_{T\in \mathcal{T}}|T\cap S|\le |S|\overset{\eqref{eq:S-is-small}}{\le} r_{j-2}\cdot n,\]
where in the first step we used that $W$ is an independent set.

Furthermore, there are at most $|W|\cdot r_{j-2}=r_j\cdot r_{j-2}$ pairs $(e,T)\in W\times \mathcal{T}$ such that the colour $c(e)$ of $e$ appears on $T\cap S$. Indeed, for each element $e\in W$, we have $|S\cap B_{c(e)}|\le \rk(S)=r_{j-2}$ (recalling that $B_{c(e)}$ is an independent set) and therefore there can be at most $r_{j-2}$ different sets $T\in \mathcal{T}$ with $T\cap S\cap B_{c(e)}\ne \emptyset$ (meaning that $T\cap S$ contains an element of $B_{c(e)}$, i.e., an element of colour $c(e)$).

Thus, the number of pairs $(e,T)\in W\times \mathcal{T}$ with $e\not\in \spn(T\cap S)$ such that $c(e)$ does not appear on $T\cap S$ is at least
\begin{align*}
   r_j\cdot \lfloor (1+\lambda)n\rfloor-r_{j-2}\cdot n-r_j\cdot r_{j-2}&=r_j\cdot \lfloor (1+\lambda)n\rfloor-r_{j-2}\cdot (n+r_j)\\
   &\overset{\eqref{eq:rankDsmall}}{\ge} r_j\cdot (1+\lambda/2)n-r_{j-2}\cdot (n+\lambda n/5)\\
   &\ge r_j\cdot (1+\lambda/2)n-(1+\lambda/10)r_j \cdot (1+\lambda /5)n>0.\qedhere
\end{align*}
\claimproofend

As in the claim, let $e\in W$ and $T\in \mathcal{T}$ be chosen such that $(T\cap S)+e$ is a rainbow independent set. Now, we can form a rainbow independent set $T^*$ of the form $T^*=(T+e)\setminus F$ for some set $F\su T\setminus S$ of size $|F|\le 2$, by adding the element $e$ to the rainbow independent set $T$ and reinstating independence and resolving colouring conflicts by deleting a set $F\su T\setminus S$ consisting of at most two elements (it suffices to delete one element to ensure independence and one element to resolve a potential colouring conflict with $e$, more formally this follows from Lemma~\ref{lem:rainbow-independence-simple-adding} applied to the rainbow independent sets $(T\cap S)+e$ and $T$).

Now, let $\mathcal{T}^*$ be the family of $\lfloor (1+\lambda)n\rfloor$ disjoint rainbow independent sets obtained from $\mathcal{T}$ when replacing $T\in \mathcal{T}$ by $T^*=(T+e)\setminus F$. Then we have $E(\mathcal{T}^*)=(E(\mathcal{T})\setminus F)+e$ and the set $U^*=(B_1\cup \dots\cup B_n)\setminus E(\mathcal{T}^*)$ of elements not covered by $\mathcal{T}^*$ can be described as $U^*=(U\cup F)-e$.

    Now, for any integer $k\ge j$ we have
    \[D_k(U^*)=D_k((U\cup F)-e)\su D_k(U\cup F)=D_k(U),\]
    where in the last step we used Lemma~\ref{lem:deadlock-lemma}, recalling that $F\su T\setminus S=T\setminus \spn(D_{j-2}(U))$ and $|F|\le 2\le k-(j-2)$. Furthermore, for $k=j$ we have \[D_j(U^*)=D_j((U\cup F)-e)\subsetneqq D_j(U\cup F)=D_j(U),\]
    where the strict inclusion in the second step follows from the fact that $e\in D_j(U)\su D_j(U\cup F)$, but $e\not\in D_j((U\cup F)-e)$ as $e\not\in (U\cup F)-e$. This shows that for all positive integers $k\ge j$ we have $|D_k(U^*)|\le |D_k(U)|$, and for $k=j$ we even have $|D_k(U^*)|< |D_k(U)|$. This contradicts our choice of $\mathcal{T}$ to lexicographically minimise \eqref{eq:tuple-lexicographic}, recalling that $j\in\{z,z-2,z-4,\dots,z/2+2\}$. Thus, we have $D_z(U)=\emptyset$.
\end{proof}

Using Lemmas~\ref{lem:inclusion-maximal-is-large} and \ref{lem:collectionwithnodeadlock}, we can now prove Theorem~\ref{thm:main-covering}.

\begin{proof}[Proof of Theorem~\ref{thm:main-covering}]
We may assume that $0<\eps<1$. Let $\lambda=\eps/3$ and $\nu=\lambda^2/4=\eps^2/36$, and recall that we are assuming that $n$ is sufficiently large with respect to $\eps$ (in particular, large enough such that the statement in Lemma~\ref{lem:inclusion-maximal-is-large} holds for $\lambda=\eps/3$ and $\nu=\eps^2/36$, and the statement in Lemma~\ref{lem:collectionwithnodeadlock} holds for $\lambda=\eps/3$).

Consider a rank-$n$ matroid with $n$ given bases $B_1, \dots, B_n$ (we may assume that $B_1, \dots, B_n$ are disjoint by introducing additional copies of any element that appears in more than one of these bases). By restricting the matroid to the set $B_1\cup \dots\cup B_n$, and colouring each of the bases $B_1, \dots, B_n$ with a different colour, we obtain a coloured matroid. Using Lemma~\ref{lem:collectionwithnodeadlock}, let $\mathcal{T}$ be a family of $\lfloor (1+\lambda)n\rfloor$ disjoint rainbow independent sets such that, when denoting the set of uncovered elements by $U=(B_1\cup\dots\cup B_n)\setminus E(\mathcal{T})$, we have $D_{\lfloor \lambda n\rfloor}(U)=\emptyset$ and, subject to this, $\sum_{c=1}^n |B_c\cap U|^2$ is minimised.

Now, similarly as in the proof of Lemma~\ref{lem:collectionwithnodeadlock}, the set $E(\mathcal{T})$ is inclusion-wise maximal among all families $\mathcal{T}$ of $\lfloor (1+\lambda)n\rfloor$ disjoint rainbow independent sets (indeed, if there was such a family $\mathcal{T}^*$ with $E(\mathcal{T})\subsetneqq E(\mathcal{T}^*)$, then for its set of uncovered elements $U^*=(B_1\cup\dots\cup B_n)\setminus E(\mathcal{T}^*)$ we would have $U^*\subsetneqq U$ and therefore $D_{\lfloor \lambda n\rfloor}(U^*)\su D_{\lfloor \lambda n\rfloor}(U)=\emptyset$ and $\sum_{c=1}^n |B_c\cap U^*|^2<\sum_{c=1}^n |B_c\cap U|^2$). Consequently, by Lemma~\ref{lem:inclusion-maximal-is-large}, we have $|E(\mathcal{T})|\ge (1-\nu)n^2$ and hence
\[|U|=n^2-|E(\mathcal{T})|\le \nu n^2=\lambda^2 n^2/4.\]
In particular, there can be at most $\lambda n/2$ different colours $c\in \{1,\dots,n\}$ with $|B_c\cap U|\ge \lambda n/2$.

Furthermore, as the set $D_{\lfloor \lambda n\rfloor-1}(U)\su U$ is $(\lfloor \lambda n\rfloor-1)$-overcrowded, it follows from Definition~\ref{def:overcrowded} that (similarly to \eqref{eq:rankDsmall} in the proof of Lemma~\ref{lem:collectionwithnodeadlock})
\[
\rk(D_{\lfloor \lambda n\rfloor-1}(U))\le \frac{\nu n^2}{\lfloor \lambda n\rfloor-1}\le \frac{\lambda^2 n^2/4}{\lambda n/2}= \frac{\lambda n}{2}.
\]

\begin{claim}\label{claim:colours-rare-in-U}
For every colour $c=1,\dots,n$, we have $|B_c\cap U|\le \lambda n$.
\end{claim}
\claimproofstart
Assume, for contradiction, that there is some colour $c'\in [n]$ with $|B_{c'} \cap U|> \lambda n$. Let $T\in \mathcal{T}$ be a set on which colour $c'$ does not appear, and apply Lemma~\ref{lem:injecttST} to the independent sets $S=B_{c'}\cap U=B_{c'}\setminus E(\mathcal{T})$ and $T$. If there was an element $x\in B_{c'}\setminus E(\mathcal{T})$ as in \ref{injopt:1} in Lemma~\ref{lem:injecttST}, then $T+x$ would be a rainbow independent set, and so one could obtain a family  $\mathcal{S}$ of $\lfloor (1+\lambda)n\rfloor$ rainbow independent sets with $E(\mathcal{S})=E(\mathcal{T})+x$ from the family $\mathcal{T}$ simply by adding $x$ to the set $T\in \mathcal{T}$. This would contradict $E(\mathcal{T})$ being inclusion-wise maximal, and so \ref{injopt:2} in Lemma~\ref{lem:injecttST} must hold. This means that there is an injection $\phi: B_{c'}\cap U \to T$ such that for every element $x\in B_{c'}\cap U$, the set $T-\phi(x)+x$ is a rainbow independent set.

We claim that there must be an element $x\in B_{c'}\cap U$ with $\phi(x)\not\in \spn(D_{\lfloor \lambda n\rfloor-1}(U))$ and $|B_{c(\phi(x))}\cap U|< \lambda n/2$.  Indeed, since $T$ is an independent set, we have $|T\cap \spn(D_{\lfloor \lambda n\rfloor-1}(U))|\le \rk(\spn(D_{\lfloor \lambda n\rfloor-1}(U)))=\rk(D_{\lfloor \lambda n\rfloor-1}(U))\le \lambda n/2$, so there are at most $\lambda n/2$ elements $x\in B_{c'}\cap U$ with $\phi(x)\in \spn(D_{\lfloor \lambda n\rfloor-1}(U))$. Furthermore, recalling that there are at most $\lambda n/2$ colours $c\in [n]$ with $|B_{c}\cap U|\ge \lambda n/2$, there are at most $\lambda n/2$ elements $x\in B_{c'}\cap U$ with $|B_{c(\phi(x))}\cap U|\ge \lambda n/2$. Thus, since $|B_{c'}\cap U|> \lambda n$, there must indeed be an element $x\in B_{c'}\cap U$ with $\phi(x)\not\in \spn(D_{\lfloor \lambda n\rfloor-1}(U))$ and $|B_{c(\phi(x))}\cap U|< \lambda n/2$.

Now, let us consider the family $\mathcal{T}^*$ of $\lfloor (1+\lambda)n\rfloor$ rainbow independent sets obtained from $\mathcal{T}$ when replacing $T\in \mathcal{T}$ by $T-\phi(x)+x$. Note that for the set $U^*=(B_1\cup\dots\cup B_n)\setminus E(\mathcal{T}^*)$ we then have $U^*=U-x+\phi(x)$. In particular, we obtain $B_{c'}\cap U^*=(B_{c'}\cap U)-x$ and $B_{c(\phi(x))}\cap U^*=(B_{c(\phi(x))}\cap U)+\phi(x)$ (and $B_{c}\cap U^*=B_{c}\cap U$ for all other colours $c$).

Using $\phi(x)\not\in \spn(D_{\lfloor \lambda n\rfloor-1}(U))$, Lemma~\ref{lem:deadlock-lemma} applied to $U'=\{\phi(x)\}$, $U$, $k=\lfloor \lambda n\rfloor$ and $k'=\lfloor \lambda n\rfloor-1$ yields
\[D_{\lfloor \lambda n\rfloor}(U^*)=D_{\lfloor \lambda n\rfloor}(U-x+\phi(x))\su D_{\lfloor \lambda n\rfloor}(U\cup\{\phi(x)\})=D_{\lfloor \lambda n\rfloor}(U)=\emptyset.\]
Furthermore, we have
\begin{align*}
&\sum_{c=1}^n |B_c\cap U|^2-\sum_{c=1}^n |B_c\cap U^*|^2=|B_{c'}\cap U|^2-|B_{c'}\cap U^*|^2-|B_{c(\phi(x))}\cap U^*|^2+|B_{c(\phi(x))}\cap U|^2\\
&\qquad\qquad=|B_{c'}\cap U|^2-(|B_{c'}\cap U|-1)^2-(|B_{c(\phi(x))}\cap U|+1)^2+|B_{c(\phi(x))}\cap U|^2>0,
\end{align*}
using $|B_{c(\phi(x))}\cap U|<\lambda n/2<\lambda n-1<|B_{c'}\cap U|-1$ and the general inequality $a^2+b^2>(a-1)^2+(b+1)^2$ for all real numbers $a,b$ with $a-b>1$. In other words, we obtain $\sum_{c=1}^n |B_c\cap U^*|^2<\sum_{c=1}^n |B_c\cap U|^2$,  contradicting our choice of $\mathcal{T}$.
\claimproofend

Now, as $D_{\lfloor\lambda n\rfloor}(U)=\emptyset$, by Corollary~\ref{cor:deadlock-detects-decomposability}, the set $U$ can be decomposed into $\lfloor \lambda n\rfloor$ independent sets. Furthermore, by Claim~\ref{claim:colours-rare-in-U}, every colour appears at most $\lfloor \lambda n\rfloor$ times in $U$. Thus, by Theorem~\ref{thm:aharoni-berger}, the set $U=(B_1\cup\dots\cup B_n)\setminus E(\mathcal{T})$ can be decomposed into $2\cdot \lfloor \lambda n\rfloor=2\cdot \lfloor \eps n/3\rfloor\le (2\eps/3)n$ rainbow independent sets. Together with the $\lfloor (1+\lambda)n\rfloor\le (1+\eps/3)n$ rainbow independent sets in the family $\mathcal{T}$, this forms a decomposition of $B_1\cup\dots\cup B_n$ into at most $(1+\eps)n$ rainbow independent sets.
Each of these rainbow independent sets can greedily be extended to a rainbow basis. Thus, we obtain a covering of $B_1\cup\dots\cup B_n$ by at most $(1+\eps)n$ rainbow bases (i.e., transversal bases).
\end{proof}

\section{Packing: proof of Lemma~\ref{keylem:completingbases-new} and Theorem~\ref{thm:main-packing}}

\label{sec:finish-packing}

Let us start this section by showing how Theorem~\ref{thm:main-packing} can be deduced from Lemmas~\ref{keylem:almostalmost-new} and \ref{keylem:completingbases-new}. The rest of this section will then be devoted to proving Lemma~\ref{keylem:completingbases-new} (recall that we already proved Lemma~\ref{keylem:almostalmost-new} in Section~\ref{sec:proof-almostalmost-new}).

\begin{proof}[Proof of Theorem~\ref{thm:main-packing}]
We may assume without loss of generality that $\eps<1/10$. Given $0<\eps<1/10$, let $\sigma=\sigma(\eps)>0$ and $L=L(\eps)>0$ be as in Lemma~\ref{keylem:completingbases-new}. Now, choose $0<\nu<1/20$ such that $3\nu\cdot (\ln(1/\nu)+1)<\sigma/L$, define $\eta=\eps+\nu$, and assume that $n$ is sufficiently large with respect to the other parameters.

Consider a rank-$n$ matroid and $n$ bases $B_1, \dots, B_n$ of it (we may assume that $B_1, \dots, B_n$ are disjoint by introducing additional copies of any element that appears in more than one of these bases). By restricting the matroid to the set $B_1\cup \dots\cup B_n$, and colouring each of the bases $B_1, \dots, B_n$ with a different colour, we obtain a coloured matroid. Our goal is to show that there is a family of $\lceil (1-\eps)n \rceil$ disjoint rainbow bases in $B_1\cup \dots\cup B_n$.

Now, let $R\su B_1\cup \dots\cup B_n$ be a set of elements drawn independently at random with probability $\eta$, then $R$ satisfies the statements in Lemmas~\ref{keylem:almostalmost-new} and \ref{keylem:completingbases-new} with high probability. In particular, we can fix an outcome of $R$ satisfying both of these statements.

By the statement in Lemma~\ref{keylem:almostalmost-new}, we can find a family $\mathcal{T}$ of $\lceil (1-\eps)n\rceil$ disjoint rainbow independent sets in $(B_1\cup \dots\cup B_n)\setminus R$ such that $|E(\mathcal{T})|\geq (1-\eta-\nu)n^2> \frac{3}{4}n^2$ (note that a priori the statement in Lemma~\ref{keylem:almostalmost-new} gives such a family $\mathcal{T}$ consisting of $\lfloor (1-\eps)n\rfloor$ sets, but in case $\lfloor (1-\eps)n\rfloor\ne \lceil (1-\eps)n\rceil$, we can add a copy of the empty set to the family $\mathcal{T}$, or partition one of the sets $T\in \mathcal{T}$ into two subsets).

Let $m=\lceil (1-\eps)n\rceil\cdot n$ be the number of elements that a family of $\lceil (1-\eps)n\rceil$ disjoint rainbow bases would have. By our assumption that $n$ is sufficiently large, we have $m\le (1-\eps+\nu)n^2$ and therefore $m-|E(\mathcal{T})|\le (1-\eps+\nu)n^2-(1-\eta-\nu)n^2=(\eta -\eps+2\nu)n^2=3\nu n^2$. Furthermore, note that for every family $\mathcal{S}$ of $\lceil (1-\eps)n\rceil$ disjoint rainbow independent sets we have $|E(\mathcal{S})|\le \lceil (1-\eps)n\rceil\cdot n=m$.

Now, let $\mathcal{S}$ be a family of $\lceil (1-\eps)n\rceil$ disjoint rainbow independent sets in $B_1\cup \dots\cup B_n$ maximising $|E(\mathcal{S})|$ subject to the constraint
\[
|E(\mathcal{S})\cap R|\leq \sum_{i=|E(\mathcal{T})|}^{|E(\mathcal{S})|-1}L\cdot \ln\Big(\frac{n^2}{m-i}\Big).
\]
The family $\mathcal{T}$ satisfies this constraint (as $|E(\mathcal{T})\cap R|=0$), so $\mathcal{S}$ is well-defined and we have $|E(\mathcal{S})|\ge |E(\mathcal{T})|> \frac{3}{4}n^2$.

If $|E(\mathcal{S})|=m$, then $\mathcal{S}$ must be a family of $\lceil (1-\eps)n\rceil$ disjoint rainbow bases, finishing the proof of Theorem~\ref{thm:main-packing}. Assume, then, for contradiction that $|E(\mathcal{S})|<m$. Note that then we have
\[|E(\mathcal{S})\cap R|\le \sum_{i=|E(\mathcal{T})|}^{|E(\mathcal{S})|-1}L\cdot \ln\Big(\frac{n^2}{m-i}\Big)=L\cdot \sum_{j=m-|E(\mathcal{S})|+1}^{m-|E(\mathcal{T})|} \ln\Big(\frac{n^2}{j}\Big)\le L\cdot \sum_{j=2}^{\lfloor 3\nu n^2\rfloor} \ln\Big(\frac{n^2}{j}\Big)\le L \cdot \frac{\sigma n^2}{L}= \sigma n^2,\]
using that
\[\sum_{j=2}^{\lfloor 3\nu n^2\rfloor} \ln\Big(\frac{n^2}{j}\Big)\le \int_{1}^{3\nu n^2} \ln\Big(\frac{n^2}{x}\Big) \operatorname{d}x=\Big[x\ln\Big(\frac{n^2}{x}\Big)+x\Big]_{x=1}^{3\nu n^2}\le 3\nu \ln(1/\nu)\cdot n^2 +3\nu n^2\le \frac{\sigma n^2}{L}.\]
Thus, we can apply Lemma~\ref{keylem:completingbases-new} to the family $\mathcal{S}$ (also recalling that $\frac{3}{4}n^2< |E(\mathcal{S})|<m=\lceil (1-\eps)n\rceil\cdot n$), and conclude that there is a family $\mathcal{S}'$ of $\lceil (1-\eps)n\rceil$ disjoint rainbow independent sets in $B_1\cup \dots\cup B_n$ with $|E(\mathcal{S}')|>|E(\mathcal{S})|$ and
\begin{align*}
|E(\mathcal{S}')\cap R|&\leq |E(\mathcal{S})\cap R|+L \cdot \ln\Big(\frac{n^2}{\lceil (1-\eps)n\rceil\cdot n-|E(\mathcal{S})|}\Big)=|E(\mathcal{S})\cap R|+L \cdot \ln\Big(\frac{n^2}{m-|E(\mathcal{S})|}\Big)\\
&\le \sum_{i=|E(\mathcal{T})|}^{|E(\mathcal{S})|-1}L\cdot \ln\Big(\frac{n^2}{m-i}\Big)+L \cdot \ln\Big(\frac{n^2}{m-|E(\mathcal{S})|}\Big)\le \sum_{i=|E(\mathcal{T})|}^{|E(\mathcal{S}')|-1}L\cdot \ln\Big(\frac{n^2}{m-i}\Big).
\end{align*}
This contradicts our choice of the family $\mathcal{S}$.
\end{proof}

It remains to prove Lemma~\ref{keylem:completingbases-new}. The following lemma characterises a property of the random set $R$ needed in the proof (as we will show later, the statement in Lemma~\ref{keylem:completingbases-new} holds for any subset $R\su B_1\cup \dots\cup B_n$ satisfying property \ref{property:spade} in the lemma below with $\gamma$ chosen sufficiently small with respect to $\eps$).

\begin{lemma}\label{lem:reservoir-extension-2}
For any $0<\eta<1$ and $\gamma>0$, the following holds for any coloured rank-$n$ matroid with colour classes $B_1,\dots,B_n$, such that each of the sets $B_i$ for $i=1,\dots,n$ is a basis. Let $R\su B_1\cup \dots\cup B_n$ be a set of elements drawn independently at random with probability $\eta$. Then, with high probability (more precisely, with probability tending to $1$ as $n\to \infty$ with $\eta$ and $\gamma$ fixed), the following holds.
\begin{enumerate}[label = \emph{($\spadesuit$)}]
\item\labelinthm{property:spade} For every independent set $T\su B_1\cup \dots\cup B_n$, and every colour subset $C\su [n]$, there are at least $\eta\cdot (n-|T|)\cdot |C|-\gamma n^2$ elements $e\in \bigcup_{c\in C} (B_c\cap R)$ such that $T+e$ is independent.
\end{enumerate}
\end{lemma}

\begin{proof}
The proof is completely analogous to the proof of Lemma~\ref{lem:reservoir-extension}. Alternatively, the lemma here can be deduced directly from the statement of Lemma~\ref{lem:reservoir-extension} applied to the random set $(B_1\cup \dots\cup B_n)\setminus R$ (in which elements are drawn with probability $1-\eta$).
\end{proof}

The relevance of property \ref{property:spade} is actually that it implies property \ref{property:triangle} in the following lemma. Roughly speaking, this property states that, for any reasonably large colour subset $C\su [n]$, the union $\bigcup_{c\in C} (B_c\cap R)$ has large rank, even after deleting up to $\gamma n^2$ elements from $R$.

\begin{lemma}\label{lem:reservoir-big-span}
For $0<\eta<1$ and $\gamma, \eps'>0$ with $(\eps')^2\eta>2\gamma$, consider a coloured rank-$n$ matroid with colour classes $B_1,\dots,B_n$ and a subset $R\su B_1\cup \dots\cup B_n$ satisfying \emph{\ref{property:spade}} above. Then the following holds.
\begin{enumerate}[label = \emph{($\blacktriangle$)}]
\item \labelinthm{property:triangle} For any subset $Q\su R$ of size $|Q|\ge |R|-\gamma n^2$ and any colour subset $C\su [n]$ of size $|C|\ge \eps' n$ we have $\rk(\bigcup_{c\in C} (B_c\cap Q))\ge (1-\eps')n$.
\end{enumerate}
\end{lemma}

\begin{proof}
    Suppose for contradiction that $|\rk(\bigcup_{c\in C} (B_c\cap Q))|< (1-\eps')n$, and let $T$ be a maximal independent subset of $\bigcup_{c\in C} (B_c\cap Q)$. Then $|T|=|\rk(\bigcup_{c\in C} (B_c\cap Q))|<(1-\eps')n$, and so by \ref{property:spade} there are at least $\eta\cdot (n-|T|)\cdot |C|-\gamma n^2\ge \eta\cdot \eps'n\cdot \eps'n-\gamma n^2>\gamma n^2$ elements $e\in \bigcup_{c\in C} (B_c\cap R)$ such that $T+e$ is independent. By the maximality of $T$, for each of these elements $e$, we have $e\not\in\bigcup_{c\in C} (B_c\cap Q)$. Thus, there are more than $\gamma n^2$ elements in $\bigcup_{c\in C} (B_c\cap R)\setminus \bigcup_{c\in C} (B_c\cap Q)\su R\setminus Q$, and we must have $|R\setminus Q|>\gamma n^2$, contradicting our assumptions on $Q$.
\end{proof}

We also need the following lemma, stating that if we can make two different switches to an independent set, then we can also combine these switches in certain ways.

\begin{lemma}\label{lem:double-switch}
    Let $T$ be an independent set in a matroid and let $x,x',q,q'$ be elements of the matroid with $x,q\not\in T$ and $x',q'\in T$, such that $T-x'+x$ and $T-q'+q$ are independent sets with  $\spn(T-x'+x)=\spn(T-q'+q)=\spn(T)$. Then at least one of the following statements holds.
    \begin{enumerate}[label = \emph{(\alph{enumi})}]
\item\labelinthm{doubleswitchopt:1} The set $T-q'+x$ is independent with $\spn(T-q'+x)=\spn(T)$.
\item\labelinthm{doubleswitchopt:2} We have $x'\ne q'$, and the set $T-x'-q'+x+q$ is independent with $\spn(T-x'-q'+x+q)=\spn(T)$.
\end{enumerate}
\end{lemma}

\begin{proof}
If $x'=q'$, then $T-q'+x=T-x'+x$ and so \ref{doubleswitchopt:1} holds by assumption. So let us assume that $x'\ne q'$. Now, note that $x\in \spn(T-x'+x)=\spn(T)$, and $q\in \spn(T-q'+q)=\spn(T)$, so we can conclude that $\spn(T-q'+x)\su \spn(T)$ and $\spn(T-x'-q'+x+q)\su\spn(T)$. Since we also have $|T-q'+x|=|T-x'-q'+x+q|=|T|$, it suffices to show that the set $T-q'+x$ is independent or the set $T-x'-q'+x+q$ is independent.

    Note that the set $T-x'-q'+x\su T-x'+x$ is independent and has size $|T-x'-q'+x|=|T|-1$, which is strictly smaller than the size $|T-q'+q|=|T|$ of the independent set $T-q'+q$. Thus, there exists an element $e\in (T-q'+q)\setminus (T-x'-q'+x)=\{q,x'\}$ such that $T-x'-q'+x+e$ is an independent set. If $e=q$, this means $T-x'-q'+x+q$ is independent, and if $e=x'$,  this means that $T-q'+x$ is independent.
\end{proof}

The following definition is crucial for our proof of Lemma~\ref{keylem:completingbases-new}. Roughly speaking, for distinct members $T_1,\dots,T_r$ of some family $\mathcal{T}$ of rainbow independent sets, it describes the situation that an element $e$ outside $T_1\cup\dots\cup T_r$ can be absorbed into $T_1,\dots,T_r$, by slightly modifying the sets $T_1,\dots,T_r$ in such a way that the total size $|T_1|+\dots+|T_{r}|$ increases (but without touching any elements in $E(\mathcal{T})\setminus (T_1\cup\dots\cup T_r)$ except $e$). See also the outline in Section~\ref{sec-outline-packing} for the motivation behind this definition. Our definition here is inspired by the notion of cascade-addability in \cite[Definition 2.8]{bucic2020halfway}. However, our notion of absorbability here is more flexible, and this additional flexibility is crucial in our proof of Lemma~\ref{keylem:completingbases-new} (the switching operations considered for cascade-addability in \cite[Definition 2.8]{bucic2020halfway} are not strong enough for our purposes).

\begin{defn}\label{def:absorbable}
    Consider a family of disjoint rainbow independent sets $\mathcal{T}$ in a coloured matroid, and let $U$ be the set of matroid elements outside of $E(\mathcal{T})$ (i.e., the uncovered set). For distinct $T_1,\dots,T_r\in \mathcal{T}$, we say that an element $e$ of the matroid with $e\not\in T_1\cup\dots\cup T_r$ is \emph{$(T_1,\dots,T_r)$-absorbable with respect to $\mathcal{T}$} if there are disjoint rainbow independent sets $T_1', \dots , T_{r}'\su U\cup T_1\cup\dots \cup T_{r}\cup \{e\}$ such that $|T'_i\setminus T_i|\le 3$ for $i=1,\dots r$, and $|T_1'|+\dots+|T_{r}'|\ge |T_1|+\dots+|T_{r}|+1$.

Furthermore, for a set $T_{r+1}\in \mathcal{T}\setminus \{T_1,\dots,T_r\}$, let us denote by
$\numberofabsorbable_{\mathcal{T}}(T_1,\dots,T_r,T_{r+1})$ the number of elements $e\in T_{r+1}$ such that $e$ is $(T_1,\dots,T_r)$-absorbable with respect to $\mathcal{T}$.
\end{defn}

Note that in Definition~\ref{def:absorbable} we do not demand that $e\in T_1'\cup \dots \cup T_{r}'$. However, in the setting of the proof of Lemma~\ref{keylem:completingbases-new} we will automatically have that $e\in T_1'\cup \dots \cup T_{r}'$. Indeed, if we have $e\not\in T_1'\cup \dots \cup T_{r}'$, then we can obtain a family of disjoint rainbow independent sets from $\mathcal{T}$ with strictly larger total size, simply replacing $T_1,\dots,T_r$ by $T'_1,\dots,T'_r$. In our proof of Lemma~\ref{keylem:completingbases-new}, this would trivially give the desired family $\mathcal{S}'$. So we may assume that $e\in T_1'\cup \dots \cup T_{r}'$ (then $T_1',\dots ,T_{r}'$ are not disjoint from the sets in $\mathcal{T}\setminus \{T_1,\dots,T_r\}$), which motivates the word ``absorbable''.

The following lemma states that under certain assumptions for a given $T\in \mathcal{T}$ we can find many $(T)$-absorbable elements (i.e.\ many elements satisfying the condition in Definition~\ref{def:absorbable} for $r=1$ and $T_1=T$).

\begin{lemma}\label{lem:many-1-absorbable}
    Let $\gamma>0$ and $\eps'>0$. Consider a family $\mathcal{T}$ of at most $(1-\eps')n$ rainbow independent sets in a coloured rank-$n$ matroid with colour classes $B_1,\dots,B_n$, such that each of the sets $B_i$ for $i=1,\dots,n$ is a basis, and let  $U=(B_1\cup \dots\cup B_n)\setminus E(\mathcal{T})$. Let $R\su B_1\cup \dots \cup B_n$ be a subset satisfying \emph{\ref{property:triangle}},
     and assume that $|E(\mathcal{T})\cap R|\le \gamma n^2$. Let $T\in \mathcal{T}$ have size $|T|<n$, and assume that there is no rainbow independent set $S\su U\cup T$ with $|S|>|T|$ and $|S\setminus T|\le 2$.

    Then there is a set of colours $C\su [n]$ of size $|C|\ge (1-\eps')n$, such that the following condition holds: For every colour $c\in C$, every element in $B_c\setminus \spn(T)$ is $(T)$-absorbable with respect to $\mathcal{T}$.
\end{lemma}

In other words, the condition on the set $C$ here states that every element $e$ such that $T+e$ is independent and $e$ has a colour in $C$ must be $(T)$-absorbable with respect to $\mathcal{T}$. The proof of this lemma relies on switching arguments like in the proof outline in Section~\ref{sec-outline-packing} (see also Figure~\ref{fig:alterationsforpacking}).

\begin{proof}[Proof of Lemma~\ref{lem:many-1-absorbable}]
    Since $|T|<n$, we can fix a colour $c^*\in [n]$ not appearing on $T$. Note that we have $B_{c^*}\cap E(\mathcal{T})\le |\mathcal T|\le (1-\eps')n$ and therefore $|B_{c^*}\cap U|\ge \eps' n$.

    Let us now apply Lemma~\ref{lem:injecttST} to the independent sets $B_{c^*}\cap U$ and $T$. If option \ref{injopt:1} there holds, there would be an element $x\in B_{c^*}\cap U$ such that $T+x$ is independent. But then  $T+x\su U\cup T$ is a rainbow independent set with $|T+x|>|T|$ and $|(T+x)\setminus T|< 2$. This would be a contradiction to our assumption.

    Thus, we must have option \ref{injopt:2} in Lemma~\ref{lem:injecttST}, so there is an injection $\phi^*: B_{c^*}\cap U\to T$ such that, for each $x\in B_{c^*}\cap U$, the set $T-\phi^*(x)+x$ is independent and $\spn(T-\phi^*(x)+x)=\spn(T)$. Note that the set $T-\phi^*(x)+x$ is also rainbow, since the colour $c^*$ of $x$ does not appear on $T$, and we have $T-\phi^*(x)+x\su U\cup T$.

    Now, let $C^*\su [n]$ be the set of colours appearing on the elements in the image of $\phi^*$, i.e., the elements $\phi^*(x)$ for $x\in B_{c^*}\cap U$. Then $|C^*|=|B_{c^*}\cap U|\ge \eps' n$.
    Furthermore, taking $Q=R\cap U$, note that $|Q|=|R|-|E(\mathcal{T})\cap R|\ge |R|-\gamma n^2$. Thus, by \ref{property:triangle} we have $\rk(\bigcup_{c\in C^*}(B_c \cap Q))\ge (1-\eps')n$. So we can find an independent set $Q^*\su \bigcup_{c\in C^*}(B_c \cap Q)$ of size $|Q^*|\ge (1-\eps')n$. By definition of $C^*$, for every element $q\in Q^*$ there exists an element $x\in B_{c^*}\cap U$ such that $\phi^*(x)$ has the same colour as $q$. Note that then the set $T-\phi^*(x)+x+q$ is rainbow.

    Now, let us again apply Lemma~\ref{lem:injecttST}, this time to the independent sets $Q^*$ and $T$. If option \ref{injopt:1} holds, there is some $q\in Q^*\su Q\su U$ such that the set $T+q$ is independent. This means that $q\not\in \spn(T)=\spn(T-\phi^*(x)+x)$, where $x\in B_{c^*}\cap U$ is such that $\phi^*(x)$ has the same colour as $q$. But then $T-\phi^*(x)+x+q\su U\cup T$ is a rainbow independent set of size $|T|+1$ with $|(T-\phi^*(x)+x+q)\setminus T|\le 2$. This again contradicts our assumption.

    Thus, again option \ref{injopt:2} in Lemma~\ref{lem:injecttST} must hold. So we obtain an injection $\phi': Q^*\to T$ such that, for each $q\in Q^*$, the set $T-\phi'(q)+q$ is independent and $\spn(T-\phi'(q)+q)=\spn(T)$. Finally, let $C\su [n]$ be the set of colours appearing on the image $\phi'(Q^*)\su T$. Then we indeed have $|C|=|\phi'(Q^*)|=|Q^*|\ge (1-\eps')n$.

    It remains to show that, for each $c\in C$, every element in $B_c\setminus \spn(T)$ is $(T)$-absorbable with respect to $\mathcal{T}$. So let $c\in C$, then there exists an element $q\in Q^*\su U$ such that its image $\phi'(q)$ has colour $c$. Let $x\in B_{c^*}\cap U$ be such that $\phi^*(x)$ has the same colour as $q$. Recall that $T-\phi^*(x)+x$ is an independent set with $\spn(T-\phi^*(x)+x)=\spn(T)$, and that $T-\phi'(q)+q$ is an independent set with $\spn(T-\phi'(q)+q)=\spn(T)$. Thus, by Lemma~\ref{lem:double-switch}, we have that $T-\phi'(q)+x$ is an independent set with $\spn(T-\phi'(q)+x)=\spn(T)$, or that $T-\phi^*(x)-\phi'(q)+x+q$ is an independent set with $\spn(T-\phi^*(x)-\phi'(q)+x+q)=\spn(T)$. Noting that both of these sets are rainbow, in either case we find a rainbow independent set $T'\su U\cup T$ of size $|T'|=|T|$ with $\spn(T')=\spn(T)$ and $|T'\setminus T|\le 2$, such that the colour $c=c(\phi'(q))$ does not appear on $T'$. Now, for each $e\in B_c\setminus \spn(T)$, note that $T'+e$ is a rainbow independent set (since $e\not\in \spn(T)=\spn(T')$ and the colour $c$ of $e$ does not appear on $T'$). So $e$ is $(T)$-absorbable with respect to $\mathcal{T}$, since we can take $T_1'=T'+e$ in Definition~\ref{def:absorbable}.
\end{proof}

Finally, we are ready to prove Lemma~\ref{keylem:completingbases-new}. As outlined in Section~\ref{sec-outline-packing}, the proof strategy is to find a cascade of sets $T_1,\dots,T_{r_{\mathrm{max}}}\in \mathcal{T}$ such that the number $\numberofabsorbable_{\mathcal{T}}(T_1,\dots,T_r,T_{r+1})$ of $(T_1,\dots,T_r)$-absorbable elements in $T_{r+1}$ grows exponentially with $r$. Here, $\mathcal{T}$ is either identical to the original family $\mathcal{S}$ given in Lemma~\ref{keylem:completingbases-new}, or $\mathcal{T}$ is a family of $\lceil (1-\eps)n\rceil$ rainbow independent sets of the same total size $|E(\mathcal{T})|= |E(\mathcal{S})|$ as $\mathcal{S}$, where $|E(\mathcal{T})\cap R|$ is not much larger than $|E(\mathcal{S})\cap R|$, and which contains sets $T_1,T_2\in \mathcal{T}$ such that $\numberofabsorbable_\mathcal{T}(T_1,T_2)$ is very large (i.e., such that $T_2$ contains a lot of $(T_1)$--absorbable elements). Finding such $T_1,T_2\in \mathcal{T}$ with very large $\numberofabsorbable_\mathcal{T}(T_1,T_2)$ is very helpful for starting our cascade with growing $\numberofabsorbable_{\mathcal{T}}(T_1,\dots,T_r,T_{r+1})$. If no family $\mathcal{T}$ with these conditions containing such $T_1,T_2\in \mathcal{T}$ exists, then we simply take $\mathcal{T}=\mathcal{S}$. Then at every step, we will have a set $T_{r+1}\in \mathcal{T}=\mathcal{S}$ containing many $(T_1,\dots,T_r)$-absorbable elements. As outlined in Section~\ref{sec-outline-packing}, we can then make some small modifications to $T_{r+1}$ to incorporate many different possible new elements, in order to show that there are many $(T_1,\dots,T_r,T_{r+1})$-absorbable elements. However, if most of these possible new elements to be added to $T_{r+1}$ (after small modifications) lie in $T_1\cup \dots\cup T_r$, they are by definition not $(T_1,\dots,T_r,T_{r+1})$-absorbable, so this would be problematic. In the case where we take $\mathcal{T}=\mathcal{S}$, this cannot happen because otherwise one of the sets $T_1,\dots,T_r$ would contain many $(T'_{r+1})$-absorbable elements (where $T'_{r+1}$ denotes the modified version of $T_{r+1}$), and this would allow us to choose $\mathcal{T}$ differently.

\begin{proof}[Proof of Lemma~\ref{keylem:completingbases-new}]
Let $0<\eps<1/10$, and note that then $1+\eps/4\ge e^{\eps/8}$ (since the inequality $1+2x\ge e^x$ holds for all $x\in[0,1]$). Define $\sigma=\eps^3/20$ and $L=10^7/\eps^5$. As in the statement of Lemma~\ref{keylem:completingbases-new}, let $\eps\le \eta<1$, assume that $n$ is large with respect to $\eps$ and $\eta$, and consider a coloured rank-$n$ matroid where the colour classes $B_1,\dots,B_n$ are bases. Given Lemma~\ref{lem:reservoir-extension-2}, it suffices to prove that the desired conclusion in Lemma~\ref{keylem:completingbases-new} holds for every subset $R\su B_1\cup\dots\cup B_n$ satisfying \ref{property:spade} with $\gamma=2\sigma=\eps^3/10$.
Now, defining $\eps'=\eps/2$, note that $(\eps')^2\eta\ge \eps^3/4>2\gamma$, and so \ref{property:triangle} holds by Lemma~\ref{lem:reservoir-big-span}.

Finally, as in the statement of Lemma~\ref{keylem:completingbases-new}, let $\mathcal{S}$ be a family of $\lceil (1-\eps)n\rceil$ rainbow independent sets with $\frac{3}{4}n^2\le |E(\mathcal{S})|< \lceil (1-\eps)n\rceil\cdot n$ and
$|E(\mathcal{S})\cap R|\leq \sigma n^2$. Let $s=\lceil (1-\eps)n\rceil\cdot n-|E(\mathcal{S})|\ge 1$, and assume for contradiction that there is no family $\mathcal{S}'$ of $\lceil (1-\eps)n\rceil$ rainbow independent sets such that $|E(\mathcal{S}')|>|E(\mathcal{S})|$ and $|E(\mathcal{S}')\cap R|\leq |E(\mathcal{S})\cap R|+L \cdot \ln(n^2/s)$.

Let $\mathcal{T}$ be a family of $\lceil (1-\eps)n\rceil$ rainbow independent sets with $|E(\mathcal{T})|= |E(\mathcal{S})|$ and $|E(\mathcal{T})\cap R|\leq |E(\mathcal{S})\cap R|+(L/2) \cdot \ln(n^2/s)$ such that there are distinct $T_1,T_2\in \mathcal{T}$ with $\numberofabsorbable_\mathcal{T}(T_1,T_2)\ge \eps^4 n/10^5$, if such a family $\mathcal{T}$ exists. In this case, also define $r_{\mathrm{min}}=2$.
If no such family exists, then define instead $\mathcal{T}=\mathcal{S}$ and $r_{\mathrm{min}}=1$, and choose $T_1\in \mathcal{T}=\mathcal{S}$ such that $n-|T_1|$ is maximal. Note that $\sum_{T\in \mathcal{S}}(n-|T|)=\lceil (1-\eps)n\rceil\cdot n-|E(\mathcal{S})|=s$, so in this case we have $n-|T_1|\ge s/n>0$.

Note that in either case we have $|E(\mathcal{T})|=|E(\mathcal{S})|$ and $|E(\mathcal{T})\cap R|\leq |E(\mathcal{S})\cap R|+(L/2) \cdot \ln(n^2/s)$. Let $U=(B_1\cup \dots \cup B_n)\setminus E(\mathcal{T})$ and $r_{\mathrm{max}}=\lceil (L/8)\cdot \ln(n^2/s)\rceil+2$. Note that $s\le n^2-|E(\mathcal{S})|\le n^2/4$, and hence $\ln(n^2/s)\ge \ln(4)>1$ and $r_{\mathrm{max}}-2\ge L/8 >10^6/\eps^5$.

Now, for any $r_{\mathrm{min}}\leq r< r_{\mathrm{max}}$, let us recursively define  $T_{r+1}\in \mathcal{T}\setminus \{T_1,\dots,T_r\}$ to be the set maximising $\numberofabsorbable_\mathcal{T}(T_1,\dots,T_r,T_{r+1})$. Our goal is to show that $\numberofabsorbable_\mathcal{T}(T_1,\dots,T_r,T_{r+1})$ must grow very quickly with $r$, contradicting the fact that we have $\numberofabsorbable_\mathcal{T}(T_1,\dots,T_r,T_{r+1})\le |T_{r+1}|\le n$ for all $r_{\mathrm{min}}\leq r< r_{\mathrm{max}}$.

More specifically, we claim that for any $r_{\mathrm{min}}\leq r< r_{\mathrm{max}}$ we have
    \begin{equation}\label{eq:Nincrease}
\numberofabsorbable_{\mathcal{T}}(T_1,\dots,T_r,T_{r+1})\geq \Big(1+\frac{\eps}{4}\Big)\cdot \Big(\numberofabsorbable_{\mathcal{T}}(T_1,\dots,T_{r-1},T_{r})+n-|T_{r}|\Big).
    \end{equation}
    (where in the case $r_{\mathrm{min}}=1$ we define $\numberofabsorbable_{\mathcal{T}}(T_1)=0$). Noting that $\numberofabsorbable_{\mathcal{T}}(T_1,\dots,T_{r_{\mathrm{max}}-1},T_{r_{\mathrm{max}}})\le |T_{r_{\mathrm{max}}}|\le n$, this is sufficient to get a contradiction. Indeed, in the case $r_{\mathrm{min}}=2$, \eqref{eq:Nincrease} implies \[\numberofabsorbable_{\mathcal{T}}(T_1,\dots,T_{r_{\mathrm{max}}})\ge \Big(1+\frac{\eps}{4}\Big)^{r_{\mathrm{max}}-2}\cdot \numberofabsorbable_{\mathcal{T}}(T_1,T_2)\ge (e^{\eps/8})^{10^6/\eps^5}\cdot \frac{\eps^4n}{10^5}>e^{10^5/\eps^4}\cdot \frac{\eps^4n}{10^5}>\frac{10^5}{\eps^4}\cdot \frac{\eps^4n}{10^5}=n,\]
    and in the case $r_{\mathrm{min}}=1$, \eqref{eq:Nincrease} implies
    \[\numberofabsorbable_{\mathcal{T}}(T_1,\dots,T_{r_{\mathrm{max}}})\ge \Big(1+\frac{\eps}{4}\Big)^{r_{\mathrm{max}}-1}\cdot (n-|T_1|)\ge (e^{\eps/8})^{(10^6/\eps^5)\cdot \ln(n^2/s)}\cdot \frac{s}{n}>\frac{n^2}{s}\cdot \frac{s}{n}=n.\]
    Thus, it suffices to prove \eqref{eq:Nincrease}.

We show \eqref{eq:Nincrease} by induction on $r$. So, for the rest of this proof, take some $r\in \{r_{\mathrm{min}}, \dots, r_{\mathrm{max}}-1\}$ and assume \eqref{eq:Nincrease} holds for all smaller values of $r$ in $\{r_{\mathrm{min}}, \dots, r_{\mathrm{max}}-1\}$. Then, if $r\ge 2$, we in particular have
\[\numberofabsorbable_{\mathcal{T}}(T_1,\dots,T_r)\ge \Big(1+\frac{\eps}{4}\Big)\cdot \numberofabsorbable_{\mathcal{T}}(T_1,\dots,T_{r-1})\ge \dots\ge \Big(1+\frac{\eps}{4}\Big)^{r-2} \numberofabsorbable_{\mathcal{T}}(T_1,T_2).\]
So in the case $r_{\mathrm{min}}=2$, we conclude $\numberofabsorbable_{\mathcal{T}}(T_1,\dots,T_r)\ge \numberofabsorbable_{\mathcal{T}}(T_1,T_2)\ge \eps^4 n/10^5$, and in the case $r_{\mathrm{min}}=1$ we conclude $\numberofabsorbable_{\mathcal{T}}(T_1,\dots,T_r)\ge (1+\eps/4)^{r-2}\numberofabsorbable_{\mathcal{T}}(T_1,T_2)\ge (1+\eps/4)^{r-1} (n-|T_1|)\ge (1+\eps/4)^{r-1}$ if $r\ge 2$.

In order to show \eqref{eq:Nincrease}, we will show that there are many elements $e'\in E(\mathcal{T})\setminus (T_1\cup \dots\cup T_r)$ which are $(T_1,\dots,T_r)$-absorbable with respect to $\mathcal{T}$. In the special case $r=1$, this is shown by the following relatively easy claim (in contrast, in the case $r\ge 2$, this is only shown in Claim~\ref{clm:absorbing-total-counting} and requires significantly more arguments).

\begin{claim}\label{clm:absorbing-counting-r-1}
    If $r=1$, then there are at least $(1-\eps/2)n\cdot (n-|T_{1}|)$ elements $e'\in E(\mathcal{T})\setminus T_1$ which are $(T_1)$-absorbable with respect to $\mathcal{T}$.
\end{claim}
\claimproofstart
    Recall that $|E(\mathcal{T})\cap R|\leq |E(\mathcal{S})\cap R|+(L/2) \cdot \ln(n^2/s)$ and $|E(\mathcal{S})\cap R|\le \sigma n^2$, so $|E(\mathcal{T})\cap R|\le 2\sigma n^2=\gamma n^2$. Furthermore, note that $|T_1|<n$ (since $r=1$ implies $r_{\mathrm{min}}=1$) and that there cannot be a rainbow independent set $S\su U\cup T_1$ with $|S|>|T_1|$ and $|S\setminus T_1|\le 3$ (since otherwise taking $\mathcal{S}'$ to be the family obtained from $\mathcal{T}$ when replacing $T_1\in \mathcal{T}$ by $S$ would contradict our assumption). Therefore, for any element $e'$ which is $(T_1)$-absorbable with respect to $\mathcal{T}$, we must have $e\in E(\mathcal{T})\setminus T_1$. Furthermore, by Lemma~\ref{lem:many-1-absorbable} there is a set of colours $C\su [n]$ of size $|C|\ge (1-\eps')n=(1-\eps/2)n$, such that, for each colour $c\in C$, every element in $B_c\setminus \spn(T_1)$ is $(T_1)$-absorbable with respect to $\mathcal{T}$. Note that, for every colour $c\in C$, we have $|B_c\cap \spn(T_1)|\le \rk(\spn(T_1))=|T_1|$ and hence $|B_c\setminus \spn(T_1)|=n-|B_c\cap \spn(T_1)|\ge n-|T_1|$. Thus in total there are at least $|C|\cdot (n-|T_{1}|)\ge (1-\eps/2)n\cdot (n-|T_{1}|)$ elements $e'\in E(\mathcal{T})\setminus T_1$ which are $(T_1)$-absorbable with respect to $\mathcal{T}$.
\claimproofend

Our goal is to show that also in the case $r\ge 2$ there are many $(T_1,\dots,T_r)$-absorbable elements.

Let $E^*\su T_r$ be the set of elements $e\in T_r$ which are $(T_1,\dots,T_{r-1})$-absorbable with respect to $\mathcal{T}$, and note that $|E^*|=\numberofabsorbable_{\mathcal{T}}(T_1,\dots,T_{r})$. For each $e\in E^*$, we can find a collection of disjoint rainbow independent sets $T_1^{(e)}, \dots , T_{r-1}^{(e)}\su U\cup T_1\cup\dots \cup T_{r-1}\cup \{e\}$ such that $|T^{(e)}_i\setminus T_i|\le 3$ for $i=1,\dots r-1$, and $|T^{(e)}_1|+\dots+|T^{(e)}_{r-1}|\ge |T_1|+\dots+|T_{r-1}|+1$. Define $\mathcal{T}^{(e)}$ to be the family of $\lceil (1-\eps)n\rceil$ disjoint rainbow independent sets obtained from $\mathcal{T}$ when replacing $T_1,\dots,T_r$ by $T_1^{(e)}, \dots , T_{r-1}^{(e)}, T_r-e$. Note that
\[|E(\mathcal{T}^{(e)})|-|E(\mathcal{T})|= |T_1^{(e)}|+\dots ,+|T_{r-1}^{(e)}|+(|T_r|-1)-(|T_1|+\dots+|T_r|)\ge 0\]
and
\[|E(\mathcal{T}^{(e)})\cap R|\le |E(\mathcal{T})\cap R|+\sum_{i=1}^{r-1}|T^{(e)}_i\setminus T_i|\le |E(\mathcal{T})\cap R|+3r\le |E(\mathcal{T})\cap R|+(L/2)\cdot \ln(n^2/s)-3\]
(as $3r\le 3r_{\mathrm{max}}\le (3/8)L\cdot \ln(n^2/s)+9\le (L/2)\cdot \ln(n^2/s)-3$). Thus, we in particular have $|E(\mathcal{T}^{(e)})|\ge |E(\mathcal{T})|=|E(\mathcal{S})|$ and $|E(\mathcal{T}^{(e)})\cap R|\le |E(\mathcal{S})\cap R|+L\cdot \ln(n^2/s)-3$ for all $e\in E^*$.

For each $e\in E^*$, let $U^{(e)}=(B_1\cup\dots\cup B_n)\setminus E(\mathcal{T}^{(e)})$ be the set of elements uncovered by $\mathcal{T}^{(e)}$, and note that $U^{(e)}\su U\cup T_1\cup\dots\cup T_r$. Recalling that $|E(\mathcal{S})\cap R|\le \sigma n^2$, we can also observe that $|E(\mathcal{T}^{(e)})\cap R|\le 2\sigma n^2=\gamma n^2$. Furthermore $|T_r-e|<n$, and there cannot be a rainbow independent set $S\su U^{(e)}\cup (T_r-e)$ with $|S|>|T_r-e|$ and $|S\setminus (T_r-e)|\le 2$ (since otherwise taking $\mathcal{S}'$ to be the family obtained from $\mathcal{T}^{(e)}$ when replacing $T_r-e\in \mathcal{T}^{(e)}$ by $S$ would yield a contradiction to our assumption). Thus, by Lemma~\ref{lem:many-1-absorbable}, there is a set of colours $C^{(e)}\su [n]$ of size $|C^{(e)}|\ge (1-\eps')n$, such that, for each colour $c\in C^{(e)}$, every element in $B_c\setminus \spn(T_r-e)$ is $(T_r-e)$-absorbable with respect to $\mathcal{T}^{(e)}$.

The following claim shows, roughly speaking, that in order to get many $(T_1,\dots,T_r)$-absorbable elements, it suffices to find many elements which are $(T_r-e)$-absorbable with respect to $\mathcal{T}^{(e)}$ for some $e\in E^*$.

\begin{claim}\label{clm:absorbing-implication}
    Let $e\in E^*$, and let $e'$ be $(T_r-e)$-absorbable with respect to $\mathcal{T}^{(e)}$. Then, unless $e'\in T_1^{(e)}\cup \dots \cup T_{r-1}^{(e)}$, the element $e'$ is $(T_1,\dots,T_r)$-absorbable with respect to $\mathcal{T}$ and we have $e'\in E(\mathcal{T})\setminus (T_1\cup \dots\cup T_r)$.
\end{claim}

\claimproofstart
    By assumption, there exists a rainbow independent set $T_r'\su U^{(e)}\cup (T_r-e)\cup \{e'\}$ with $|T_r'\setminus (T_r-e)|\le 3$ and $|T_r'|\ge |T_r-e|+1=|T_r|$.
    We claim that $e'\not\in U^{(e)}$. Indeed, if $e'\in U^{(e)}$, then the family $\mathcal{S}'$ obtained from $\mathcal{T}^{(e)}$ when replacing $T_r-e\in \mathcal{T}^{(e)}$ by $T_r'$ would be a family of $\lceil (1-\eps)n\rceil$ disjoint rainbow independent sets with $|E(\mathcal{S}')|>|E(\mathcal{T}^{(e)})|\ge  |E(\mathcal{S})|$ and
    \[|E(\mathcal{S}')\cap R|\le |E(\mathcal{T}^{(e)})\cap R|+3\le |E(\mathcal{S})\cap R|+L\cdot \ln(n^2/s)\]
     (noting that $|E(\mathcal{S}')\cap R|-|E(\mathcal{T}^{(e)})\cap R|=|T'_r\cap R|-|(T_r-e)\cap R|\le |T'_r\setminus (T_r-e)|\le 3$). This would be a contradiction to our assumption in the beginning that no such family $\mathcal{S}'$ exists.

     So indeed $e'\not\in U^{(e)}$, which means that $e'\in E(\mathcal{T}^{(e)})=T_1^{(e)}\cup \dots \cup T_{r-1}^{(e)}\cup (T_r-e)\cup E(\mathcal{T}\setminus \{T_1, \dots,T_r\})$. Since $e'$ is $(T_r-e)$-absorbable, we have $e'\not\in T_r-e$ (see Definition~\ref{def:absorbable}). Thus, unless $e'\in T_1^{(e)}\cup \dots \cup T_{r-1}^{(e)}$, we therefore have $e'\in E(\mathcal{T}\setminus \{T_1, \dots,T_r\})=E(\mathcal{T})\setminus (T_1\cup \dots\cup T_r)$.

     It remains to show that $e'$ is $(T_1,\dots,T_r)$-absorbable with respect to $\mathcal{T}$ if $e'\not\in T_1^{(e)}\cup \dots \cup T_{r-1}^{(e)}$. To this end, note that $T_1^{(e)}, \dots ,T_{r-1}^{(e)}, T_r'\su U\cup U^{(e)}\cup T_1\cup\dots\cup T_r\cup \{e'\}= U\cup T_1\cup\dots\cup T_r\cup \{e'\}$ are disjoint rainbow independent sets with $|T_1^{(e)}|+ \dots +|T_{r-1}^{(e)}| +|T_r'|\ge |T_1|+ \dots +|T_{r-1}|+1+|T_{r}|$ and $|T_r'\setminus T_r|\le |T_r'\setminus (T_r-e)|\le 3$ (as well as $|T^{(e)}_i\setminus T_i|\le 3$ for $i=1,\dots r-1$). Thus, by Definition~\ref{def:absorbable}, the element $e'$ is $(T_1,\dots,T_r)$-absorbable with respect to $\mathcal{T}$ if $e'\not\in T_1^{(e)}\cup \dots \cup T_{r-1}^{(e)}$.
\claimproofend

Note that we obtain the conclusion in Claim~\ref{clm:absorbing-implication} only for $(T_r-e)$-absorbable elements $e'$ with $e'\not\in T_1^{(e)}\cup \dots \cup T_{r-1}^{(e)}$. Our next claim gives an upper bound for the number of $(T_r-e)$-absorbable elements $e'$ with $e'\in T_1^{(e)}\cup \dots \cup T_{r-1}^{(e)}$ if $r_{\mathrm{min}}=1$.

\begin{claim}\label{clm:absorbing-backwards-counting}
    Let $e\in E^*$, and assume that $r_{\mathrm{min}}=1$. Then, for every $i=1,\dots,r-1$, there are at most $\eps^4n/10^5$ elements $e'\in T_i^{(e)}$ which are $(T_r-e)$-absorbable with respect to $\mathcal{T}^{(e)}$.
\end{claim}

\claimproofstart
    Suppose for contradiction that for some $i\in \{1,\dots,r-1\}$ there are more than $\eps^4n/10^5$ elements $e'\in T_i^{(e)}$ which are $(T_r-e)$-absorbable with respect to $\mathcal{T}^{(e)}$. This means that $\numberofabsorbable_{\mathcal{T}^{(e)}}(T_r-e,T_i^{(e)})> \eps^4n/10^5$.

    Now, recall that $\mathcal{T}^{(e)}$ is a family of $\lceil (1-\eps)n\rceil$ disjoint rainbow independent sets with $|E(\mathcal{T}^{(e)})|\ge |E(\mathcal{S})|$ and $|E(\mathcal{T}^{(e)})\cap R|\le |E(\mathcal{T})\cap R|+(L/2)\cdot \ln(n^2/s)=|E(\mathcal{S})\cap R|+(L/2)\cdot \ln(n^2/s)$ (where in the last step we used that $\mathcal{T}=\mathcal{S}$ by our assumption $r_{\mathrm{min}}=1$).

    Thus, if $|E(\mathcal{T}^{(e)})|>|E(\mathcal{S})|$, this contradicts our assumption that there is no family $\mathcal{S}'$ of $\lceil (1-\eps)n\rceil$ disjoint rainbow independent sets with $|E(\mathcal{S}')|>|E(\mathcal{S})|$ and $|E(\mathcal{S}')\cap R|\le |E(\mathcal{S})\cap R|+L\cdot \ln(n^2/s)$. On the other hand, if $|E(\mathcal{T}^{(e)})|=|E(\mathcal{S})|$, we obtain a contradiction to $r_{\mathrm{min}}=1$ and our choice of $\mathcal{T}$ (since we should have taken $\mathcal{T}$ to be $\mathcal{T}^{(e)}$, making $r_{\mathrm{min}}=2$).
\claimproofend

Combining the two previous claims, we can now give the following lower bound for the number of $(T_1,\dots,T_r)$-absorbable elements.

\begin{claim}\label{clm:absorbing-total-counting}
    There are at least $(1-3\eps/4)n\cdot (\numberofabsorbable_{\mathcal{T}}(T_1,\dots,T_{r})+n-|T_{r}|)$ elements $e'\in E(\mathcal{T})\setminus (T_1\cup \dots\cup T_r)$ which are $(T_1,\dots,T_r)$-absorbable with respect to $\mathcal{T}$.
\end{claim}

\claimproofstart
    In the case $r=1$, the desired statement follows from Claim~\ref{clm:absorbing-counting-r-1} (recalling that $\numberofabsorbable_{\mathcal{T}}(T_1)=0$). So let us assume that $r\ge 2$, which in particular implies $|E^*|=\numberofabsorbable_{\mathcal{T}}(T_1,\dots,T_{r})\ge 1$.

    Now let $E'$ be the set of elements $e'\in B_1\cup\dots\cup B_n$ for which there is some $e\in E^*$ so that $e'$ is $(T_r-e)$-absorbable with respect to $\mathcal{T}^{(e)}$. Furthermore, let $E''$ be the set of elements $e'\in B_1\cup\dots\cup B_n$ for which there is some $e\in E^*$ so that $e'$ is $(T_r-e)$-absorbable with respect to $\mathcal{T}^{(e)}$ and $e'\in T_1^{(e)}\cup \dots \cup T_{r-1}^{(e)}$. Note that $E''\su E'$.
    By Claim~\ref{clm:absorbing-implication}, every element $e'\in E'\setminus E''$ is $(T_1,\dots,T_r)$-absorbable with respect to $\mathcal{T}$ and satisfies $e'\in E(\mathcal{T})\setminus (T_1\cup \dots\cup T_r)$. Thus it suffices to prove that $|E'\setminus E''|\ge (1-3\eps/4)\cdot (\numberofabsorbable_{\mathcal{T}}(T_1,\dots,T_{r})+n-|T_{r}|)$.

    Let us now show a lower bound for $|E'|$ (afterwards, we use Claim~\ref{clm:absorbing-backwards-counting} to show an upper bound for $|E''|$). For every colour $c\in [n]$, let us apply Lemma~\ref{lem:injecttoBi} to the independent set $T_r$ and the basis $B_c$, to find an injection $\phi_c: T_r\to B_c$ such that, for each $e\in T_r$, the set $T_r-e+\phi_c(e)$ is independent, and for each $b\in B_c\setminus \phi_c(T_r)$, the set $T_r+b$ is independent (and hence also the set $T_r-e+b$ for each $e\in T_r$). So, for any $c\in [n]$ we have $\phi_c(e)\in B_c\setminus \spn(T_r-e)$ and $B_c\setminus \phi_c(T_r)\su B_c\setminus\spn(T_r-e)$ for all $e\in T_r$ (and in particular for all $e\in E^*$).

    Recall that for every $e\in E^*$ there is a set of colours $C^{(e)}\su [n]$ of size $|C^{(e)}|\ge (1-\eps')n$, such that for each colour $c\in C^{(e)}$ every element in $B_c\setminus \spn(T_r-e)$ is $(T_r-e)$-absorbable with respect to $\mathcal{T}^{(e)}$. This means that $B_c\setminus \spn(T_r-e)\su E'$ for every $e\in E^*$ and $c\in C^{(e)}$, and hence $\phi_c(e)\in E'$ and $B_c\setminus \phi_c(T_r)\su E'$. Therefore, since $|C^{(e)}|\ge (1-\eps')n$ for all $e\in E^*$ and each map $\phi_c:T_r\to B_c$ is an injection, we can find $|E^*|\cdot (1-\eps')n$ distinct elements of the form $\phi_c(e)$ in $E'$. In addition, when fixing an arbitrary element $e^*\in E^*$ (recalling that $|E^*|\ge 1$), we have $B_c\setminus \phi_c(T_r)\su E'$ for the at least $(1-\eps')n$ colours $c\in C^{(e^*)}$ (and $|B_c\setminus \phi_c(T_r)|=n-|T_r|$). This gives at least $(1-\eps')n\cdot (n-|T_r|)$ elements in $E'$ which do not lie in $\phi_1(T_r)\cup \dots\cup \phi_n(T_r)$. Thus, we can conclude that
    \[|E'|\ge |E^*|\cdot (1-\eps')n+(1-\eps')n\cdot (n-|T_r|)= (1-\eps')n\cdot (|E^*|+n-|T_{r}|).\]

    Next, we claim that $|E''|\le (\eps n/4)\cdot |E^*|$. If $|E^*|\ge 8r/\eps$, this automatically holds as
    \[E''\su \bigcup_{e\in E^*}(T_1^{(e)}\cup \dots \cup T_{r-1}^{(e)})\su (T_1\cup \dots \cup T_{r-1})\cup \bigcup_{e\in E^*}\bigcup_{i=1}^{r-1}(T_1^{(e)}\setminus T_i),\]
    and therefore
    \[|E''|\le |T_1\cup \dots \cup T_{r-1}|+\sum_{e\in E^*}\sum_{i=1}^{r-1}|T_1^{(e)}\setminus T_i|\le (r-1)\cdot n+|E^*|\cdot (r-1)\cdot 3\le  \frac{\eps n}{8} \cdot |E^*|+\frac{\eps n}{8} \cdot |E^*|=\frac{\eps n}{4} \cdot |E^*|\]
    by our assumption $|E^*|\ge 8r/\eps$ and by $r\le r_{\mathrm{max}}\le L\cdot \ln(n^2)\le  \eps n/24$. So in particular this automatically holds if $r_{\mathrm{min}}=2$ (since then $|E^*|=\numberofabsorbable_{\mathcal{T}}(T_1,\dots,T_r)\ge \eps^4n/10^5\ge (8/\eps)L\ln(n^2)\ge 8r_{\mathrm{max}}/\eps\ge 8r/\eps$). It also automatically holds if $r_{\mathrm{min}}=1$ and $r\ge 10^4/\eps^3$ (since then $|E^*|\ge \numberofabsorbable_{\mathcal{T}}(T_1,\dots,T_r)\ge (1+\eps/4)^{r-1}\ge e^{(\eps/8)(r-1)}\ge (e^{\eps r/20})^2\ge (1+\eps r/20)^2\ge \eps^2r^2/400\ge 8r/\eps$). Finally, if $r_{\mathrm{min}}=1$ and $r\le 10^4/\eps^3$, observe that, by Claim~\ref{clm:absorbing-backwards-counting} for each $e\in E^*$ and each $i=1,\dots,r-1$, there are at most $\eps^4 n/10^5$ elements $e'\in T_i^{(e)}$ that are $(T_r-e)$-absorbable with respect to $\mathcal{T}^{(e)}$. Thus, by the definition of $E''$, we have $|E''|\le |E^*|\cdot (r-1)\cdot \eps^4 n/10^5 \le (\eps n/4)\cdot |E^*|$. Thus, we have shown $|E''|\le (\eps n/4)\cdot |E^*|$ in any case.

    All in all, we can conclude (recalling that $\eps'=\eps/2$ and $|E^*|=\numberofabsorbable_{\mathcal{T}}(T_1,\dots,T_{r})$)
    \begin{align*}
    |E'\setminus E''|=|E'|-|E''|&\ge (1-\eps')n\cdot (|E^*|+n-|T_{r}|)-(\eps n/4)\cdot |E^*|\\
    &\ge (1-3\eps/4)n\cdot (|E^*|+n-|T_{r}|)=(1-3\eps/4)n\cdot \Big(\numberofabsorbable_{\mathcal{T}}(T_1,\dots,T_{r})+n-|T_{r}|\Big).\qedhere
    \end{align*}
\claimproofend

By Claim~\ref{clm:absorbing-total-counting}, at least $(1-3\eps/4)n\cdot (\numberofabsorbable_{\mathcal{T}}(T_1,\dots,T_{r})+n-|T_{r}|)$ elements $e'\in E(\mathcal{T})\setminus (T_1\cup \dots\cup T_r)\allowbreak =\bigcup_{T\in \mathcal{T}\setminus\{T_1,\dots,T_r\}} T$ are $(T_1,\dots,T_r)$-absorbable with respect to $\mathcal{T}$. In particular, there exists some $T\in \mathcal{T}\setminus\{T_1,\dots,T_r\}$ containing at least
\[\frac{(1-3\eps/4)n}{|\mathcal{T}|-r}\cdot \Big(\numberofabsorbable_{\mathcal{T}}(T_1,\dots,T_{r})+n-|T_{r}|\Big)\ge (1+\eps/4)\cdot \Big(\numberofabsorbable_{\mathcal{T}}(T_1,\dots,T_{r})+n-|T_{r}|\Big)\]
$(T_1,\dots,T_r)$-absorbable elements with respect to $\mathcal{T}$ (here, we used that $|\mathcal{T}|-r=\lceil (1-\eps)n\rceil-r\le (1-\eps)n$ and $(1+\eps/4)(1-\eps)\le 1-3\eps/4$). This means that $\numberofabsorbable_{\mathcal{T}}(T_1,\dots,T_{r},T)\ge (1+\eps/4)\cdot (\numberofabsorbable_{\mathcal{T}}(T_1,\dots,T_{r})+n-|T_{r}|)$. Recalling that $\numberofabsorbable_{\mathcal{T}}(T_1,\dots,T_{r},T_{r+1})\ge \numberofabsorbable_{\mathcal{T}}(T_1,\dots,T_{r},T)$ by the choice of $T_{r+1}$, this implies \eqref{eq:Nincrease}, as desired.
\end{proof}


\section{Proofs of simple matroid lemmas}\label{sec:simple-matroid-lemmas}

This section contains the proofs of several simple lemmas concerning matroids, rainbow independent sets and deadlocks, which we omitted in the previous sections. We start by proving Lemma~\ref{lem:rainbow-independence-simple-adding}.

\begin{proof}[Proof of Lemma~\ref{lem:rainbow-independence-simple-adding}]
First, we can find an independent set $S'$ with $S\su S'\su S\cup T$ of size $|S'|\ge |T|$ by starting with the independent set $S$ and adding elements from the independent set $T$ using the augmentation property (adding elements from $T$ one at a time until $|S'|\ge |T|$). Note that then we have $|T\setminus S'|=|T\cup S'|-|S'|=|T\cup S|-|S'|\le |T\cup S|-|T|=|S\setminus T|$. Every colour appears at most twice in $S'$, since each of the sets $S$ and $T$ is rainbow. Furthermore, if a colour appears twice in $S'$, then it must appear once in $S\setminus T\su S'$ (and once in $S'\setminus S$). So there are at most $|S\setminus T|$ colours appearing twice in $S'$, and for each of them we can delete an element of $S'\setminus S$ to remove the colouring conflict. This way, we obtain a rainbow independent set $S^*$ with $S\su S^*\su S'\su S\cup T$ and $|S'\setminus S^*|\le |S\setminus T|$. Note that then we have $|T\setminus S^*|\le |T\setminus S'|+|S'\setminus S^*|\le 2\cdot |S\setminus T|$.
\end{proof}

Next, we will prove Lemmas~\ref{lem:injecttoBi} and \ref{lem:injecttST}. We will deduce both of these lemmas from the following statement due to Brualdi~\cite{brualdi1969comments} (for the reader's convenience we also include a proof here).

\begin{lemma}[{\cite{brualdi1969comments}}]\label{lem:two-bases-hall}
    Let $B$ and $B'$ be bases of some matroid. Then, there is a bijection $\psi: B\to B'$ such that, for each $x\in B$, the set $B'-\psi(x)+x$ is independent.
\end{lemma}
\begin{proof}
    For each $x\in B$, let $S_x\su B'$ be the set of all $x'\in B'$ such that $B'-x'+x$ is independent.
    We claim that $S_x+x$ cannot be an independent set. Indeed, if $S_x+x$ was an independent set for some $x\in B$, then we could find an independent set $T$ with $S_x+x\su T\su (S_x+x)\cup B'=B'+x$ of size $|T|=|B'|$ via the augmentation property by successively adding elements from $B'$. But then the independent set $T$ is of the form $B'-x'+x$ for some $x'\not\in S_x$, which is a contradiction to the definition of $S_x$.

    So indeed, for each $x\in B$, the set $S_x+x$ is not independent. Thus we have $x\in \spn(S_x)$ for every $x\in B$. Then, for any subset $X\su B$, we have $X\su \spn(\bigcup_{x\in X}S_x)$ and hence $|\bigcup_{x\in X}S_x|\ge  \rk(\bigcup_{x\in X}S_x)=\rk(\spn(\bigcup_{x\in X}S_x))\ge \rk(X)=|X|$. Thus, by Hall's Theorem, we can find a bijection $\psi: B\to B'$ with $\psi(x)\in S_x$ for all $x\in B$. By the definition of $S_x$, this means that $B'-\psi(x)+x$ is an independent set for each $x\in B$.
\end{proof}

Using Lemma~\ref{lem:two-bases-hall}, it is not hard to deduce Lemmas~\ref{lem:injecttoBi} and \ref{lem:injecttST}.

\begin{proof}[Proof of Lemma~\ref{lem:injecttoBi}]
    Let $B'$ be a basis of the matroid with $S\su B'$. Apply Lemma~\ref{lem:two-bases-hall} to find a bijection $\psi: B\to B'$ such that $B'-\psi(x)+x$ is an independent set for every $x\in B$. Letting $\phi:B'\to B$ be the inverse map of $\psi$, the set $B'-x+\phi(x)$ is independent for every $x\in B'$. In particular, this means that for every $x\in S$, the set $S-x+\phi(x)$ is independent. Furthermore, for each $b\in B\setminus \phi(S)$, the set $B'-x+b$ is independent for some $x\in B'\setminus S$ (namely, the element $x$ with $\phi(x)=b$). Thus, the set $S+b\su B'-x+b$ is also independent. Finally, note that restricting $\phi$ to $S$ gives an injection from $S$ into $B$.
\end{proof}

\begin{proof}[Proof of Lemma~\ref{lem:injecttST}]
    Let $B$ and $B'$ be bases of the matroid with $S\su B$ and $T\su B'$. By Lemma~\ref{lem:two-bases-hall}, there is a bijection $\psi: B\to B'$ such that $B'-\psi(x)+x$ is an independent set for every $x\in B$. If there is some $x\in S\su B$ with $\psi(x)\not\in T$, then this means that $T+x\su B'-\psi(x)+x$ is an independent set, and thus \ref{injopt:1} holds. So let us assume that $\psi(x)\in T$ for all $x\in S$. Then, by restricting $\psi$ to $S$, we get an injection $\phi:S\to T$ such that $B'-\phi(x)+x$, and hence also $T-\phi(x)+x$, is an independent set for every $x\in S$.

    Furthermore, we may assume that $x\in \spn(T)$ for every $x\in S$. Indeed, if there is some $x\in S$ with $x\not\in \spn(T)$, then $T+x$ is an independent set and \ref{injopt:1} holds again.
        So, for any $x\in S$ we may assume that $x\in \spn(T)$, which implies $T\cup \{x\}\su \spn(T)$ and hence $\spn(T\cup \{x\})=\spn(T)$ (since $\spn(T)\su \spn(T\cup \{x\})\su \spn(\spn(T))=\spn(T)$ by the basic matroid rules in Section~\ref{sec:prelim}). Thus,
    \[\spn(T-\phi(x)+x)\su \spn(T\cup \{x\})=\spn(T),\]
    and together with $\rk(T-\phi(x)+x)= |T-\phi(x)+x|=|T|=\rk(T)$, this implies $\spn(T-\phi(x)+x)=\spn(T)$. Thus, \ref{injopt:2} holds.
\end{proof}

It remains to prove Lemmas~\ref{lem:union-overcrowded} and \ref{lem:deadlock-lemma}. To do so, we first need some preparation. First, the following well-known lemma states that the rank function in a matroid is submodular (see for example \cite[Lemma 1.3.1]{oxley2011book}).  Again, we include a short proof for completeness.

\begin{lemma}\label{lem:rank-submodular}
    For any subsets $A$ and $B$ of a matroid, we have $\rk(A\cap B)+\rk(A\cup B)\le \rk(A)+\rk(B)$.
\end{lemma}

\begin{proof}
Let $S\su A\cap B$ be a maximal independent subset of $A\cap B$, then $|S|=\rk(A\cap B)$. Let us now extend $S$ to a maximal independent subset $S_A$ of $A$, then $S\su S_A\su A$ and $|S_A|=\rk(A)$. Similarly, let us extend $S$ to a maximal independent subset $S_B$ of $B$, then $S\su S_B\su B$ and $|S_B|=\rk(B)$. Now, $S_A\cup S_B=S\cup (S_A\setminus S)\cup (S_B\setminus S)$ is a set of size
\[|S|+|S_A\setminus S|+|S_B\setminus S|=\rk(A\cap B)+(\rk(A)-\rk(A\cap B))+(\rk(B)-\rk(A\cap B))=\rk(A)+\rk(B)-\rk(A\cup B).\]
Thus, as $A\cup B\su \spn(S_A)\cup \spn(S_B)\su \spn(S_A\cup S_B)$, we have
\[\rk(A\cup B)\le \rk(\spn(S_A\cup S_B))=\rk(S_A\cup S_B)\le |S_A\cup S_B|=\rk(A)+\rk(B)-\rk(A\cup B),\]
and rearranging yields the desired inequality.
\end{proof}

Next, we prove Lemma~\ref{lem:union-overcrowded}.

\begin{proof}[Proof of Lemma~\ref{lem:union-overcrowded}]
Let $S'\su S_1\cup S_2$, then we have
\begin{align*}
    |(S_1\cup S_2)\setminus S'|&=|S_1\setminus (S'\cap S_1)|+|S_2\setminus ((S'\cup S_1)\cap S_2)|\\
    &\ge k\cdot (\rk(S_1)-\rk(S'\cap S_1))+k\cdot (\rk(S_2)-\rk((S'\cup S_1)\cap S_2))\\
    &\ge k\cdot (\rk(S_1)-\rk(S'\cap S_1))+k\cdot (\rk((S'\cup S_1)\cup S_2)-\rk(S'\cup S_1))\\
    &= k\cdot (\rk(S_1\cup S_2) +\rk(S_1)-\rk(S'\cap S_1)-\rk(S'\cup S_1))\\
    &\ge k\cdot (\rk(S_1\cup S_2) -\rk(S')),
\end{align*}
where in the third step we used Lemma~\ref{lem:rank-submodular} (the submodularity of the rank function) for $A=S_2$ and $B=S'\cup S_1$ and in the last step we used it for $A=S_1$ and $B=S'$. Thus, $S_1\cup S_2$ is $k$-overcrowded.
\end{proof}

Our last remaining task is to prove Lemma~\ref{lem:deadlock-lemma}.

\begin{proof}[Proof of Lemma~\ref{lem:deadlock-lemma}]
Let $D=D_k(U\cup U')$, so that $D\su U\cup U'$ is $k$-overcrowded and we wish to show that $D_k(U)=D$. 
First, we show that $D\setminus U'$ is $k'$-overcrowded. For this, let $S'\su D\setminus U'$. If $\rk{(S')}=\rk{(D\setminus U')}$, then $|(D\setminus U')\setminus S'|\geq k'\cdot (\rk{(D\setminus U')}-\rk{(S')})$ trivially holds. So assume $\rk{(S')}<\rk{(D\setminus U')}$, then $\rk{(S')}\le \rk{(D\setminus U')}-1\le \rk{(D)}-1$. As $D$ is $k$-overcrowded and $S'\su D$, we have (using $|U'|\le k-k'$)
\[
|(D\setminus U')\setminus S'|\geq |D\setminus S'|-|U'|\geq k(\rk{(D)}-\rk{(S')})-(k-k')\geq k'(\rk{(D)}-\rk{(S')})\ge k'(\rk{(D\setminus U')}-\rk{(S')}).
\]
Therefore, $D\setminus U'$ is indeed $k'$-overcrowded.

Since $D\setminus U'\su (U\cup U')\setminus U'=U$, this implies $D\setminus U'\su D_{k'}(U)$ and hence $D\setminus (D_{k'}(U)\cap D)\su U'$. Thus, we obtain (again using that $D$ is $k$-overcrowded)
\[k>k-k'\ge |U'|\ge |D\setminus (D_{k'}(U)\cap D)|\ge k\cdot \Big(\rk(D)-\rk(D_{k'}(U)\cap D)\Big).\]
This implies $\rk(D)-\rk(D_{k'}(U)\cap D)<1$, meaning that $\rk(D)=\rk(D_{k'}(U)\cap D)$ and therefore $\spn(D)=\spn(D_{k'}(U)\cap D)$. In particular, we obtain $D\su \spn(D_{k'}(U)\cap D)\su \spn(D_{k'}(U))$. Recalling that $U'\cap \spn{(D_{k'}(U))}=\emptyset$, this implies $D\cap U'=\emptyset$ and hence $D\su U$. So we can conclude that $D\su D_k(U)$, as $D$ is $k$-overcrowded. Since we also have $D_k(U)\su D_k(U\cup U')=D$, we obtain $D_k(U\cup U')=D=D_k(U)$.
\end{proof}

\end{document}